%% file: goafem_semilinear.tex
\documentclass[12pt,reqno,a4paper]{amsart}
\usepackage{fullpage}
\usepackage[utf8]{inputenc}
\usepackage[T1]{fontenc}
\usepackage{hyperref} 
\usepackage{enumitem}
\usepackage{amsmath,amssymb,amsthm,color}
\usepackage{ifthen}
\usepackage{nameref}
\usepackage{mathtools}

\usepackage{tikz}

\usepackage[%
	backend=biber,%
	citestyle=ieee-alphabetic,%
	bibstyle=alphabetic,%
	maxnames=20,%
	giveninits=true,%
	safeinputenc,
			    doi=false,
		    isbn=false,
		    url=false,
		    eprint=false,
		    sortcites=false,
	]{biblatex}
	
\DeclareDelimFormat[bib]{nametitledelim}{\addcolon\space}

\DeclareFieldFormat
[article,inbook,incollection,inproceedings,patent,thesis,unpublished]
{title}{\mkbibemph{#1\isdot}}

\DeclareFieldFormat
[article,inbook,incollection,inproceedings,patent,thesis,unpublished]
{journaltitle}{#1}

\renewbibmacro*{volume+number+eid}{%
	\printfield{volume}%
	\setunit*{\addcolon}%
	\printfield{number}%
	\setunit{\addcomma\space}%
	\printfield{eid}}

\DeclareFieldFormat{pages}{#1}

\AtEveryBibitem{%
	\clearlist{language}%
	\clearfield{urlyear}%
	\clearfield{note}%
	\ifentrytype{article}{%
		\renewbibmacro*{in:}{}%
		\clearfield{issn}%
		\clearfield{issue}}{}%
	\ifentrytype{book}{%
		\clearfield{volume}%
		\clearfield{isbn}%
		\clearfield{series}%
		\clearfield{pages}}{}%
	\ifentrytype{inproceedings}{%
		\clearfield{series}%
		\clearname{editor}%
		\clearfield{pages}}{}%
}

\renewbibmacro*{publisher+location+date}{%
  \ifentrytype{book}
    {\printlist{publisher}%
     \iflistundef{location}
       {\setunit*{\addcomma\space}}
       {\setunit*{\addcomma\space}}%
     \printlist{location}}
    {\printlist{location}%
     \iflistundef{publisher}
      {\setunit*{\addcomma\space}}
      {\setunit*{\addcolon\space}}%
     \printlist{publisher}}
  \setunit*{\addcomma\space}%
  \usebibmacro{date}%
  \newunit}
  
\usepackage{graphicx}
\usepackage{subcaption}
\usepackage{pgfplots}
\pgfplotsset{compat=1.9}
\usepackage{pgfplotstable}

\usepackage{bbm}
\usepackage{nicefrac}
\usepackage{fancyhdr}
\usepackage[mathlines]{lineno}

\usepackage{crossreftools}
\usepackage{cancel}
\usepackage{stmaryrd}

\usepackage{epstopdf}
\def\exact{{}}
\def\g{\boldsymbol{g}}
\def\prod#1#2{\langle#1\,,\,#2\rangle}
\def\sprod#1#2{\langle\hspace{-1.2mm} \langle #1\,,\,#2 \rangle\hspace{-1.2mm}\rangle}
\def\B{\boldsymbol{B}}

\def\hook{\hookrightarrow}

\def\CGNS{C_{\mathrm{GNS}}}
\def\Cbnd{C_{\mathrm{bnd}}}
\def\CLip{C_{\mathrm{Lip}}}

\def\Cdual{C_{\mathrm{dual}}}

\def\Cest{C_{\mathrm{est}}}

\def\Cdiff{C_{\mathrm{diff}}}

\def\weakly{\, {\rightharpoonup} \, }
\def\KK{\mathcal{K}}

\newcommand\bcdot{\ensuremath{%
   \mathchoice%
    {\mskip\thinmuskip\lower0.2ex\hbox{\scalebox{1.5}{$\cdot$}}\mskip\thinmuskip}}%
    {\mskip\thinmuskip\lower0.2ex\hbox{\scalebox{1.5}{$\cdot$}}\mskip\thinmuskip}%
    {\lower0.3ex\hbox{\scalebox{1.2}{$\cdot$}}}%
    {\lower0.3ex\hbox{\scalebox{1.2}{$\cdot$}}}%
} 
\newcommand{\pushright}[1]{\ifmeasuring@#1\else\omit\hfill$\displaystyle#1$\fi\ignorespaces}

\title[GOAFEM for semilinear elliptic PDEs]{Goal-oriented adaptive finite element method\\for semilinear elliptic PDEs}
\author{Roland Becker}
\address{Université de Pau et des Pays de l’Adour, IPRA-LMAP, Avenue de l’Université BP 1155, 64013 PAU Cedex, France}
\email{roland.becker@univ-pau.fr}

\address{TU Wien, Institute of Analysis and Scientific Computing, Wiedner Hauptstr. 8-10/E101/4, 1040 Vienna, Austria}
\author{Maximilian Brunner}
\email{maximilian.brunner@asc.tuwien.ac.at \quad \rm (corresponding author)}

\author{Michael Innerberger}
\email{michael.innerberger@asc.tuwien.ac.at}%

\author{Jens Markus Melenk}
\email{jens.melenk@asc.tuwien.ac.at}%

\author{Dirk Praetorius}
\email{dirk.praetorius@asc.tuwien.ac.at}

\thanks{{\bf Acknowledgment.} The authors thankfully acknowledge support by the Austrian Science Fund (FWF) through the doctoral school \emph{Dissipation and dispersion in nonlinear PDEs} (grant W1245), the SFB \emph{Taming complexity in partial differential systems} (grant SFB F65), and the stand-alone project \emph{Computational nonlinear PDEs} (grant P33216). Additionally, Maximilian Brunner and Michael Innerberger are supported by the \emph{Vienna School of Mathematics}.}

\makeatother

\def\set#1#2{\big\{#1 \,\mid\, #2 \big\}}

\def\eps{\varepsilon}

\advance\footskip0.4cm
\advance\textheight-0.4cm
\calclayout
\makeatletter
\def\@seccntformat#1{\hspace*{4mm}%
  \protect\text{}{\protect\@secnumfont
    \ifnum\pdfstrcmp{subsection}{#1}=0 \bfseries\fi
    \csname the#1\endcsname
    \protect\@secnumpunct
  }%
}
\makeatother

\def\tmp#1{\normalsize#1\small}

\renewcommand{\section}[2][]{%
  \vskip4mm
  \refstepcounter{section}%
  \begin{center}\bf%
    \uppercase{\thesection.~\tmp#2\normalsize}%
   \end{center}%
   \vskip2mm
}

\let\div\relax
\DeclareMathOperator{\div}{div}

\def\coarse{H}
\def\fine{h}

\def\N{\mathbb{N}}
\def\T{\mathbb{T}}
\def\MM{\mathcal{M}}
\def\TT{\mathcal{T}}
\def\XX{\mathcal{X}}

\def\Crel{C_{\rm rel}}

\def\Cstab{C_{\rm stab}}

\def\qred{q_{\rm red}}
\def\Clin{C_{\rm lin}}
\def\qlin{q_{\rm lin}}

\def\Ccea{C_{\text{\rm C\'ea}}}

\def\Cnvb{C_{\rm mesh}}

\def\Cdrel{C_{\rm drel}}

\def\refine{{\tt refine}}

\def\enorm#1{|\!|\!| #1 |\!|\!|}
\def\reff#1#2{\stackrel{\eqref{#1}}{#2}}

\def\Cmark{C_{\rm mark}}
\def\Copt{C_{\rm opt}}
\def\norm#1#2{\|#1\|_{#2}}
\def\OO{\mathcal{O}}
\def\R{\mathbb{R}}
\def\d#1{\,{\rm d}#1}

\let\div\relax
\DeclareMathOperator{\div}{div}
\let\I\relax
\DeclareMathOperator{\I}{\operatorname{I}}

\def\AA{\mathcal{A}}

\def\UU{\mathcal{U}}

\newcounter{statement}
\newenvironment{statement}[2][!]{%
\vskip3mm
\hrule
\hrule
\hrule
\vskip1mm
\noindent%
\refstepcounter{statement}%
\bf#2~\thestatement%
\ifthenelse{\equal{#1}{!}}{.\ }{~(#1).\ }%
\it%
}{%
\vskip1mm
\hrule
\hrule
\hrule
\vskip2mm
}

\newcounter{algorithm}
\renewcommand{\thealgorithm}{\Alph{algorithm}}

\newenvironment{algorithm}[1][!]{\begin{statement}[#1]{Algorithm}}{\end{statement}}
\newenvironment{theorem}[1][!]{\begin{statement}[#1]{Theorem}}{\end{statement}}
\newenvironment{lemma}[1][!]{\begin{statement}[#1]{Lemma}}{\end{statement}}

\newenvironment{proposition}[1][!]{\begin{statement}[#1]{Proposition}}{\end{statement}}

\newenvironment{remark}[1][!]{\begin{statement}[#1]{Remark}}{\end{statement}}

\newenvironment{example}[1][!]{\begin{statement}[#1]{Example}}{\end{statement}}

\begin{document}

\begin{abstract}
We formulate and analyze a goal-oriented adaptive finite element method (GOAFEM) for a semilinear elliptic PDE and a linear goal functional. The strategy involves the finite element solution of a linearized dual problem, where the linearization is part of the adaptive strategy.
Linear convergence and optimal algebraic convergence rates are shown.
\end{abstract}

\maketitle
\thispagestyle{fancy}

\def\A{\boldsymbol{A}}
\def\b{\boldsymbol{b}}
\def\f{\boldsymbol{f}}
\def\n{\boldsymbol{n}}
\def\SS{\mathcal{S}}
\def\KK{\mathcal{K}}


\input{01_introduction.tex}
\input{02_modelproblem.tex}
\input{03_algorithm.tex}

\input{04_proofs.tex}

\input{05_numerics.tex}


{
	\renewcommand{\section}[3][]{\vskip4mm\begin{center}\bf\normalsize R\small EFERENCES\normalsize\end{center}\vskip2mm}
	\printbibliography  
}


\end{document}

%% file: 01_introduction.tex

\section{Introduction}
\subsection{Goal-oriented adaptive FEM and contributions}
While standard adaptivity aims to approximate the exact solution $u^\exact \in H^1_0(\Omega)$ of a suitable PDE at optimal rate in the energy norm (see, e.g., \cite{doerfler1996,mns2000,bdd2004,stevenson2007,ckns2008} for some seminal contributions and \cite{ffp2014} for the present model problem), goal-oriented adaptivity aims to approximate, at optimal rate, only the functional value $G(u^\exact) \in \R$ (also called \emph{quantity of interest} in the literature). Usually, goal-oriented adaptivity is more important in practice than standard adaptivity and, therefore, has attracted much interest also in the mathematical literature; see, e.g.,~\cite{MR1960405,MR1352472,MR2009374,MR2009692} for some prominent works and~\cite{MR3980263, MR4250624,bgip2020+, MR4271572} for some recent contributions. Unlike standard adaptivity, there are only few works that aim for a mathematical understanding of optimal rates for goal-oriented adaptivity; see~\cite{ms2009,bet2011,ffghp2016,fpz2016} for linear problems with linear goal functional and~\cite{bip2020} for a linear problem, but nonlinear goal functional. The works~\cite{hpz2015, xhym2021} consider semilinear PDEs and linear goal functionals, but only prove convergence, while optimal convergence rates remain open (and can hardly be proved for the proposed algorithms). The present work proves, for the first time, optimal convergence rates for goal-oriented adaptivity for a \emph{nonlinear} problem. To this end, we see, in particular, that the marking strategy used in~\cite{hpz2015, xhym2021} must be modified along the ideas of~\cite{bip2020}.

\subsection{Model problem}
For $d \in \{1, 2, 3\}$, let $\Omega \subset \R^d$ be a bounded Lipschitz domain. 
Given $f, g \in L^2(\Omega)$ and $\f, \g \in [L^2(\Omega)]^d$, we aim to approximate the linear goal quantity 
\begin{align}\label{eq:goal}
 G(u^\exact) := \int_\Omega gu^\exact \d{x} + \int_\Omega \g \cdot \nabla u^\exact \d{x},
\end{align}
where $u^\exact \in H^1_0(\Omega)$ is the weak solution of the semilinear elliptic PDE
\begin{align}\label{eq:strongform:primal}
 -\div(\A \nabla u^\exact) + b(u^\exact) = f - \div \f
 \text{ \ in } \Omega
 \quad \text{subject to} \quad 
 u^\exact = 0 \text{ \ on } \Gamma := \partial \Omega.
\end{align}
While the precise assumptions on the coefficients $\A \colon \Omega \to \R_{\rm sym}^{d \times d}$ and $b \colon \Omega \times \R \to \R$ are given in Section~\ref{subsection:assump:diff}--\ref{subsection:assump:nonlinear},
we note that, here and below, we abbreviate $\A \nabla u^\exact \equiv \A(\cdot) \nabla u^\exact(\cdot) \colon \Omega \to \R^d$ and $b(u^\exact) \equiv b(\cdot, u^\exact(\cdot)) \colon \Omega \to \R$.

The weak formulation of the so-called \emph{primal problem}~\eqref{eq:strongform:primal} reads as follows: Find $u^\exact \in H^1_0(\Omega)$ such that
\begin{align}\label{eq:weakform:primal}
 \sprod{u^\exact}{v} + \prod{b(u^\exact)}{v} = F(v) 
 := \prod{f}{v} + \prod{\f}{\nabla v}
 \quad \text{for all } v \in H^1_0(\Omega),
\end{align}
where $\prod{v}{w} := \int_\Omega vw \d{x}$ denotes the $L^2(\Omega)$-scalar product and $\sprod{v}{w} := \prod{\A\nabla v}{\nabla w}$ is the $\A$-induced energy scalar product on $H^1_0(\Omega)$. 
We stress that existence and uniqueness of the solution $u^\exact \in H^1_0(\Omega)$ of~\eqref{eq:weakform:primal} follow from the Browder--Minty theorem on monotone operators (see Section~\ref{subsection:browder-minty} for details).

Based on conforming triangulations $\TT_\coarse$ of $\Omega$ and fixed polynomial degree $m \in \N$, let $\XX_\coarse := \set{v_\coarse \in H^1_0(\Omega)}{\forall \,T \in \TT_\coarse\colon \, v_\coarse|_T  \text{ is a polynomial of degree} \le m}$. Then, the FEM discretization of the primal problem~\eqref{eq:weakform:primal} reads: Find $u_\coarse^\exact \in \XX_\coarse$ such that
\begin{align}\label{eq:weakform:primal:discrete}
 \sprod{u_\coarse^\exact}{ v_\coarse} + \prod{b(u_\coarse^\exact)}{v_\coarse} = F(v_\coarse) 
 \quad \text{for all } v_\coarse \in \XX_\coarse.
\end{align}
Then, the finite element method approximates the sought goal quantity $G(u^\exact)$ by means of the computable quantity $G(u_\coarse^\exact)$.

\subsection{Error control and GOAFEM algorithm}

The optimal error control of the goal error $G(u^\exact) - G(u^\exact_\coarse)$ involves the so-called \emph{(practical) dual problem}: Find $z^\exact[u_\coarse^\exact] \in H^1_0(\Omega)$ such that
\begin{align}\label{eq:weakform:dual:exact}
 \sprod{z^\exact[u_\coarse^\exact]}{ v} + \prod{b'(u_\coarse^\exact) z^\exact[u_\coarse^\exact]}{v} = G(v)
 \quad \text{for all } v \in H^1_0(\Omega),
\end{align}
where $b'(x,t) := \partial_t b(x,t)$. Existence and uniqueness of $z^\exact[u_\coarse^\exact]$ follow from the Lax--Milgram lemma (see Section \ref{subsection:wellposed:dual}). With the same FEM spaces as for the primal problem, the FEM discretization of the dual problem~\eqref{eq:weakform:dual:exact} reads: Find $z_\coarse^\exact[u_\coarse^\exact] \in \XX_\coarse$ such that
\begin{align}\label{eq:weakform:dual:discrete:exact}
 \sprod{z_\coarse^\exact[u_\coarse^\exact]}{ v_\coarse} + \prod{b'(u_\coarse^\exact) z_\coarse^\exact[u_\coarse^\exact]}{v_\coarse} = G(v_\coarse) 
 \quad \text{for all } v_\coarse \in \XX_\coarse.
\end{align}
The notation $z^\exact[u_\coarse^\exact]$ emphasizes that the dual solution depends on the (exact) discrete primal solution $u_\coarse^\exact$ (instead of the practically unavailable exact primal solution $u^\exact$); the same holds for the discrete dual solution $z^\exact_\coarse[u^\exact_\coarse]$.

For this setting, we derive below (see Theorem~\ref{theorem:errorestimate}) the goal error estimate 
\begin{align}\label{eq:aux:error-estimate}
 | G(u^\exact) - G(u_\coarse^\exact) |
 \lesssim \norm{u^\exact - u_\coarse^\exact}{H^1(\Omega)} \norm{z^\exact[u_\coarse^\exact] - z_\coarse^\exact[u_\coarse^\exact]}{H^1(\Omega)} + \norm{u^\exact - u_\coarse^\exact}{H^1(\Omega)}^2
\end{align}
where $\lesssim$ denotes $\le$ up to some generic multiplicative constant $C > 0$. The arising error terms are controlled by computable \emph{a posteriori} error estimates
\begin{align*}
\norm{u^\exact-u^\exact_\coarse}{H^1(\Omega)} \lesssim \eta_\coarse(u^\exact_\coarse) \quad \text{ and } \quad \norm{z^\exact[u^\exact_\coarse]-z^\exact_\coarse[u^\exact_\coarse]}{H^1(\Omega)} \lesssim \zeta_\coarse(z^\exact_\coarse[u^\exact_\coarse]). 
\end{align*}
Hence,~\eqref{eq:aux:error-estimate} gives rise to the computable error bound
\begin{align}\label{eq:prodstructure}
|G(u^\exact) - G(u^\exact_\coarse)| \lesssim \eta_\coarse(u^\exact_\coarse) \,[\eta_\coarse(u^\exact_\coarse)^2 + \zeta_\coarse(z^\exact_\coarse[u^\exact_\coarse])^2]^{1/2}
\end{align}
that, following~\cite{ms2009, fpz2016, bip2020}, is used to steer an adaptive loop of the type
\begin{align}\label{adaptiveloop}
\begin{tikzpicture}[baseline={([yshift=-.5ex]current bounding box.center)}]
\node[rectangle,draw, rounded corners] (r) at (0,-0.5) {Solve};
\node[rectangle,draw, rounded corners] (s) at (2.5,-0.5) {Estimate};
\node[rectangle,draw, rounded corners] (t) at (5,-0.5) {Mark};
\node[rectangle,draw, rounded corners] (u) at (7.5,-0.5) {Refine};
\draw[->, thick] (r) -- (s);
\draw[->, thick] (s) -- (t);
\draw[->, thick] (t) -- (u);
\draw[->, thick] (u) -- (7.5, 0) -- (0, 0) -- (r);
\end{tikzpicture}
\end{align}
\subsection{Contributions}
Let $(\TT_\ell)_{\ell \in \N_0}$ be the sequence of meshes generated by the adaptive loop~\eqref{adaptiveloop} with corresponding error estimators $\eta_\ell := \eta_\ell(u^\exact_\ell)$ and $\zeta_\ell := \zeta_\ell(z^\exact_\ell[u^\exact_\ell])$. We prove that the proposed adaptive strategy leads to linear convergence
\begin{align}\label{eq:linearconvergence}
\eta_{\ell + n}\, [ \eta_{\ell + n}^2 + \zeta_{\ell + n}^2]^{1/2} \le \Clin \,\qlin^n \,\eta_{\ell}\,[ \eta_{\ell}^2 + \zeta_{\ell}^2]^{1/2} \quad \text{ for all } \ell, n \in \N_0,
\end{align}
where $\Clin > 0$ and $0 < \qlin < 1$ are generic constants. Moreover, we prove that this estimator product leads to convergence
\begin{align}\label{eq:rates}
\eta_\ell \,[ \eta_{\ell}^2 + \zeta_{\ell}^2]^{1/2} = \mathcal{O}((\# \TT_\ell)^\alpha),
\end{align}
where the rate $\alpha = \min\{2s, s + t \}$ is optimal in the sense that $s > 0$ is any possible rate for $\eta_\ell$ and $t > 0$ is any possible rate for $\zeta_\ell$ (in the sense of the usual approximation classes~\cite{axioms}). We stress that this is the first optimality result on GOAFEM for a nonlinear model problem. While $\alpha = s + t$ for linear model problems~\cite{ms2009, fpz2016}, the slightly worse rate $\alpha = \min\{ 2s, s+t\}$ stems from the fact that the adaptive algorithm must also control the linearization of the dual problem. Besides the goal error estimate~\eqref{eq:aux:error-estimate}, technical key results also include Pythagoras-type quasi-orthogonalities for the semilinear model problem~\eqref{eq:strongform:primal}. Finally, we note that our analysis allows to modify the marking strategies of~\cite{hpz2015, xhym2021} to ensure linear convergence of $\eta_\ell^2 + \zeta_\ell^2 = \mathcal{O}((\# \TT_\ell)^{-\alpha})$ with rate $\alpha = \min\{2s, 2t\}$. 
\subsection{Outline}
This work is organized as follows: In Section~\ref{section:modelproblem}, the analytical 	preliminaries for the semilinear setting and its linearizations are presented. This includes the precise assumptions on the problem and the right-hand sides as well as well-posedness of the arising continuous and discrete problems. In Section~\ref{subsection:goalerrorestimate}, the key estimate~\eqref{eq:aux:error-estimate} is proved; cf.\ Theorem~\ref{theorem:errorestimate}.
In Section~\ref{section:mainresults}, we formulate the GOAFEM algorithm (cf.\ Algorithm~\ref{algorithm:first}), which employs a marking strategy that respects the product structure found in~\eqref{eq:prodstructure}. We proceed with stating the main results. First, Theorem~\ref{theorem:main} shows linear convergence~\eqref{eq:linearconvergence} of the proposed algorithm. Second, Theorem~\ref{theorem:main2} shows optimal convergence rates~\eqref{eq:rates}. Section~\ref{section:proofs} is devoted to the proofs of the aforementioned results, which contain the axioms of adaptivity~\cite{axioms} for the semilinear setting (Section~\ref{subsection:proofaxioms}), a stability result for the linearized dual problem (Section~\ref{subsection:stability}),which turns out to be important, and the necessary quasi-orthogonalities (Section~\ref{subsection:quasiorthogonalities}). Numerical experiments underline our theoretical findings in Section~\ref{section:numerical}.

\subsection{General notation}
For $1 \le p \le \infty$, let $1 \le p' \le \infty$ be the conjugate H\"{o}lder index which ensures that $\norm{\phi \,\psi}{L^1(\Omega)} \le \norm{\phi}{L^p(\Omega)}\norm{\psi}{L^{p'}(\Omega)}$ for $\phi \in L^p(\Omega)$ and $\psi \in L^{p'}(\Omega)$, i.e., $1/p + 1/p' = 1$ with the convention that $p' = 1$ for $p=\infty$ and vice versa. Moreover, for $1 \le p < d$, let $1 \le p^\ast := dp/(d-p) < \infty$ denote the critical Sobolev exponent of $p$.
We recall the Gagliardo--Nirenberg--Sobolev inequality
\begin{align}\label{eq:intro:gns}
\norm{v}{L^r(\Omega)} \le \CGNS \, \norm{\nabla v}{L^p(\Omega)} \quad \text{ for all } v \in W^{1,p}_0(\Omega),
\end{align}
with a constant $\CGNS = \CGNS(|\Omega|, d, p, r)$.
If $d \in \{1,2\}$,~\eqref{eq:intro:gns} holds for any $1 \le r < \infty$.  
If $d =3$,~\eqref{eq:intro:gns} holds for all $1 \le r \le p^\ast$, where $r = p^\ast$ is the largest possible exponent such that the embedding $W^{1,p}(\Omega) \hook L^r(\Omega)$ is continuous. 

Moreover, we use $| \, \cdot \, |$ to denote the absolute value $|\lambda|$ of a scalar $\lambda \in \R$, the Euclidean norm $|x|$ of a vector $x \in \R^d$, and the Lebesgue measure $|\omega|$ of a set $\omega \subseteq \overline{\Omega}$, depending on the respective context. Furthermore, $\# \UU$ denotes the cardinality of a finite set $\UU$.

%% file: 02_modelproblem.tex

\section{Model problem}\label{section:modelproblem}

\subsection{Assumptions on diffusion coefficient}\label{subsection:assump:diff}

The diffusion coefficient $\A \colon \Omega \to \R_{\rm sym}^{d \times d}$ satisfies the following standard assumptions:

\begin{itemize}
\renewcommand{\labelitemi}{$\bullet$}
\item[{\crtcrossreflabel{\rm{(ELL)}}[assump:ell]}] $\A \in L^\infty( \Omega; \R_{\rm sym}^{d\times d})$, where $\A(x) \in \R_{\rm sym}^{d\times d}$ is a symmetric and uniformly positive definite matrix, i.e., the minimal and maximal eigenvalues satisfy
\begin{align*}
 0 < \mu_0 := \inf_{x \in \Omega} \lambda_{\rm min} (\A(x))
 \le \sup_{x \in \Omega} \lambda_{\rm max} (\A(x)) =: \mu_1 < \infty.
\end{align*}
\end{itemize}
In particular, the $\A$-induced energy scalar product $\sprod{v}{w} = \prod{\A\nabla v}{\nabla w}$ induces an equivalent norm $\enorm{v} := \sprod{v}{v}^{1/2}$ on $H^1_0(\Omega)$. 

To guarantee later that the residual \textsl{a~posteriori} error estimators are well-defined, we additionally require that $A|_T \in W^{1, \infty}(T)$ for all $T \in \TT_0$, where $\TT_0$ is the initial triangulation of the adaptive algorithm.
\subsection{Assumptions on the nonlinear reaction coefficient}\label{subsection:assump:nonlinear}

The nonlinearity $b \colon \Omega \times \R \to \R$ satisfies the following assumptions, which follow~\cite[(A1)--(A3)]{bhsz2011}:
\begin{itemize}
\renewcommand{\labelitemi}{$\bullet$}
\item[{\crtcrossreflabel{\rm{(CAR)}}[assump:car]}] $b\colon \Omega \times \R \to \R$ is a \textit{Carathéodory} function, i.e., for all $n \in \N_0$, the $n$-th derivative $\partial_{\xi}^n b$ of $b$ with respect to the second argument $\xi$ satisfies that
\begin{itemize}
\item[$\bullet$] for any $\xi \in \R$, the function $x \mapsto \partial_{\xi}^n \, b(x,\xi)$ is measureable on $\Omega$,
\item[$\bullet$] for any $x \in \Omega$, the function $\xi \mapsto \partial_{\xi}^n \, b(x,\xi)$ is smooth.
\end{itemize}
\item[{\crtcrossreflabel{\rm{(MON)}}[assump:mon]}] We assume monotonicity in the second argument, i.e., $b'(x, \xi) := \partial_{\xi} b(x,\xi) \ge 0$ for all $x \in \Omega$ and $\xi \in \R$. In order to avoid technicalities, we assume that\footnote{The assumption $b(0)=0$ is without loss of generality, since we could consider $\tilde{b}(v):= b(v) - b(0)$ and $\tilde{f} := f - b(0)$ instead.} $b(x,0)=0$.

\end{itemize}
To establish continuity of $\prod{b(v)}{w}$ resp.~$\prod{b'(v) \varphi}{w}$, we impose the following growth condition on $b(v)$; see, e.g., \cite[Chapter III,~(12)]{fk1980} or~\cite[(A4)]{bhsz2011}:
\begin{itemize}
\item[{\crtcrossreflabel{\rm{(GC)}}[assump:poly]}] If $d \in \{1,2\}$, let $n \in \N$ be arbitrary with $2 \le n  < \infty$. For $d=3$, let $2 \le n \le 2^\ast - 1= (d+2)/(d-2) = 5$. Suppose that, for $d \in \{1,2, 3\}$, there exists $R > 0$ such that
\begin{align*}
|b^{(k)}(x,\xi)| \le R (1+|\xi|^{n-k}) \quad \text{ for all } x \in \Omega, \text{ all } \xi \in \R, \text{ and all } 0 \le k \le n.
\end{align*}
\end{itemize}
While~\ref{assump:poly} turns out to be sufficient for plain convergence of the later GOAFEM algorithm, we require the following stronger assumption for linear convergence and optimal convergence rates.
\begin{itemize}
\item[{\crtcrossreflabel{\rm{(CGC)}}[assump:compact]}] There holds~\ref{assump:poly}, if $d \in \{1,2\}$. If $d=3$, there holds~\ref{assump:poly} with the stronger assumption $ n \in \{ 2, 3\}$.
\end{itemize}
\begin{remark}\label{remark:reg12}
{\rm{(i)}} Let $v, w \in H^1_0(\Omega)$. To establish continuity of $(v, w) \mapsto \prod{b(v)}{w}$, we apply the H\"{o}lder inequality with H\"{o}lder conjugates $1 \le s, s' \le \infty$ to obtain that
\begin{align}
\begin{split}\label{rem:gc:1}
\prod{b(v)\!}{\!w} \!\le\! \norm{b(v)}{L^{s'}(\Omega)} \norm{w}{L^{s}(\Omega)} \!\stackrel{\mathclap{\ref{assump:poly}}}{\lesssim}\! (1\! +\! \norm{v^n}{L^{s'}(\Omega)} ) \norm{w}{L^{s}(\Omega)}\! =\! (1\! +\! \norm{v}{L^{ns'}(\Omega)}^n ) \norm{w}{L^{s}(\Omega)}.
\end{split}
\end{align}
To guarantee that $\prod{b(v)}{w} < \infty$, condition~\ref{assump:poly} has to ensure that the embedding 
\begin{align}\label{eq:embedding}
H^1_0(\Omega) \hook L^{r}(\Omega) \quad \text{ is continuous } \quad \text{ for } \quad r = s \quad  \text{ and } \quad r = ns'.
\end{align}
 If $d \in \{1,2\}$,~\eqref{eq:embedding} holds true for arbitrary $1 \le r < \infty$ and hence arbitrary $1 < s < \infty$ and $n \in \N$. If $d=3$, $r = s = 2^\ast=6$ is the maximal index in~\eqref{eq:embedding}. Hence, it follows that $ n \le 2^\ast/s' = 2^\ast/{2^\ast}' = 2^\ast-1 = 5$. Furthermore, if $d=3$, note that it suffices to consider $n = 2^\ast-1$ since for $n < 2^\ast -1$, we can estimate $(1+|\xi|^{n-k}) \lesssim (1 + |\xi|^{2^\ast-1-k})$ for all $\xi \in \R$, and all $0 \le k \le n < 2^\ast -1$. Altogether, we conclude continuity of $(v, w) \mapsto \prod{b(v)}{w}$ for all $n \in \N$ if $d \in \{1,2\}$, and $n \le 5$ if $d =3$.
\newline
{\rm{(ii)}} Let $v, w, \varphi \in H^1_0(\Omega)$. In the same spirit, we establish continuity of $(v, w, \varphi) \mapsto \prod{b'(v) \varphi}{w}$. If $d \in \{1,2\}$, for arbitrary $1 < t < \infty$, we use the generalized H\"{o}lder inequality; see, e.g., \cite[Section~2.2]{kof1977}. To this end, define $t''$ by $1 = 1/{t''} + 1/t + 1/t$ and observe that
\begin{align}\label{rem:gc:2}
\prod{b'(v)\varphi}{w} &\le \norm{b'(v)}{L^{t''}(\Omega)}\norm{\varphi}{L^{t}(\Omega)}\norm{w}{L^{t}(\Omega)} \stackrel{\mathclap{\ref{assump:poly}}}{\le} \, (1+\norm{v^{n-1}}{L^{t''}(\Omega)})\norm{\varphi}{L^{t}(\Omega)}\norm{w}{L^{t}(\Omega)}.
\end{align}
Using $\norm{v^{n-1}}{L^{t''}(\Omega)} = \norm{v}{L^{(n-1)t''}(\Omega)}^{n-1}$, the~\ref{assump:poly} needs to ensure that the Sobolev embedding $H^1_0(\Omega) \hook L^r(\Omega)$ is continuous for both $r=(n-1)t''$ and $r=t$. If $d \in \{1,2\}$, this holds for arbitrary $1 < t < \infty$ and $n \in \N$. If $d=3$, $r = t = 2^\ast=6$ is the maximal index in~\eqref{eq:embedding} and hence $t'' = 3/2$. The upper bound $ (n-1) \le 2^\ast/t'' = 4$ thus guarantees continuity. \newline
{\rm{(iii)}} Let $v, \varphi \in H^1_0(\Omega)$ and $w \in L^\infty(\Omega)$. Then, the reasoning of~{\rm{(ii)}} reduces to the H\"{o}lder conjugates from~{\rm{(i)}}.\newline
{\rm{(iv)}} The additional constraints on the upper bounds of $n$ in \ref{assump:compact} will become apparent later; see Remark~\ref{rem:upperbound}. \newline  
{\rm{(v)}} The lower bound $2 \le n$ imposed for $d \in \{1,2, 3\}$ stems from the necessity of a Taylor expansion of the dual problem; cf.\ \eqref{eq2:dual:tbp}. 
\hfill \qed
\end{remark}
\subsection{Assumptions on the right-hand sides}\label{subsection:rhs}
For $d=1$, the exact solution $u^\exact$ from~\eqref{eq:weakform:primal} and dual solutions $\tilde{z}^\exact[w]$ and $z^\exact[w]$ with arbitrary $w \in H^1_0(\Omega)$ from~\eqref{eq:weakform:theoretical} and~\eqref{eq:weakform:dual} below satisfy $L^\infty$-bounds, since $H^1$-functions are absolutely continuous. For $d \in \{2, 3\}$, we need the following assumption:
\begin{itemize}
\item[{\crtcrossreflabel{\rm{(RHS)}}[assump:rhs]}] We suppose that the right-hand sides fulfil that
\begin{align*}
\f, \g \in L^{p}(\Omega) \text{ for some } p  > d \ge 2 \quad \text { and } \quad f, g \in L^{q}(\Omega) \text{ where } 1/q := 1/p + 1/d.
\end{align*}
\end{itemize}
To guarantee later that the residual \textsl{a~posteriori} error estimators from~\eqref{eq:estimator:primal}--\eqref{eq:estimator:dual} are well-defined, we additionally require that $\f|_T, \g|_T \in H(\div,T)$ and $\f|_T \cdot \n, \,\g|_T \cdot \n \in L^2(\partial T)$ for all $T \in \TT_0$, where $\TT_0$ is the initial triangulation of the adaptive algorithm.
\subsection{Well-posedness of primal problem}\label{subsection:browder-minty}
First, we deal with the continuous primal problem~\eqref{eq:weakform:primal}. With the dual space $H^{-1}(\Omega) := H^1_0(\Omega)^\ast$, we consider the operator
\begin{align}\label{eq:op:bm}
 \AA \colon H^1_0(\Omega) \to H^{-1}(\Omega),
 \quad 
 \AA w := \sprod{w}{\cdot \,} + \prod{b(w)}{\cdot \,}.
\end{align}
Assumption~\ref{assump:poly} and the resulting estimate~\eqref{rem:gc:1} yield that
\[
\prod{b(v)}{w}  \stackrel{\eqref{rem:gc:1}}{\lesssim} \norm{b(v)}{L^{s'}(\Omega)} \, \enorm{w} < \infty.
\]
Together with the continuity of $\sprod{ \,\cdot }{ \cdot \,}$, we infer that $\AA$ is well-defined.

For $x \in \Omega$ and $\xi_1, \xi_2 \in \R$ with $\xi_1 < \xi_2$, the mean value theorem proves that
\begin{align*}
 b(x, \xi_2) - b(x, \xi_1) = b'(x,\zeta) (\xi_2-\xi_1) 
 \quad \text{for some } \xi_1 < \zeta < \xi_2.
\end{align*}
Since $b'(x,\zeta) \ge 0$ according to~\ref{assump:mon}, this implies that
\begin{align}\label{eq:aux:bmmon}
 \big( b(x, \xi_2) - b(x, \xi_1) \big) (\xi_2-\xi_1) \ge 0
 \quad \text{for all } x \in \Omega \text{ and } \xi_1, \xi_2 \in \R.
\end{align}
Together with~\ref{assump:ell}, we thus see that
\begin{align}
\begin{split}\label{eq:bm:stronglymonotone}
 \langle \AA w - \AA v \,,\, w - v \rangle
 &= \sprod{w - v}{ w - v} + \prod{b(w)-b(v)}{w-v}
 \\&
 \ge \enorm{w-v}^2 
 \simeq \norm{\nabla (w-v)}{L^2(\Omega)}^2
 \quad \text{for all } v, w \in H^1_0(\Omega),
 \end{split}
\end{align}
where the hidden constant depends only on $\mu_0$ from~\ref{assump:ell}. This proves that $\AA$ is strongly monotone and hence, in particular, monotone and coercive. Moreover, the solution $u^\exact \in H^1_0(\Omega)$ of~\eqref{eq:weakform:primal} is necessarily unique. Finally, recall from~\ref{assump:car} that $b$ is smooth in $\xi$. Therefore, the mapping
\begin{align*}
 \tau \mapsto \int_\Omega b(v+\tau w) \varphi \d{x} \in \R
 \quad \text{for } \tau \in [0,1] \text{ and } v, w, \varphi \in H^1_0(\Omega)
\end{align*}
is continuous, i.e., $\AA$ is hemi-continuous. Overall, the Browder--Minty theorem (see, e.g.,~\cite[Theorem~26.A~(a)--(c)]{zeidler}) applies and proves that the primal problem~\eqref{eq:weakform:primal} admits a (unique) solution $u^\exact \in H^1_0(\Omega)$.
The same argument shows that the discrete primal problem~\eqref{eq:weakform:primal:discrete} admits a unique solution $u_\coarse^\exact \in \XX_\coarse$.
\subsection{Well-posedness of dual problem and goal error identity}\label{subsection:wellposed:dual}
For $v, w \in H^1_0 (\Omega)$, define
\begin{align}\label{eq:aux:integral}
\B(w, v) := \int_0^1 b'\big(w + (v- w)\tau \big) \d{\tau} \ge 0 \quad \text{ a.e. in } \Omega.
\end{align}
Note that $\B(w, v ) \colon \Omega \to \R_{\ge0}$. If $v = u^\exact$, we introduce the shorthand $\B^\exact(w) := \B(w, u^\exact)$. With this notation, the \emph{theoretical} dual problem reads as follows: Find $\tilde z^\exact[w] \in H^1_0(\Omega)$ and $\tilde z^\exact_\coarse[w] \in \XX_\coarse$ such that
\begin{subequations}\label{eq:weakform:theoretical}
\begin{alignat}{2}
 \sprod{\tilde z^\exact[w]}{ v} +
  \prod{\B^\exact(w) \tilde z^\exact[w]}{v} 
  &= G(v)
 \quad &&\text{for all } v \in H^1_0(\Omega), \label{eq:aux:theoretical:dual}
\\
 \sprod{\tilde z^\exact_\coarse[w]}{ v_\coarse} +
  \prod{\B^\exact(w) \tilde z^\exact_\coarse[w]}{v_\coarse} 
  &= G(v_\coarse)
 \quad &&\text{for all } v_\coarse \in \XX_\coarse. \label{eq:aux:theoretical:dual:discrete}
\end{alignat}
\end{subequations}
To address well-posedness, we show that ~\ref{assump:poly} implies that $\int_\Omega |\B^\exact(w)zv| \d{x}< \infty$ for all $v, w, z \in H^1_0(\Omega)$. The cases $d \in \{1,2\}$ are covered, e.g., in~\cite[Lemma A.1]{aw2015}. If $d=3$, we exploit~\ref{assump:compact} and apply the same reasoning as for the estimate~\eqref{rem:gc:2} to obtain that
\begin{align}\label{eq:reg2:dual}
\prod{\B^\exact(w) z}{v} \lesssim \norm{\B^\exact(w)}{L^{t''}(\Omega)} \enorm{z} \enorm{v} < \infty.
\end{align}
Using \ref{assump:ell} and \ref{assump:mon} for coercivity (see, e.g., \eqref{eq:bm:stronglymonotone} above), the Lax--Milgram lemma proves existence and uniqueness of $\tilde z^\exact[w] \in H^1_0(\Omega)$ and $\tilde z^\exact_\coarse[w] \in \XX_\coarse$. 

According to the Taylor theorem, it holds that
\begin{align}\label{eq:aux:taylor}
 b(u^\exact) - b(w) = (u^\exact - w) \, \B^\exact(w)
 \quad \text{ in } \Omega.
\end{align}
For any approximation $\tilde z^\exact[u_\coarse^\exact] \approx z_\coarse \in \XX_\coarse$, this yields the error identity
\begin{align}\label{eq:aux:error-identity}
 \begin{split}
  G(u^\exact) - G(u_\coarse^\exact) &
~\stackrel{\mathmakebox[\widthof{=}]{\eqref{eq:aux:theoretical:dual}}}{=}~ \sprod{\tilde z^\exact[u^\exact_\coarse]}{ u^\exact - u_\coarse^\exact} + \prod{\B^\exact(u^\exact_\coarse) \tilde z^\exact[u^\exact_\coarse]}{u^\exact - u_\coarse^\exact}
 \\
 &~\stackrel{\mathmakebox[\widthof{=}]{\eqref{eq:aux:taylor}}}{=} ~  \sprod{u^\exact - u_\coarse^\exact}{ \tilde z^\exact[u^\exact_\coarse]} + \prod{b(u^\exact) - b(u^\exact_\coarse)}{\tilde z^\exact[u^\exact_\coarse]}
 \\
 & ~\stackrel{\mathmakebox[\widthof{=}]{\eqref{eq:weakform:primal},\, \eqref{eq:weakform:primal:discrete}}}{=} \hspace{2mm}  \sprod{u^\exact - u_\coarse^\exact}{ \tilde z^\exact[u^\exact_\coarse] - z_\coarse} + \prod{b(u^\exact) - b(u^\exact_\coarse)}{\tilde z^\exact[u^\exact_\coarse] - z_\coarse}.
 \end{split}
\end{align}
While this error identity looks similar to the one for linear problems (see, e.g.,~\cite{ms2009,bet2011,fpz2016} as well as~\cite{bgip2020+} in the presence of inexact solvers), we stress that it suffers from one essential shortcoming: The theoretical dual problem~\eqref{eq:weakform:theoretical} involves $\B^\exact(u^\exact_\coarse)= \B(u^\exact_\coarse, u^\exact)$ which depends on the unknown exact solution $u^\exact$. Consequently, the corresponding bilinear form cannot be implemented in practice and, hence, $\tilde z^\exact[u^\exact_\coarse]$ cannot be approximated by its FEM solution $\tilde z_\coarse^\exact[u^\exact_\coarse]$.

However, it follows formally that $\B^\exact(u^\exact_\coarse) - b'(u^\exact_\coarse) \to 0$ as $u^\exact_\coarse \to u^\exact$. 
Hence, we introduce the \emph{practical} dual problem~\eqref{eq:weakform:dual:exact} and its discretization~\eqref{eq:weakform:dual:discrete:exact}, now considered for a general argument:
Given $w \in H^1_0(\Omega)$, find $z^\exact[w] \in H^1_0(\Omega)$ and $z^\exact_\coarse[w] \in \XX_\coarse$ such that
\begin{subequations}\label{eq:weakform:dual}
\begin{alignat}{2}
 \sprod{z^\exact[w]}{v} + \prod{b'(w) z^\exact[w]}{v} &= G(v)
 \quad &&\text{for all } v \in H^1_0(\Omega), \label{subeq:weakform:dual}
 \\
 \sprod{z^\exact_\coarse[w]}{v_\coarse} + \prod{b'(w) z^\exact_\coarse[w]}{v_\coarse} &= G(v_\coarse)
 \quad &&\text{for all } v_\coarse \in \XX_\coarse. \label{subeq:weakform:dual:discrete}
\end{alignat}
\end{subequations}
The same arguments as for the theoretical problem~\eqref{eq:weakform:dual} apply and prove existence and uniqueness of $z^\exact[w] \in H^1_0(\Omega)$ and $z^\exact_\coarse[w] \in \XX_\coarse$.

Overall, the error identity \eqref{eq:aux:error-identity} for $z_\coarse = z_\coarse^\exact[u_\coarse^\exact]$ then takes the following form
\begin{align}
 \begin{split}\label{eq2:aux:error-identity}
\hspace{-2mm}G(u^\exact) - G(u_\coarse^\exact)
\reff{eq:aux:error-identity}= &\sprod{u^\exact - u_\coarse^\exact}{ z^\exact[u_\coarse^\exact] - z_\coarse^\exact[u_\coarse^\exact]} + \prod{b(u^\exact) - b(u_\coarse^\exact)}{z^\exact[u_\coarse^\exact] - z_\coarse^\exact[u_\coarse^\exact]}
 \\& \hspace{-5mm} 
 + \sprod{u^\exact - u_\coarse^\exact}{ \tilde z^\exact[u_\coarse^\exact] - z^\exact[u_\coarse^\exact]} + \prod{b(u^\exact) - b(u_\coarse^\exact)}{\tilde z^\exact[u_\coarse^\exact] - z^\exact[u_\coarse^\exact]}.
 \end{split}
\end{align}
This identity will be the starting point for proving the goal error estimate~\eqref{eq:aux:error-estimate}; see Theorem~\ref{theorem:errorestimate} below for the formal statement.
\subsection{Pointwise boundedness of primal and dual solutions}
\label{subsection:maxprinciple}
In this section, we prove that imposing regularity assumptions on the right-hand side yields that the exact solution $u^\exact$ and the dual solutions $\tilde z^\exact[w]$ and $z^\exact[w]$ are bounded in $L^\infty(\Omega)$. For $d=1$, this is immediate, since $H^1(\Omega) \subset C(\overline{\Omega})$. For $d \in \{2,3\}$ and $\f = 0$, we refer to, e.g.,~\cite[Theorem~2.2]{bhsz2011}. These results turn out to be crucial for the goal error estimate (Theorem~\ref{theorem:errorestimate}) as well as for the numerical analysis of the proposed adaptive goal-oriented strategy (Algorithm~\ref{algorithm:first}). In particular, they also allow one to derive C\'{e}a-type estimates for the discrete primal and dual solutions (Proposition~\ref{prop:cea},~\ref{prop:cea:dual}).

\begin{proposition}\label{proposition:weakform:primal:Linf}
Suppose~\ref{assump:rhs},~\ref{assump:ell},~\ref{assump:car},~\ref{assump:mon}, and~\ref{assump:poly}. 
Then, the weak solution $u^\exact \in H^1_0(\Omega)$ of~\eqref{eq:weakform:primal} is bounded in $L^\infty(\Omega)$. In particular, for $d \in \{2,3\}$, there holds 
\begin{align}\label{eq:weakform:primal:bounds}
\norm{u^\exact}{L^\infty(\Omega)} \le C \, \mu_0^{-1} \,
|\Omega|^{(1/d - 1/p )}  \,
 (\norm{f}{L^q(\Omega)} + \norm{\f}{L^p(\Omega)})
\end{align}
with a constant $C = C(d, p) > 0$.
\end{proposition}
\begin{remark}
In this remark, we consider special choices of $p$ and $q$ from~\ref{assump:rhs}. If $d=2$ and $p:= \infty$, then $q = 2$. If $d=3$ and $p := 2^\ast = 6$, then also $q=2$. In~\cite[Theorem~2.2]{bhsz2011}, the following statement is proven with a slightly simplified proof: Suppose $f \in L^2(\Omega)$ and $\f = 0$ as well as~\ref{assump:ell},~\ref{assump:car}, and ~\ref{assump:mon}. Then, the weak solution $u^\exact \in H^1_0(\Omega)$ of~\eqref{eq:weakform:primal} is bounded in $L^\infty(\Omega)$, i.e.,
$\norm{u^\exact}{L^\infty(\Omega)} \le C(\Omega, d) \,\mu_0^{-1} \, \norm{f}{L^2(\Omega)}$. \qed
\end{remark}
The proof of Proposition~\ref{proposition:weakform:primal:Linf} requires the following elementary result from~\cite[Lemma~4.1.1]{wwy2006}:
\begin{lemma}\label{lemma:max:bound}
With positive constants $C, \kappa_0 > 0 $ and $\kappa_1 > 1$, let  $\phi\colon \R_{\ge 0} \to \R_{\ge 0}$ satisfy
\begin{align}
0 \le \phi(\Lambda) \le \phi(\lambda) \quad \text{ and } \quad \phi(\Lambda) \le \left( \frac{C}{\Lambda - \lambda} \right)^{\kappa_0} \, \phi(\lambda)^{\kappa_1} \quad \text{ for all }~0 \le \lambda < \Lambda.
\end{align}
Then, there holds that
$\phi(\Lambda) = 0$ for all $\Lambda \ge  C \, \phi(0)^{(\kappa_1 - 1)/{\kappa_0}} \, 2^{{\kappa_1}/(\kappa_1-1)}$.\qed
\end{lemma}
Furthermore, the proof of Proposition~\ref{proposition:weakform:primal:Linf} requires the Gagliardo--Nirenberg--Sobolev inequality; see, e.g.,~\cite[Theorem~16.6]{fk1980} or \cite[Theorem~12.7]{c2009}:
\begin{lemma}[Gagliardo--Nirenberg--Sobolev inequality]\label{lemma:GNS}
Let $\Omega \subseteq \R^d$ be open and bounded and suppose that $1 \le p < \infty$. If $1 \le p < d$, let $1 \le r \le p^\ast$. If $d \le p < \infty$, let $1 \le r < \infty$.
Then, there exists a constant $\CGNS' = \CGNS'(d, p, r)$ such that  
\begin{align}\label{eq:lemma:gns}
\norm{v}{L^{r}(\Omega)} \le \CGNS \,  \norm{\nabla v}{L^{p}(\Omega)} \quad \text{ for all } v \in W^{1,p}_0(\Omega),
\end{align}
where $\CGNS := \CGNS' \,|\Omega|^{1/d - 1/p + 1/r}$.
In case of $1 \le p < d$ and $r = p^\ast$, \linebreak$\CGNS = p(d-1)/(d-p)$ depends only on $d$ and $p$. \qed
\end{lemma}
\begin{proof}[Proof of Proposition~\ref{proposition:weakform:primal:Linf}]
If $d=1$, $u^\exact \in C(\overline{\Omega}) \subset L^\infty(\Omega)$ holds due to the Sobolev embedding. If $d \in \{2,3\}$ and for $\lambda \ge 0$, we define the test function
\begin{align*}
\varphi^{+}_{\lambda}(x) := \max\{u^\exact(x) - \lambda, 0 \}
\end{align*}
and recall from~\cite[Theorem~12.4]{c2009} that $\varphi^{+}_{\lambda} \in H^1_0(\Omega)$ with
\begin{align}\label{eq:max:gset}
\nabla \varphi^{+}_{\lambda}(x) = \nabla u^\exact(x) \quad \text{for almost all } x \in \Omega(\lambda) := \set{x \in \Omega}{u^\exact(x) >\lambda}.
\end{align}
From~\eqref{eq:aux:bmmon}, it follows that $\left\langle b(u^\exact) -b(\lambda)\, , \,   u^\exact - \lambda \right\rangle_{\Omega(\lambda)} \ge 0$. Using \ref{assump:mon}, we see that $b(\lambda) \ge b(0) = 0$ and hence $\left\langle b(\lambda) \, , \,   u^\exact - \lambda \right\rangle_{\Omega(\lambda)} \ge 0.$
Using the coercivity assumption~\ref{assump:ell} and testing the weak formulation~\eqref{eq:weakform:primal} with $\varphi^{+}_{\lambda}$, we observe that 
\begin{align}\label{eq:max:cor:helper}
\mu_0 \norm{\nabla \varphi^{+}_{\lambda}}{L^2(\Omega)}^2 
& \hspace{-2mm} \stackrel{\ref{assump:ell}}{\le} \sprod{\varphi^{+}_{\lambda}}{ \varphi^{+}_{\lambda}}=\sprod{u^\exact}{ \varphi^{+}_{\lambda}} \stackrel{\eqref{eq:weakform:primal}}{=} \prod{f}{\varphi^{+}_{\lambda}} 
 + \prod{\f}{\nabla \varphi^{+}_{\lambda}} - \prod{b(u^\exact)}{\varphi^{+}_{\lambda}} \nonumber\\
 &\hspace{-10mm} = \prod{f}{\varphi^{+}_{\lambda}} 
 + \prod{\f}{\nabla \varphi^{+}_{\lambda}} - \langle b(u^\exact) \, , \,u^\exact - \lambda \, \rangle_{\Omega(\lambda)} \nonumber \\
 & \hspace{-10mm} = \prod{f}{\varphi^{+}_{\lambda}} 
 + \prod{\f}{\nabla \varphi^{+}_{\lambda}} - \langle b(u^\exact) - b(\lambda) \, , \,u^\exact - \lambda \, \rangle_{\Omega(\lambda)} - \langle b(\lambda) \, , \,u^\exact - \lambda \, \rangle_{\Omega(\lambda)} \nonumber\\
 & \hspace{-10mm} \le \prod{f}{\varphi^{+}_{\lambda}} + \prod{\f}{\nabla \varphi^{+}_{\lambda}}.
\end{align}
With the H\"older inequality, we arrive at
\begin{align}\label{eq:hoelder:remark}
\mu_0 \norm{\nabla \varphi^{+}_{\lambda}}{L^{2}(\Omega)}^2 
&\le  \norm{f}{L^q(\Omega)} \, \norm{\varphi^{+}_{\lambda}}{L^{q'}(\Omega)} 
+ \norm{\f}{L^p(\Omega)} \,  
\norm{\nabla \varphi^{+}_{\lambda}}{L^{p'}(\Omega)}.
\end{align} 
Moreover,~\ref{assump:rhs} yields that $1/q' = 1- 1/q = 1 - 1/p - 1/d = 1/p' - 1/d = 1/p'^\ast$, where $p' < 2 \le d$. Since $|\Omega| < \infty$, we have that $H^1_0(\Omega) \hook W^{1,p'}_0(\Omega)$. Therefore, Lemma~\ref{lemma:GNS} (applied to $1 \le p' < d$ and $r = {p'}^\ast = q'$) yields that
\begin{align}\label{eq:max:estimate}
\norm{\varphi^{+}_{\lambda}}{L^{q'}(\Omega)} \le \CGNS\,\norm{\nabla \varphi^{+}_{\lambda}}{L^{p'}(\Omega)},
\end{align}
where $\CGNS$ depends only on $d, p'$. Collecting~\eqref{eq:max:cor:helper}--\eqref{eq:max:estimate}, we obtain that
\begin{align}\label{eq:max:helper}
\norm{\nabla \varphi^{+}_{\lambda}}{L^{2}(\Omega)}^2 
&\le  C_1 \,  
 \norm{\nabla \varphi^{+}_{\lambda}}{L^{p'}(\Omega)} \enspace \text{ with } \enspace C_1 =  \frac{\max\{\CGNS, 1 \}}{\mu_0} \left( \norm{f}{L^q(\Omega)} 
 + \norm{\f}{L^p(\Omega)} \right).
\end{align}

Now, we aim at a lower bound for the left-hand side of \eqref{eq:max:helper}. Recall the definition of $\Omega(\lambda)$ from~\eqref{eq:max:gset}. Applying the H\"{o}lder inequality, we observe that
\[
\norm{\nabla \varphi^{+}_{\lambda}}{L^{p'}(\Omega)}^{p'}  
=  \int_{\Omega(\lambda)} | \nabla \varphi^{+}_{\lambda}|^{p'} \d{x} 
\le  \left( \int_{\Omega(\lambda)} |\nabla \varphi^{+}_{\lambda}|^{2} \d{x} \right)^{p'/2} \, |\Omega(\lambda)|^{1-p'/2}.
\]
Taking the last equation to the power of $2/p' > 1$, we show that
\begin{align} \label{eq:max:L2estimate}
\norm{\nabla \varphi^{+}_{\lambda}}{L^{p'}(\Omega)}^{2}  \le \norm{\nabla \varphi^{+}_{\lambda}}{L^{2}(\Omega)}^{2} \, |\Omega(\lambda)|^{2/p' -1} \, \stackrel{\eqref{eq:max:helper}}{\le} \, C_1 \,\CGNS \, \norm{\nabla \varphi^{+}_{\lambda}}{L^{p'}(\Omega)} \, |\Omega(\lambda)|^{2/p' -1}.
\end{align} 
In combination with~\eqref{eq:max:estimate}, we arrive at
\begin{align}\label{eq:max:helper2}
\norm{\varphi^{+}_{\lambda}}{L^{q'}(\Omega)} \le \CGNS \, \norm{\nabla \varphi^{+}_{\lambda}}{L^{p'}(\Omega)} 
 \le C_1\, \CGNS^{2} \, |\Omega(\lambda)|^{2/p' -1}.
\end{align}
For $0 < \lambda < \Lambda$, we observe that
\[
\Omega(\lambda) \supseteq \Omega(\Lambda).
\]
This observation and $u^\exact  > \Lambda$ on $\Omega(\Lambda)$ provide a lower bound for the left-hand side of~\eqref{eq:max:helper2}:
\[
\norm{\varphi^{+}_{\lambda}}{L^{q'}(\Omega)}
 \ge \left( \int_{\Omega(\Lambda)} \left(  \max\{u^\exact(x) - \lambda, 0 \} \right)^{q'} \d{x} \right)^{1/q'} 
 \ge (\Lambda-\lambda) |\Omega(\Lambda)|^{1/q'}. 
\]
Combining this estimate with \eqref{eq:max:helper2}, we see that
\begin{align}\label{eq:max:intermediate}
|\Omega(\Lambda)| \le \left(\frac{C_2}{\Lambda-\lambda} \right)^{q'}\, 
|\Omega(\lambda)|^{q' \, (2/p' - 1)} \quad \text{ with } \quad C_2 = C_1 \, \CGNS.
\end{align}
Recall that $1/\kappa_0 := 1/q' = 1 - 1/p - 1/d$. Together with $p > d \ge 2$, we thus observe that
\[
\kappa_1 := q' \left( 2/p' - 1 \right) = (2 - 2/p - 1 )/(1- 1/p - 1/d) = (1-2/p) / (1-1/p - 1/d) > 1.
\]
Therefore, we are able to apply Lemma~\ref{lemma:max:bound} to~\eqref{eq:max:intermediate}. This yields that $|\Omega(\Lambda)| = 0$ for $\Lambda \ge  C_2 \, |\Omega(0)|^{(\kappa_1 - 1)/\kappa_0} \, 2^{\kappa_1/(\kappa_1-1)}$. By definition of $\Omega(\Lambda)$, this proves that
\[
u^\exact(x) \le  C_2 \, |\Omega|^{(1/d - 1/p )}  \,
 2^{(1/d - 1/p)/(1-2/p)} \quad \text{ for almost all } x \in \Omega.
\]
To see that $-u^\exact$ satisfies the same bound, we argue analogously. For $\lambda \ge 0$, we define the test function $\varphi^{-}_{\lambda} := \min\{u^\exact(x) - \lambda, 0 \} \le 0$ and observe that $\varphi^{-}_{\lambda}  \in H^1_0(\Omega)$. With $\Omega(\lambda) := \set{x \in \Omega}{u(x) < - \lambda }$ and the above arguments, we then conclude the proof.
\end{proof}
In the same spirit as in Proposition~\ref{proposition:weakform:primal:Linf}, we are able to establish $L^\infty$-bounds for the solutions of the theoretical and practical dual problems~\eqref{eq:aux:theoretical:dual} and~\eqref{subeq:weakform:dual}.
\begin{proposition}\label{proposition:weakform:dual:exact:Linf}
Suppose~\ref{assump:rhs},~\ref{assump:ell},~\ref{assump:car},~\ref{assump:mon}, and~\ref{assump:poly}. Let $w \in H^1_0(\Omega)$. 
Then, the weak solutions $\tilde{z}^\exact[w] \in H^1_0(\Omega)$ of the theoretical dual problem~\eqref{eq:aux:theoretical:dual} and $z^\exact[w] \in H^1_0(\Omega)$ of the practical dual problem~\eqref{subeq:weakform:dual} are bounded in $L^\infty(\Omega)$. In particular, for $d \in \{2,3\}$, there holds 
\begin{align}\label{eq:weakform:dual:bounds}
\norm{\tilde{z}^\exact[w]}{L^\infty(\Omega)} + \norm{z^\exact[w]}{L^\infty(\Omega)} \le C \, \mu_0^{-1} \,
|\Omega|^{(1/d - 1/p )}  \,
 (\norm{g}{L^q(\Omega)} + \norm{\g}{L^p(\Omega)})
\end{align}
with a constant $C = C(d, p) > 0$, which is, in particular, independent of $w$.
\end{proposition}
\begin{proof}
We argue as for Proposition~\ref{proposition:weakform:primal:Linf}. The case $d=1$ follows from the Sobolev embedding. For $d \in \{2,3\}$ and for $\lambda \ge 0$, we define the test function
\begin{align*}
\varphi^{+}_{\lambda}(x) := \max\{\tilde{z}^\exact[w](x) - \lambda, 0 \}
\end{align*}
and recall that $\varphi^{+}_{\lambda} \in H^1_0(\Omega)$ with
\begin{align*}
\nabla \varphi^{+}_{\lambda}(x) = \nabla \tilde{z}^\exact[w](x) \quad \text{for almost all } x \in \Omega(\lambda) := \set{x \in \Omega}{\tilde{z}^\exact[w] >\lambda}.
\end{align*}
From \eqref{eq:aux:integral}, recall that $\B^\exact(w)= \B(w, u^\exact) \ge 0$. In particular, it follows that \linebreak[4] $\langle \B^\exact(w)\tilde{z}^\exact[w] \, , \,\tilde{z}^\exact[w] - \lambda \rangle_{\Omega(\lambda)} \ge 0$.
Using the coercivity assumption~\ref{assump:ell} and testing the weak formulation~\eqref{eq:aux:theoretical:dual} with $\varphi^{+}_{\lambda}$, we observe that 
\begin{align*}
\mu_0 \norm{\nabla \varphi^{+}_{\lambda}}{L^2(\Omega)}^2 
& \hspace{-2mm} \stackrel{\ref{assump:ell}}{\le} \sprod{\varphi^{+}_{\lambda}}{ \varphi^{+}_{\lambda}}=\sprod{\tilde{z}^\exact[w]}{\varphi^{+}_{\lambda}} \nonumber \\
 &  \stackrel{\mathclap{\eqref{eq:aux:theoretical:dual}}}{=} \, \prod{g}{\varphi^{+}_{\lambda}} 
 + \prod{\g}{\nabla \varphi^{+}_{\lambda}} - \prod{\B^\exact(w) \tilde{z}^\exact[w]}{\varphi^{+}_{\lambda}} \nonumber\\
 &= \prod{g}{\varphi^{+}_{\lambda}} 
 + \prod{\g}{\nabla \varphi^{+}_{\lambda}} - \langle \B^\exact(w) \tilde{z}^\exact[w] \, , \,\tilde{z}^\exact[w] - \lambda \, \rangle_{\Omega(\lambda)} \nonumber \\
 & \le \prod{g}{\varphi^{+}_{\lambda}} + \prod{\g}{\nabla \varphi^{+}_{\lambda}}.
\end{align*}
Following the steps of the proof of Proposition~\ref{proposition:weakform:primal:Linf} (where the latter estimate corresponds to \eqref{eq:max:cor:helper}), we conclude the proof for $\tilde{z}[w]$. The same argument applies for the practical dual problem, where $\B^\exact(w)$ is replaced by $b'(w) \ge 0$. This concludes the proof.
\end{proof}


\subsection{Goal error estimate}\label{subsection:goalerrorestimate}
The following theorem provides, up to norm equivalence, the formal statement of the goal error estimate~\eqref{eq:aux:error-estimate}.
\begin{theorem}\label{theorem:errorestimate}
Suppose~\ref{assump:rhs},~\ref{assump:ell},~\ref{assump:car},~\ref{assump:mon}, and~\ref{assump:poly}. Let $u^\exact \in H^1_0(\Omega)$ solve~\eqref{eq:weakform:primal} and $u^\exact_\coarse \in \XX_\coarse$ be its approximation~\eqref{eq:weakform:primal:discrete}. Then, it holds that 
\begin{align}\label{eq:critical}
 | G(u^\exact) - G(u_\coarse^\exact) |
 \le \Cest \, \left[\enorm{u^\exact - u_\coarse^\exact} \enorm{z^\exact[u_\coarse^\exact] - z_\coarse^\exact[u_\coarse^\exact]} + \enorm{u^\exact - u_\coarse^\exact}^2 \right],
\end{align}
where $\Cest = \Cest(|\Omega|, d, \norm{u^\exact}{L^\infty(\Omega)}, n, R, p, f, \f, g, \g, \mu_0)$.
\end{theorem}
The proof of Theorem~\ref{theorem:errorestimate} requires some preparations. We start with the following lemma which extends \cite[Lemma 3.1]{bhsz2011} to $\f \neq 0$.
\begin{lemma}\label{lemma:seminorm} Suppose~\ref{assump:rhs},~\ref{assump:ell},~\ref{assump:car},~\ref{assump:mon}, and~\ref{assump:poly}. Let $w \in H^1_0(\Omega)$. Then, it holds that
\begin{subequations} \label{eq:seminorm}
\begin{equation}\label{eq:seminorm:primal}
 \enorm{u^\exact}+ \enorm{u_\coarse^\exact} \le \Cbnd,
 \end{equation}
 \begin{equation}\label{eq:seminorm:dual}
 \enorm{\tilde z^\exact[w]} + \enorm{\tilde z^\exact_\coarse[w]} + \enorm{z^\exact[w]} + \enorm{z^\exact_\coarse[w]} \le \Cbnd,
\end{equation}
\end{subequations}
where $\Cbnd = \Cbnd(|\Omega|, d, p, f, \f, \mu_0)$ for~\eqref{eq:seminorm:primal} and $\Cbnd = \Cbnd(|\Omega|, d, p, g, \g, \mu_0)$ \linebreak for~\eqref{eq:seminorm:dual}. The constant $\Cbnd$ is independent of $w \in H^1_0(\Omega)$.
\end{lemma}
\begin{proof} 
In the case $d=1$,~\eqref{eq:seminorm} follows from the Sobolev embedding and~\ref{assump:ell}. Moreover, note that $b(0) = 0$ and \eqref{eq:aux:bmmon} prove that $\prod{b(u^\exact)}{u^\exact} \ge 0$.
Using \ref{assump:ell}, \ref{assump:mon}, and the H\"{o}lder inequality, we obtain that
\begin{align}
\begin{split}
\enorm{u^\exact}^2 &= \sprod{u^\exact}{u^\exact} \stackrel{\mathclap{\eqref{eq:weakform:primal}}}{=} \prod{f}{u^\exact} 
 + \prod{\f}{\nabla u^\exact} - \prod{b(u^\exact)}{u^\exact}\\
  & \le  \norm{f}{L^q(\Omega)} \, \norm{u^\exact}{L^{q'}(\Omega)} 
+ \norm{\f}{L^p(\Omega)} \,  
\norm{\nabla u^\exact}{L^{p'}(\Omega)}.
\end{split}
\end{align}
Arguing as for \eqref{eq:max:estimate} and applying the H\"{o}lder inequality, we see that
\begin{align}
\enorm{u^\exact}^2 &\le  \max\{\CGNS, 1 \} \, \left( \norm{f}{L^q(\Omega)} 
 + \norm{\f}{L^p(\Omega)} \right) \, \norm{\nabla u^\exact}{L^{p'}(\Omega)} \\
 &\le  \max\{\CGNS, 1 \} \, \left( \norm{f}{L^q(\Omega)} 
 + \norm{\f}{L^p(\Omega)} \right) \, |\Omega|^{1/p' - 1/2} \, \norm{\nabla u^\exact}{L^2(\Omega)}, \nonumber 
\end{align}
where $\CGNS$ depends only on $d$ and $p'$. With $\norm{\nabla u^\exact}{L^2(\Omega)} \le \mu_0^{-1/2} \enorm{u^\exact}$, this concludes the proof for $u^\exact$. The same argument (based on~\eqref{eq:weakform:primal:discrete} instead of~\eqref{eq:weakform:primal}) applies for $u_\coarse^\exact$. Furthermore, the same argument applies also for the dual problems (based on~\eqref{eq:weakform:theoretical} and~\eqref{eq:weakform:dual} instead of~\eqref{eq:weakform:primal} for the theoretical and practical dual problem, respectively). For $w, v \in H^1_0(\Omega)$, the  monotonicity $\prod{b(v)}{v} \ge 0$ is substituted in case of the dual problems by $\prod{b'(w)v}{v} \ge 0$ and $\prod{\B^\exact(w)v}{v} \ge 0$, respectively. This concludes the proof.
\end{proof}
The following lemma is one of the two main ingredients for the proof of Theorem~\ref{theorem:errorestimate}.
\begin{lemma}\label{prop:energybound}
Suppose~\ref{assump:rhs},~\ref{assump:ell},~\ref{assump:car},~\ref{assump:mon}, and~\ref{assump:poly}. Let $w \in H^1_0(\Omega)$ with $\enorm{w} \le M < \infty$. Then, it holds that
\begin{align}\label{eq:critical:lipschitz}
\prod{b(u^\exact) - b(w)}{v}
 \le \CLip \, \enorm{u^\exact - w} \enorm{v}
 \quad \text{for all } v \in H^1_0(\Omega)
 \end{align}
with $\CLip = \CLip(|\Omega|, d, \norm{u^\exact}{L^\infty(\Omega)}, M, n, R, p, f, \f, \mu_0)$. 
\end{lemma}
\begin{proof} We argue as in the proof of~\cite[Theorem 3.4]{bhsz2011}. With respect to Remark~\ref{remark:reg12}, choose $s > 1$ arbitrarily for $d \in \{1, 2\}$ and $s= 2^\ast$ for $d=3$. In any case, we see that
\begin{align}\label{eq:energybound:helper}
\prod{b(u^\exact) - b(w)}{v} \le \norm{b(u^\exact) - b(w)}{L^{s'}(\Omega)} \norm{v}{L^{s}(\Omega)} \le C  \,\norm{b(u^\exact) - b(w)}{L^{s'}(\Omega)} \enorm{v}, 
\end{align}
where $C := \mu_0^{-1}  \, \CGNS$. It remains to prove that 
\[
\norm{b(u^\exact) - b(w)}{L^{s'}(\Omega)} \lesssim \enorm{u^\exact -w}.
\]
Due to the smoothness assumption~\ref{assump:car}, we may consider the Taylor expansion
\begin{align}
\begin{split}\label{eq:primal:taylorexp}
b(w) &= \sum_{k=0}^{n-1} b^{(k)}(u^\exact) \frac{( w - u^\exact )^k}{k!} + \frac{  ( w- u^\exact )^{n}}{(n-1)!} \int_0^1 ( 1 - \tau )^{n-1} \, b^{(n)}\big( u^{\exact} + (w - u^\exact) \, \tau \big)  \d{\tau}.
\end{split}
\end{align}
Since $b$ is smooth and $u^\exact \in L^\infty(\Omega)$, we obtain that
\begin{align*}
\norm{b^{(k)}(u^{\exact})}{L^{\infty} (\Omega)} \le C \quad \text{ for all } k = 1, \ldots, n-1,
\end{align*}
where $C$ depends only on $\norm{u^\exact}{L^\infty(\Omega)}$, and $n$.
Moreover,~\ref{assump:poly} allows to bound the remainder term, i.e., for any $0 \le \tau \le 1$, it holds that
\begin{align*}
\norm{b^{(n)}( u^{\exact} + (w - u^\exact) \tau  )}{L^{\infty} (\Omega)} \le C,
\end{align*}
where $C$ depends only on $|\Omega|$, $n$, and $R$. The triangle inequality yields that
\begin{align}\label{eq:taylor:compact}
\norm{b(u^\exact) - b(w)}{L^{s'}(\Omega)} 
& \lesssim \sum_{k=1}^{n} \, \norm{(u^\exact -w)^k}{L^{s'}(\Omega)} 
 = \sum_{k=1}^{n}  \,  \norm{ u^\exact -w}{L^{k{s'}}(\Omega)}^k.
\end{align}
Recall from Remark~\ref{remark:reg12} that $H^1_0(\Omega) \hook L^{ks'}(\Omega)$ for all $1 \le k \le n$ by choice of $s$ and $n$. Therefore, the Gagliardo--Nirenberg--Sobolev inequality proves that
\begin{align}
\norm{b(u^\exact) - b(w)}{L^{s'}(\Omega)}
 \lesssim \sum_{k=1}^{n} \,  \norm{\nabla ( u^\exact -w)}{L^{2}(\Omega)}^k,
\end{align}
where the hidden constant depends only on $|\Omega|$, $d$, $\norm{u^\exact}{L^\infty(\Omega)}$, $n$, and $R$. With \linebreak[4] $\norm{\nabla (u^\exact -w)}{L^2(\Omega)} \simeq \enorm{u^\exact -w} \le \Cbnd + M$, this leads to
\begin{align}
\begin{split}\label{eq:estimate:b}
\norm{b(u^\exact) - b(w)}{L^{s'}(\Omega)} 
&\lesssim  \left(  \sum_{k=1}^{n}  \norm{\nabla ( u^\exact -w)}{L^{2}(\Omega)}^{k-1}  \right)\, \norm{\nabla ( u^\exact -w)}{L^{2}(\Omega)} \lesssim  \enorm{ u^\exact -w},
\end{split}
\end{align}
with hidden constants $C = C(|\Omega|, d, \norm{u^\exact}{L^\infty(\Omega)}, M, n, R, p, f, \f, \mu_0)> 0$. Together \linebreak with~\eqref{eq:energybound:helper}, this concludes the proof of~\eqref{eq:critical:lipschitz}.
\end{proof}
The following lemma is the last missing part for establishing Theorem~\ref{theorem:errorestimate}.
\begin{lemma}\label{prop:dual:energybound}
Suppose~\ref{assump:rhs},~\ref{assump:ell},~\ref{assump:car},~\ref{assump:mon}, and~\ref{assump:poly}. Let $w \in H^1_0(\Omega)$ with $\enorm{w} \le M < \infty$. Then, it holds that
\begin{align}\label{eq2:dual:estimate}
\enorm{\tilde z^\exact[w] - z^\exact[w]}
 \le \Cdual \, \enorm{u^\exact - w},
 \end{align}
where $\Cdual = \Cdual(|\Omega|, d, \norm{u^\exact}{L^\infty(\Omega)}, M, n, R, p, f, \f, g, \g, \mu_0)$.
\end{lemma}
\begin{proof}
Define $\delta := \tilde z^\exact[w] - z^\exact[w] \in H^1_0(\Omega)$. For the exact primal solution $u^\exact$, we observe that the theoretical dual problem and the practical dual problem coincide, as $\B^\exact(u^\exact) = \B(u^\exact, u^\exact ) = \int_0^1 b'(u^\exact) \d{\tau} = b'(u^\exact)$ and hence $z[u^\exact] = \tilde z[u^\exact]$. Using monotonicity and the definition of the theoretical as well as practical dual problem, we obtain that
\begin{align}
 \enorm{\delta}^2 
 &= \sprod{\delta}{\delta}
 \stackrel{\ref{assump:mon}}{\le} \,  \sprod{\delta}{\delta} + \prod{b'(w)\delta}{\delta} \nonumber
 \\& 
 \stackrel{\mathmakebox[\widthof{=}]{\eqref{subeq:weakform:dual}}}{=}  \, [ \sprod{\tilde z^\exact[w]}{ \delta} + \prod{b'(w)\tilde z^\exact[w]}{\delta} ] - G(\delta) \nonumber
 \\& 
 = [ \sprod{\tilde z^\exact[w]}{ \delta} + \prod{\B^\exact(w)\tilde z^\exact[w]}{\delta} ] - G(\delta) + \prod{[b'(w) - \B^\exact(w)] \, \tilde z^\exact[w]}{\delta} \nonumber \\
 &  \stackrel{\mathmakebox[\widthof{=}]{\eqref{eq:aux:theoretical:dual}}}{=} \, \prod{[b'(w) - b'(u^\exact) + \B^\exact(u^\exact) -  \B^\exact(w)] \, \tilde z^\exact[w]}{\delta}. \nonumber
 \end{align} 
Since Proposition~\ref{proposition:weakform:dual:exact:Linf} yields that $\tilde{z}^\exact[w] \in L^\infty(\Omega)$ independently of $w$, we can proceed as in Remark~\ref{remark:reg12}\rm{(i)}. To this end, we choose $s > 1$ arbitrarily for $d \in \{1, 2\}$ and $s=2^\ast$ for $d=3$. Assumption~\ref{assump:poly} then yields that
 \begin{align}
 \begin{split}\label{eq2:dual:est}
 \hspace{-6.5mm}\enorm{\delta}^2 \hspace{3mm}
  & \lesssim \hspace{2mm}  \left[ \norm{b'(u^\exact) - b'(w)}{L^{s'}(\Omega)} +\norm{\B^\exact(u^\exact) - \B^\exact(w)}{L^{s'}(\Omega)}  \right] \norm{\tilde z^\exact[w]}{L^\infty(\Omega)}.
  \end{split}
\end{align}
It remains to prove that
\begin{align}\label{eq2:dual:tbp}
 \norm{b'(u^\exact) - b'(w)}{L^{s'}(\Omega)} \lesssim  \enorm{u^\exact - w} \, \text { and }  \, \norm{\B^\exact(u^\exact) - \B^\exact(w)}{L^{s'}(\Omega)} \lesssim  \enorm{u^\exact - w}.
\end{align}
We observe that the change of variables $\tau \mapsto 1- \tau$ leads to
\begin{align*}
\B^\exact(w) = \B(w, u^\exact) = \int_0^1  b'\big(w+(u^\exact - w) \, \tau\big)  \d{\tau} = \int_0^1  b'\big(u^\exact+(w - u^\exact) \, \tau\big)  \d{\tau} = \B(u^\exact, w),
\end{align*}
and, hence,
\begin{align*}
 \B^\exact(u^\exact) - \B^\exact(w)
 = \int_0^1 \big[ b'\big(u^\exact\big) - b'\big(u^\exact+(w - u^\exact) \, \tau\big) \big] \d{\tau}.
\end{align*}
We only prove the second inequality of \eqref{eq2:dual:tbp}, but note that the first estimate follows for $\tau = 1$ by the subsequent arguments:
Due to the smoothness assumption~\ref{assump:car}, we may consider the Taylor expansion of the integrand $b'(u^\exact+(w - u^\exact) \, \tau)$ for $0 \le \tau \le 1$ to see that
\begin{align}
\begin{split}\label{eq2:primal:taylorexp:dual}
b'(u^\exact+(w - u^\exact) \, \tau) &= \sum_{k=1}^{n-1} b^{(k)}(u^\exact) \, \frac{(w - u^\exact)^{k-1}\, \tau^{k-1}}{(k-1)!}  \\
& \hspace{-5mm} + \frac{ (w - u^\exact)^{n-1}\, \tau^{n-1}}{(n-2)!} \, \int_0^1 ( 1 - \sigma )^{n-2} \, b^{(n)}\big( u^{\exact} + (w - u^\exact) \, \tau\, \sigma \big)  \d{\sigma}.
\end{split}
\end{align}
Since $b$ is smooth and $u^\exact \in L^\infty(\Omega)$, we obtain that
\begin{align*}
\norm{b^{(k)}(u^{\exact})}{L^{\infty} (\Omega)} \le C \quad \text{ for all } k = 2, \ldots, n-1,
\end{align*}
where $C$ depends only on $\norm{u^\exact}{L^\infty(\Omega)}$, and $n$. 
Moreover,~\ref{assump:poly} allows us to bound the remainder term, i.e., for any $0 \le \tau \, \sigma \le 1$, it holds that
\begin{align*}
\norm{b^{(n)}\big( u^{\exact} + (w - u^\exact)\, \tau\, \sigma  \big)}{L^{\infty} (\Omega)} \le C,
\end{align*}
where $C$ depends only on $|\Omega|$, $n$, and $R$. If $d \in \{1, 2\}$, note that $(n-1)s' < \infty$. If $d=3$, it holds that $(n-1)s' < 2^\ast$. Hence, we obtain for all $2 \le k \le n-1$ that
\begin{align*}
 \norm{(u^{\exact} - w)^{k-1}}{L^{s'}(\Omega)} =\norm{u^{\exact} - w}{L^{(k-1)s'}(\Omega)}^{k-1} \lesssim (\Cbnd+M)^{n-2}\,\enorm{u^\exact-w}, 
\end{align*} 
where the hidden constant depends only on norm equivalence $\enorm{\cdot} \simeq \norm{\nabla( \cdot )}{L^2(\Omega)}$. Arguing as for~\eqref{eq:taylor:compact}--\eqref{eq:estimate:b} above, we infer that
\begin{align}\label{eq:tbp:const}
\norm{b'(u^\exact)-b'(w)}{L^{s'}(\Omega)} + \norm{\B^\exact(u^\exact) - \B^\exact(w)}{L^{s'}(\Omega)} 
& \le \Cdual' \,  \enorm{ u^\exact -w},
\end{align}
where $\Cdual' = \Cdual'(|\Omega|, d, \norm{u^\exact}{L^\infty(\Omega)}, M, n, R, p, f, \f, g, \g, \mu_0)> 0$. This shows the inequalities in~\eqref{eq2:dual:tbp}. The estimate~\eqref{eq2:dual:est} together with~\eqref{eq2:dual:tbp} yields~\eqref{eq2:dual:estimate}, where $\Cdual \!= \Cdual(|\Omega|, d, \norm{u^\exact}{L^\infty(\Omega)}, M, n, R, p, f, \f, g, \g, \mu_0)> 0$. This concludes the proof.
\end{proof}
\begin{proof}[Proof of Theorem~\ref{theorem:errorestimate}]
Since Lemma~\ref{lemma:seminorm} guarantees $\enorm{u^\exact_\coarse} \le \Cbnd$, we can apply Lemma~\ref{prop:energybound} and Lemma~\ref{prop:dual:energybound} to $w = u^\exact_\coarse$ to obtain that
\begin{subequations}\label{eq:aux:critical}
\begin{align}\label{eq1:critical}
 \prod{b(u^\exact) - b(u_\coarse^\exact)}{v}
 \le \CLip \, \enorm{u^\exact - u_\coarse^\exact} \enorm{v}
 \quad \text{for all } v \in H^1_0(\Omega)
\end{align}
as well as 
\begin{align}\label{eq2:critical}
 \enorm{\tilde z^\exact[u_\coarse^\exact] - z^\exact[u_\coarse^\exact]}
 \le \Cdual \, \enorm{u^\exact - u_\coarse^\exact}.
\end{align}
\end{subequations}
Combining these estimates with the error identity~\eqref{eq2:aux:error-identity}, we prove the error estimate
\begin{align*}
 | G(u^\exact) - G(u_\coarse^\exact) | &= | \sprod{u^\exact - u_\coarse^\exact}{ z^\exact[u_\coarse^\exact] - z_\coarse^\exact[u_\coarse^\exact]} + \prod{b(u^\exact) - b(u_\coarse^\exact)}{z^\exact[u_\coarse^\exact] - z_\coarse^\exact[u_\coarse^\exact]}
 \\& \quad \hspace{2mm}
 + \sprod{u^\exact - u_\coarse^\exact}{ \tilde z^\exact[u_\coarse^\exact] - z^\exact[u_\coarse^\exact]} + \prod{b(u^\exact) - b(u_\coarse^\exact)}{\tilde z^\exact[u_\coarse^\exact] - z^\exact[u_\coarse^\exact]} | \\
 &\stackrel{\eqref{eq:aux:critical}}{\le} \Cest \, \left[\enorm{u^\exact - u_\coarse^\exact} \enorm{z^\exact[u_\coarse^\exact] - z_\coarse^\exact[u_\coarse^\exact]} + \enorm{u^\exact - u_\coarse^\exact}^2 \right],
\end{align*}
where $\Cest = (1+\CLip) \max\{1, \Cdual\}$. This concludes the proof.
\end{proof}

The assumptions of Lemma~\ref{prop:energybound} (resp. Lemma~\ref{prop:dual:energybound}) also yield the validity of a C\'{e}a-type best approximation property for the discrete primal solution $u_\coarse^\exact \in \XX_\coarse$ (resp. for the discrete dual solutions $\tilde z_\coarse^\exact[w], z_\coarse^\exact[w]$ for any $w \in H^1_0(\Omega)$ with $\enorm{w} \le M < \infty$), even though the PDE operator $\AA$ from~\eqref{eq:op:bm} is \emph{not} Lipschitz continuous.
\begin{proposition}[C\'{e}a lemma for primal problem]\label{prop:cea}
Under the assumptions of Lemma~\ref{prop:energybound}, it holds that
\begin{align}\label{eq:cea}
\enorm{u^\exact - u_\coarse^\exact} \le \Ccea \, \min_{v_\coarse \in \XX_\coarse} \enorm{u^\exact - v_\coarse},\vspace*{-.5\baselineskip} 
\end{align}
where $\Ccea = \Ccea(|\Omega|, d, n, R, p, f, \f, \mu_0)$.
\end{proposition}
\begin{proof}
The Galerkin orthogonality reads
\begin{align}\label{eq:aux:coarse:go}
\sprod{u^\exact- u_\coarse^\exact}{ v_\coarse} + \prod{b(u^\exact)-b(u_\coarse^\exact)}{v_\coarse} = 0, \quad \text{ for all } v_\coarse \in \XX_\coarse.
\end{align}
Using~\ref{assump:mon} and the Galerkin orthogonality, we observe that
\begin{align*} 
\enorm{u^\exact-u_\coarse^\exact}^2 &\stackrel{\mathclap{\eqref{eq:bm:stronglymonotone}}}{\le}
\sprod{u^\exact-u_\coarse^\exact}{u^\exact-u_\coarse^\exact} + \prod{b(u^\exact) - b(u_\coarse^\exact)}{u^\exact-u_\coarse^\exact}\\
&\stackrel{\mathclap{\eqref{eq:aux:coarse:go}}}{=} \sprod{u^\exact-u_\coarse^\exact}{u^\exact-v_\coarse} + \prod{b(u^\exact) - b(u_\coarse^\exact)}{u^\exact-v_\coarse}\\
&\stackrel{\mathclap{\eqref{eq:critical:lipschitz}}}{\le} \Ccea\,\enorm{u^\exact-u_\coarse^\exact}\, \enorm{u^\exact-v_\coarse},
 \end{align*}
 where $\Ccea := 1 + \CLip$. This proves \eqref{eq:cea}, where the minimum is attained due to finite dimension of $\XX_\coarse$. 
\end{proof}

\begin{proposition}[C\'{e}a lemma for dual problems] \label{prop:cea:dual}
Let $w \in H^1_0(\Omega)$ with $\enorm{w} \le M < \infty$. Under the assumptions of Lemma~\ref{prop:dual:energybound}, it holds that
\begin{align}\label{eq:cea:theoretical:dual}
\enorm{\tilde z^\exact[w] - \tilde z_\coarse^\exact[w]} \le \Ccea \min_{v_\coarse \in \XX_\coarse} \enorm{ \tilde z^\exact[w] - v_\coarse},
\end{align}\vspace*{-\baselineskip} 
\begin{align}\label{eq:cea:dual}
\enorm{z^\exact[w] - z_\coarse^\exact[w]} \le \Ccea \min_{v_\coarse \in \XX_\coarse} \enorm{z^\exact[w] - v_\coarse},
\end{align} 
where $\Ccea = \Ccea(|\Omega|, d, \norm{u^\exact}{L^\infty(\Omega)}, M, n, R, p, f, \f, \mu_0)$.
\end{proposition}
\begin{proof}

We prove the statement for the practical dual problem. With minor modifications, the same argument also applies for the theoretical dual problem. We only need to show that the bilinear form of the practical dual problem is continuous and elliptic. Then, by standard theory for Lax--Milgram-type problems, this proves the C\'{e}a lemma~\eqref{eq:cea:dual}. To this end, we exploit~\ref{assump:mon} and obtain that
\begin{align*}
\enorm{v}^2 = \sprod{v}{v} \le \sprod{v}{v} + \prod{b'(w) v}{v} \quad \text{ for all } v \in H^1_0(\Omega),
\end{align*}
i.e., the bilinear form is elliptic with constant $1$.
In view of Remark~\ref{remark:reg12}, choose $t > 1$ arbitrarily for $d \in \{1,2\}$ and $t = 2^\ast$ and, hence, $t'' = 2^\ast/(2^\ast-2)$ for $d=3$. With~\ref{assump:ell} and~\ref{assump:poly}, we have that
\begin{align*}
\sprod{z}{v} + \prod{b'(w) z}{v} \le (1 + C \norm{b'(w)}{L^{t''}(\Omega)}) \enorm{z}\enorm{v}  \quad \text{ for all } v, z \in H^1_0(\Omega).
\end{align*}
 With~\eqref{eq2:dual:tbp} and $\enorm{u^\exact - w} \le \Cbnd + M$, we can finally bound
\begin{align*}
\norm{b'(w)}{L^{t''}(\Omega)} &\le \norm{b'(u^\exact)}{L^{t''}(\Omega)} + C \enorm{u^\exact - w} \le \norm{b'(u^\exact)}{L^{t''}(\Omega)} + C(\Cbnd + M).
\end{align*}
Combining the last two estimates, we prove continuity of the bilinear form with $\Ccea = \Ccea(|\Omega|, d, \norm{u^\exact}{L^\infty(\Omega)}, M, n, R, p, f, \f, \mu_0)$. This concludes the proof.
\end{proof}
\begin{remark}
If it is a priori guaranteed that $\norm{u^\exact_\coarse}{L^\infty(\Omega)} \le C< \infty$, then the proofs of Section~\ref{subsection:goalerrorestimate} simplify considerably and the use of~\ref{assump:poly} can be avoided. By Proposition~\ref{proposition:weakform:primal:Linf}, we infer 
\begin{align}\label{eq:bnd:helper}
\norm{u^\exact_\coarse - \tau (u^\exact-u^\exact_\coarse)}{L^\infty(\Omega)} < \infty \quad \text{ for all } \quad 0 \le \tau \le 1.
\end{align} 
To establish Lemma~\ref{prop:energybound}, recall $\B^\exact(u_\coarse^\exact)$ from~\eqref{eq:aux:taylor}. The observation~\eqref{eq:bnd:helper} together with the smoothness assumption~\ref{assump:car} yields that $\norm{\B^\exact(u_\coarse^\exact)}{L^\infty(\Omega)} < \infty$. Altogether, we obtain that
\begin{align*}
 \prod{b(u^\exact) - b(u_\coarse^\exact)}{v}
 \le \norm{\B^\exact(u_\coarse^\exact)}{L^\infty(\Omega)} \norm{u^\exact-u_\coarse^\exact}{L^2(\Omega)} \norm{v}{L^2(\Omega)}.
\end{align*}
Note that~\eqref{eq:bnd:helper} also establishes the crucial estimate~\eqref{eq2:dual:tbp} from Lemma~\ref{prop:dual:energybound} due to the local Lipschitz continuity from~\ref{assump:car}; see~\cite[Proposition 1]{hpz2015}. However, we stress that already for lowest-order FEM, the validity of a discrete maximum principle requires assumptions on the triangulation which are not imposed for~\ref{assump:poly} and usually not met for adaptive mesh refinement.
\qed
\end{remark}

\begin{remark}
Note that~\ref{assump:car} implies only that $b(x, \, \cdot \,)$ is locally Lipschitz. If we additionally assume global Lipschitz continuity, i.e., $L' := \sup_{x \in \Omega} \norm{b'(x,\cdot)}{L^\infty(\R)} < \infty$, then the strongly monotone operator $\AA \colon H^1_0(\Omega) \to H^{-1}(\Omega)$ from~\eqref{eq:op:bm} is also Lipschitz continuous with $L := \max\{ \mu_1, L' \}$. In particular, the problem~\eqref{eq:weakform:primal} fits into the framework of the main theorem on strongly monotone operators, and the proof of Lemma~\ref{prop:energybound} becomes trivial. The same applies to the proof of Lemma~\ref{prop:dual:energybound}, if $b'$ is globally Lipschitz continuous. \qed
\end{remark}

%% file: 03_algorithm.tex

\section{Goal-oriented adaptive algorithm and main results}
\label{section:mainresults}

\subsection{Mesh refinement}
\label{subsection:mesh-refinement}
From now on, let $\TT_0$ be a given conforming triangulation of $\Omega$. For mesh refinement, we employ newest vertex bisection (NVB); see~\cite{stevenson2008}. 
For each triangulation $\TT_\coarse$ and marked elements $\MM_\coarse \subseteq \TT_\coarse$, let $\TT_\fine := \refine(\TT_\coarse,\MM_\coarse)$ be the coarsest triangulation where all $T \in \MM_\coarse$ have been refined, i.e., $\MM_\coarse \subseteq \TT_\coarse \backslash \TT_\fine$. 
We write $\TT_\fine \in \T(\TT_\coarse)$, if $\TT_\fine$ results from $\TT_\coarse$ by finitely many steps of refinement.
To abbreviate notation, let $\T:=\T(\TT_0)$.

Throughout, each triangulation $\TT_\coarse \in \T$ is associated with the finite-dimensional FEM space $\XX_\coarse \subseteq H^1_0(\Omega)$ from the introduction, and, since we employ NVB, $\TT_\fine \in \T(\TT_\coarse)$ implies nestedness $\XX_\coarse \subseteq \XX_\fine$.

\subsection{\textsl{A~posteriori} error estimators}\label{subsection:axioms}
For $\TT_\coarse \in \T$, $v_\coarse \in \XX_\coarse$, and $w \in H^1_0(\Omega)$, let
\begin{align}
\begin{split}\label{eq:estimator:primal}
\eta_\coarse(T, v_\coarse)^2 &:= h_T^2 \,\norm{f + \div(\A \, \nabla v_\coarse - \f) - b(v_\coarse)}{L^2(T)}^2  \\
& \qquad+ h_T \, \norm{\llbracket (\A \, \nabla v_\coarse - \f ) \, \cdot \, \n \rrbracket}{L^2(\partial T \cap \Omega)}^2,
\end{split}
\end{align}
\begin{align}
\begin{split}\label{eq:estimator:dual}
\zeta_\coarse(w; T, v_\coarse)^2 &:= h_T^2 \,\norm{g + \div(\A \, \nabla v_\coarse - \g) - b'(w)(v_\coarse)}{L^2(T)}^2\\
&\qquad + h_T \, \norm{\llbracket (\A \, \nabla v_\coarse - \g ) \, \cdot \, \n \rrbracket}{L^2(\partial T \cap \Omega)}^2
\end{split}
\end{align}
be the local contributions of the standard residual error estimators, where $\llbracket \, \cdot \, \rrbracket$ denotes the jump across edges (for $d=2$) resp.\ faces (for $d=3$) and $\n$ denotes the outer unit normal vector. For $d=1$, these jumps vanish, i.e., $\llbracket \, \cdot \,\rrbracket = 0$. For $\UU_\coarse \subseteq \TT_\coarse$, let
\begin{equation*}
	\eta_\coarse(\UU_\coarse, v_\coarse)
	:=
	\Big( \sum_{T \in \,\UU_\coarse} \eta_\coarse(T, v_\coarse)^2 \Big)^{1/2}
	\quad \text{and} \quad
	\zeta_\coarse(w; \UU_\coarse, v_\coarse)
	:=
	\Big( \sum_{T \in \,\UU_\coarse} \zeta_\coarse(w; T, v_\coarse)^2 \Big)^{1/2}.
\end{equation*}
To abbreviate notation, let $\eta_\coarse(v_\coarse) := \eta_\coarse(\TT_\coarse, v_\coarse)$ and $\zeta_\coarse(w; v_\coarse) := \zeta_\coarse(w; \TT_\coarse, v_\coarse)$. Furthermore, we write, e.g., $\zeta_\coarse(\UU_\coarse, z^\exact_\coarse[w]) :=\zeta_\coarse(w; \, \UU_\coarse, z^\exact_\coarse[w])$, since $w$ is clear from the context.

The next result establishes that the error estimators~\eqref{eq:estimator:primal}--\eqref{eq:estimator:dual} satisfy the following slightly relaxed axioms of adaptivity from~\cite{axioms}. Compared to~\cite{axioms}, stability~\eqref{assumption:stab} is slightly relaxed and reduction~\eqref{assumption:red} is simplified due to the nestedness of the discrete spaces. Furthermore, we note that well-posedness of~\eqref{eq:estimator:primal}--\eqref{eq:estimator:dual} requires additional regularity assumptions on $\A$, $\f$, and $\g$ (as stated in Section~\ref{subsection:assump:diff} and~\ref{subsection:rhs}) so that the jump terms are well-defined. 
\begin{proposition}\label{proposition:axioms} Suppose~\ref{assump:rhs},~\ref{assump:ell},~\ref{assump:car},~\ref{assump:mon}, and~\ref{assump:compact}. Let $\TT_\coarse \in \T$ and $\TT_\fine \in \T(\TT_\coarse)$. Then, there hold the following properties:
\renewcommand{\theenumi}{{A\arabic{enumi}}}
\begin{enumerate}
	\bf
	\item\label{assumption:stab} stability: \rm 
	For all $M > 0$, there exists $\Cstab[M] >0$ such that for all $w \in H^1_0(\Omega)$, $v_\fine \in \XX_\fine$, and $v_\coarse \in \XX_\coarse$ with $\max\{ \enorm{w}, \enorm{v_\fine}, \enorm{v_\coarse} \} \le M$, it holds that 
	\begin{align*}
	\big| \eta_\fine(\TT_\fine \cap \TT_\coarse, v_\fine) - \eta_\coarse(\TT_\fine \cap \TT_\coarse, v_\coarse) \big|
	&\le \Cstab[M]\, \enorm{v_\fine - v_\coarse},\\
	\big| \zeta_\fine(w; \TT_\fine \cap \TT_\coarse, v_\fine) - \zeta_\coarse(w; \TT_\fine \cap \TT_\coarse, v_\coarse) \big|
	&\le \Cstab[M]  \, \enorm{v_\fine - v_\coarse}.
	\end{align*}
	\bf
	\item\label{assumption:red} reduction: \rm
With $0 < \qred := 2^{-1/(2d)}<1$, there holds that, for all $v_\coarse \in \XX_\coarse$ and all $w \in H^1_0(\Omega)$, 
	\begin{align*}
	\eta_\fine(\TT_\fine \backslash \TT_\coarse, v_\coarse) 
	&\le \qred \, \eta_\coarse(\TT_\coarse \backslash \TT_\fine, v_\coarse) 
	\quad \text{and} \quad \zeta_\fine(w; \TT_\fine \backslash \TT_\coarse, v_\coarse) \le \qred \, \zeta_\coarse(w; \TT_\coarse \backslash \TT_\fine, v_\coarse).
	\end{align*}
	\bf
	\item\label{assumption:rel} reliability: \rm
	For all $w \in H^1_0(\Omega)$, there exists $\Crel >0$ such that
	\begin{align*}
	\enorm{u^\exact - u_\coarse^\exact} 
	&\le \Crel \, \eta_\coarse(u_\coarse^\exact) \quad \text{ and } \quad\enorm{z^\exact[w] - z_\coarse^\exact[w]} 
	\le \Crel \, \zeta_\coarse(z_\coarse^\exact[w]).
	\end{align*} 
	\bf
	\item\label{assumption:drel} discrete reliability: \rm
	For all $w \in H^1_0(\Omega)$, there exists $\Cdrel >0$ such that
	\begin{align*}
	\enorm{u_\fine^{\exact} - u_\coarse^{\exact}} 
	&\le \Cdrel \, \eta_\coarse(\TT_\coarse \backslash \TT_\fine, u_\coarse^{\exact})	\quad \text{ and } \quad \enorm{z_\fine^{\exact}[w] - z_\coarse^{\exact}[w]}
	\le \Cdrel \, \zeta_\coarse(\TT_\coarse \backslash \TT_\fine, z_\coarse^{\exact}[w]). 
	\end{align*} 
\end{enumerate}
The constant $\Crel$ depends only on $d$,  $\mu_0$, and uniform shape regularity of the meshes $\TT_\coarse \in \T$. $\Cdrel$ depends additionally on the polynomial degree $m$, and $\Cstab[M]$ depends furthermore on $|\Omega|$, $M$, $n$, $R$, and $\A$.
\end{proposition}

\begin{remark}\label{remark:Cstab}
As far as the \emph{axioms of adaptivity}~\eqref{assumption:stab}--\eqref{assumption:drel} are concerned, we stress that only the constant $\Cstab[M]$ depends on $M > 0$. From Lemma~\ref{lemma:seminorm}, we know that $\enorm{v} \le \Cbnd$ for all $v \in \{u^\exact, u_\fine^\exact, u_\coarse^\exact, z^\exact[w], z^\exact_\fine[w], z^\exact_\coarse[w]\}$. Hence, for $w \in \{u^\exact, u_\fine^\exact, u_\coarse^\exact\}$, $v_\fine \in \{u_\fine^\exact, z_\fine^\exact[w]\}$, and $v_\coarse \in \{u_\coarse^\exact, z_\coarse^\exact[w]\}$, also the constant $\Cstab = \Cstab[\Cbnd]$ in~\eqref{assumption:stab} becomes generic.
\end{remark}%
\subsection{Goal-oriented adaptive algorithm}
\label{subsection:algo}

The following algorithm essentially coincides with that of~\cite{hpz2015}. Following~\cite{bip2020}, we adapt the marking strategy to mathematically guarantee optimal convergence rates.

\begin{algorithm}\label{algorithm:first}
{\bfseries Input:} Adaptivity parameters $0 \!<\! \theta \le 1$ and $\Cmark\! \ge\! 1$, initial mesh $\TT_0$.
\newline
{\bfseries Loop:} For all $\ell = 0,1,2,\dots$, perform the following steps~{\rm(i)--(v)}:
\begin{itemize}
\item[\rm(i)] Compute the discrete solutions $u_\ell^\exact, z_\ell^\exact[u_\ell^\exact] \in \XX_\ell$ 
to~\eqref{eq:weakform:primal:discrete} resp.~\eqref{eq:weakform:dual:discrete:exact}.
\item[\rm(ii)] Compute the refinement indicators $\eta_\ell(T,u_\ell^\exact)$ and $\zeta_\ell(T, z_\ell^\exact[u_\ell^\exact])$ for all $T \in \TT_\ell$.
\item[\rm(iii)] Determine sets $\overline\MM_\ell^{u}, \overline\MM_\ell^{uz} \subseteq \TT_\ell$ of up to the multiplicative constant $\Cmark$ minimal cardinality such that
\begin{subequations}\label{eq:first:doerfler}
\begin{align}\label{eq:first:doerfler:u}
 \theta \, \eta_\ell(u_\ell^\exact)^2 
 &\le \eta_\ell(\overline\MM_\ell^{u},u_\ell^\exact)^2,
 \\ \label{eq:first:doerfler:uz}
 \theta \, \big[ \, \eta_\ell(u_\ell^\exact)^2 + \zeta_\ell(u^\exact_\ell; z_\ell^\exact[u_\ell^\exact])^2 \, \big] 
 &\le \big[ \, \eta_\ell(\overline\MM_\ell^{uz},u_\ell^\exact)^2 +  \zeta_\ell(u^\exact_\ell; \overline\MM_\ell^{uz},z_\ell^\exact[u_\ell^\exact])^2 \, \big].
\end{align} 
\end{subequations}
\item[\rm(iv)] Select $\MM_\ell^{u} \subseteq \overline\MM_\ell^{u}$ and $\MM_\ell^{uz} \subseteq \overline\MM_\ell^{uz}$ with $\#\MM_\ell^{u} =  \#\MM_\ell^{uz} =  \min\{ \, \#\overline\MM_\ell^{u} \,,\, \#\overline\MM_\ell^{uz} \, \}$.
\item[\rm(v)] Define $\MM_\ell := \MM_\ell^{u} \cup \MM_\ell^{uz}$ and generate $\TT_{\ell+1} := \refine(\TT_\ell, \MM_\ell)$.
\end{itemize}
{\bfseries Output:} Sequence of triangulations $\TT_\ell$ with corresponding discrete solutions $u_\ell^\exact$ and $z_\ell^\exact[u_\ell^\exact]$ as well as error estimators $\eta_\ell(u_\ell^\exact)$ and $\zeta_\ell(u^\exact_\ell; z_\ell^\exact[u_\ell^\exact])$.
\end{algorithm}

\subsection{Main results}\label{subsection:mainres}

In the following, we give formal statements of our main results on Algorithm~\ref{algorithm:first}. The proofs are postponed to Section~\ref{section:proofs} below. Our first result states that Algorithm~\ref{algorithm:first} indeed relies on reliable \textsl{a~posteriori} error control for the goal error and guarantees plain convergence. 

\begin{proposition}
\label{prop:convergenceA}
Suppose~\ref{assump:rhs},~\ref{assump:ell},~\ref{assump:car},~\ref{assump:mon}, and~\ref{assump:poly}. 
Then, there hold the following statements~{\rm(i)}--{\rm(ii)}:
		
	{\rm(i)} There exists a constant $\Crel^\prime > 0$ such that
	\begin{equation}\label{eq:goal-upper-bound}
		\big| G(u^{\exact}) - G(u_\coarse^{\exact}) \big|
		\le
		\Crel^\prime \, \eta_\coarse(u_\coarse^{\exact}) \, \big[ \, \eta_\coarse(u_\coarse^{\exact})^2 + \zeta_\coarse(z_\coarse^{\exact}[u_\coarse^{\exact}])^2 \, \big]^{1/2}
		\quad \text{for all } \TT_\coarse \in \T.
	\end{equation}

	{\rm(ii)} For all $0 < \theta \leq 1$ and $1 < \Cmark \leq \infty$, Algorithm~\ref{algorithm:first} leads to convergence
	\begin{equation}
	\label{eq:plain-convergenceA}
		| G(u^{\exact}) - G(u_\ell^{\exact}) | 
		\le
		\Crel^\prime \eta_\ell(u_\ell^{\exact}) \, \big[ \,
			\eta_\ell(u_\ell^{\exact})^2 + \zeta_\ell(z_\ell^{\exact}[u_\ell^{\exact}])^2 \,
		\big]^{1/2}
		\longrightarrow 0
		\quad\text{as }
		\ell \to \infty
	\end{equation}
where $\Crel^\prime = \Crel^\prime(|\Omega|, \T, d, n, R, p, f, \f, g, \g, \mu_0)$.
\end{proposition}

We stress that~\eqref{eq:goal-upper-bound} is an immediate consequence of the goal error estimate~\eqref{eq:critical} from Theorem~\ref{theorem:errorestimate} and reliability~\eqref{assumption:rel}, i.e.,
\begin{align*}
\Cest^{-1} \, | G(u^\exact) - G(u_\coarse^\exact) |
 &\le \left[\enorm{u^\exact - u_\coarse^\exact} \enorm{z^\exact[u_\coarse^\exact] - z_\coarse^\exact[u_\coarse^\exact]} + \enorm{u^\exact - u_\coarse^\exact}^2 \right]
 \\&
 \le \Crel^2 \big[ \eta_\coarse(u_\coarse^\exact)\zeta_\coarse(z_\coarse^\exact[u_\coarse^\exact]) + \eta_\coarse(u_\coarse^\exact)^2 \big]
 \\&
 \le \sqrt{2} \, \Crel^2 \eta_\coarse(u_\coarse^\exact) \big[\zeta_\coarse(z_\coarse^\exact[u_\coarse^\exact])^2 + \eta_\coarse(u_\coarse^\exact)^2 \big]^{1/2}.
\end{align*}
Consequently, only the convergence~\eqref{eq:plain-convergenceA} of Proposition~\ref{prop:convergenceA}(ii) has to be proven. Replacing the assumption~\ref{assump:poly} on the nonlinearity by the stronger assumption~\ref{assump:compact}, we even get linear convergence, which improves Proposition~\ref{prop:convergenceA}{\rm(ii)}.

\begin{theorem}\label{theorem:main}
Suppose~\ref{assump:rhs},~\ref{assump:ell},~\ref{assump:car},~\ref{assump:mon}, and~\ref{assump:compact}. Then, for all $0 < \theta \le 1$ and $1 \le \Cmark \le \infty$, there exists $\ell_0 \in \N_0$, $\Clin > 0$, and $0 < \qlin < 1$ such that Algorithm~\ref{algorithm:first} guarantees that, for all $\ell, k \in \N_0$ with $k \ge \ell \ge \ell_0$,
\begin{align}\label{eq:linear}
 \eta_k(u^\exact_k) \, \big[ \eta_k(u^\exact_k)^2 + \zeta_k(z^\exact_k[u^\exact_k])^2 \, \big]^{1/2}
 \le \Clin \qlin^{k-\ell} \, \eta_\ell(u^\exact_\ell) \, \big[ \eta_\ell(u^\exact_\ell)^2 + \zeta_\ell(z^\exact_\ell[u^\exact_\ell])^2 \, \big]^{1/2}.
\end{align}
The constants $\Clin$ and $\qlin$ as well as the index $\ell_0$ depend only on~$|\Omega|$, $\T$, $d$, $m$, $\theta$, $n$, $R$, $p$, $f$, $\f$, $g$, $\g$, $\mu_0$, and $\A$.
\end{theorem}
To formulate our main result on optimal convergence rates, we need some additional notation.
For $N \in \N_0$, let $\T_N:=\set{\TT\in\T}{\#\TT-\#\TT_0\le N}$ denote the (finite) set of all refinements of $\TT_0$ which have at most $N$ elements more
than $\TT_0$. 
For $s,t>0$, we define
\begin{align*}
	\norm{u^\exact}{\mathbb{A}_s} 
	&:= \sup_{N\in\N_0} \Big((N+1)^s \min_{\TT_\coarse\in\T_N} \eta_\coarse(u^\exact_\coarse) \Big) \in \R_{\geq 0} \cup \{ \infty \},\\
	\norm{z^\exact[u^\exact]}{\mathbb{A}_t} 
	&:= \sup_{N\in\N_0} \Big((N+1)^t\min_{\TT_\coarse\in\T_N} \zeta_\coarse(z^\exact_\coarse[u^\exact]) \Big) \in \R_{\geq 0} \cup \{ \infty \}.
\end{align*}
In explicit terms, e.g., $\norm{u^\exact}{\mathbb{A}_s} < \infty$ means that an algebraic convergence rate $\OO(N^{-s})$ for the error estimator $\eta_\ell$ is possible, if the optimal triangulations are chosen.

In comparison to~\cite{hpz2015} or~\cite{xhym2021}, our proof of Theorem~\ref{theorem:main} avoids any $L^\infty$-bounds on the discrete solutions as well as the assumption that the initial mesh is sufficiently fine. Moreover, in contrast to~\cite{xhym2021}, which proves linear convergence for the marking strategy suggested in~\cite{hpz2015} (and a multilevel correction step), we even prove optimal convergence rates without assuming Lipschitz continuity for the primal and dual operators.

\begin{theorem}\label{theorem:main2}
Suppose~\ref{assump:rhs},~\ref{assump:ell},~\ref{assump:car},~\ref{assump:mon}, and~\ref{assump:compact}. Let $s,t > 0$ with $\norm{u^\exact}{\mathbb{A}_s} + \norm{z^\exact[u^\exact]}{\mathbb{A}_t} < \infty$. Then, for all $0 < \theta < \theta_{\rm opt} := (1+\Cstab^2\Cdrel^2)^{-1}$ and $1 \le \Cmark < \infty$, there holds the following: With the index $\ell_0 \in \N_0$ from Theorem~\ref{theorem:main}, there exists $\Copt > 0$ such that Algorithm~\ref{algorithm:first} guarantees that, for 
all $\ell \in \N_0$ with $\ell \geq \ell_0$,
\begin{align}\label{eq:optimal}
	\eta_\ell(u^\exact_\ell)\big[ \, \eta_\ell(u^\exact_\ell)^2 + \zeta_\ell(z^\exact_\ell[u^\exact_\ell])^2 \, \big]^{1/2}
	\le 
	\Copt \, \norm{u^\exact}{\mathbb{A}_s}(\norm{u^\exact}{\mathbb{A}_s}+\norm{z^\exact}{\mathbb{A}_t})\,(\#\TT_\ell-\#\TT_0)^{-\alpha},
\end{align}
where $\alpha := \min\{2s,s+t\}$.
The constant $\Copt$ depends only on $|\Omega|$, $\T$, $d$, $m$, $\Cnvb$, $\Cmark$, $\ell_0$, $\theta$, $n$, $R$, $p$, $f$, $\f$, $g$, $\g$, $\mu_0$, and $\A$.
\end{theorem}

\begin{remark}
Compared to the treatment of linear problems in~\cite{fpz2016}, the marking strategy  considers the combined error estimator due to the structure from~\eqref{eq:aux:error-estimate}. In addition, the proofs of the essential quasi-orthogonalities are more involved both in the semilinear primal setting as well as for the combined error estimator --- which is also a key ingredient of the analysis in~\cite{bip2020}, but note that this is a slight modification of the marking strategies of~\cite{hpz2015, xhym2021} that allows us to prove convergence rates. 
\end{remark}
\begin{remark}\label{rem:afemplus}
With the estimate $\eta_\ell(u^\exact_\ell) [ \eta_\ell(u^\exact_\ell)^2 + \zeta_\ell(z_\ell^\exact[u^\exact_\ell])^2]^{1/2} \le \eta_\ell(u^\exact_\ell)^2 + \zeta_\ell(z_\ell^\exact[u^\exact_\ell])^2$, one can also consider Algorithm~\ref{algorithm:first} with $\MM_\ell := \overline\MM_\ell^{uz}$, which then takes the form of the standard AFEM algorithm (see, e.g.,~\cite{axioms}) for the product space estimator. Then, Theorem~\ref{theorem:main} and~\ref{theorem:main2} hold accordingly with the product replaced by the square sum and $\alpha = \min \{ 2s, 2t \}$, which is slightly worse than the rate $\alpha$ from Theorem~\ref{theorem:main2}. We refer to~\cite{bip2020} for details (in a different, but structurally similar setting).
\end{remark}

%% file: 04_proofs.tex

\section{Proofs}
\label{section:proofs}

\noindent
In this section, we give the proofs of Proposition~\ref{proposition:axioms} and \ref{prop:convergenceA} as well as Theorem~\ref{theorem:main} and~\ref{theorem:main2}.

\subsection{Axioms of Adaptivity}
\label{subsection:proofaxioms}

In this section, we sketch the proof of Proposition~\ref{proposition:axioms} and verify that the residual error estimators from Section~\ref{subsection:axioms} satisfy the (relaxed) axioms of adaptivity \eqref{assumption:stab}--\eqref{assumption:drel} from~\cite{axioms}. As usual for nonlinear problems, only the verification of stability ~\eqref{assumption:stab} requires new ideas, while ~\eqref{assumption:red}--\eqref{assumption:drel} follow from standard arguments. For a triangulation $\TT_\fine \in \T$ and an element $T \in \TT_\fine$, let $\mathcal{E}(T)$ be the set of its facets (i.e., nodes for $d=1$, edges for $d=2$, and faces for $d=3$, respectively). Moreover, let
\begin{align}\label{eq:aux:patches}
\Omega_\fine[T] := \bigcup\set{T' \in \TT_\fine}{T \cap T' \neq \emptyset}
\end{align}
denote the usual element patch. Recall that \ref{assump:rhs} ensures that the error estimators~\eqref{eq:estimator:primal}--\eqref{eq:estimator:dual} are well-defined. To abbreviate notation, we define the primal and dual residuals
\begin{subequations}\label{appendix:res}
\begin{align}
\mathfrak{R}(v_\coarse)&:=f + \div(\A \, \nabla v_\coarse - \f) - b(v_\coarse),
\label{appendix:resprimal} \\
\mathfrak{R}^\ast(w; v_\coarse) &:= g + \div(\A \nabla v_\coarse - \g) -b'(w) v_\coarse \label{appendix:resdual}
\end{align}
\end{subequations}
for all $v_\coarse \in \XX_\coarse$ and $w \in H^1_0(\Omega)$. We stress that we do not explicitly state the dependence of the constants on the $\gamma$-shape regularity constant.

To prove stability~\eqref{assumption:stab}, we need the following auxiliary result:
\begin{lemma}\label{lemma:a1:vital} Suppose~\ref{assump:ell},~\ref{assump:car}, and~\ref{assump:compact}. Let $M > 0$ and $v, w \in H^1_0(\Omega)$ with $\max\{\enorm{v}, \enorm{w}\} \le M$. Then, it holds that
\begin{align}
\begin{split}
\norm{b(v) - b(w)}{L^2(\Omega)} &\le C[M]\, \enorm{ v-w}
\end{split}
\end{align}
with $C[M] = C(|\Omega|, \T, d, M, n, R, \mu_0)$.
\end{lemma}
\begin{proof}
Similarly to the Taylor expansion in \eqref{eq:primal:taylorexp}, it holds that
\begin{align}\label{appendix:a1:helper}
\begin{split}
b(v) &= \sum_{k=0}^{n-1} b^{(k)}(w) \, \frac{(v - w )^k}{k!} + \frac{ \, ( v- w )^{n}}{(n-1)!} \, \int_0^1 ( 1 - \tau )^{n-1} \, b^{(n)}( w + (v-w) \, \tau )  \d{\tau}.
\end{split}
\end{align} 
This yields that
\begin{align*}
 \norm{b(v) \!-\! b(w)}{L^2(\Omega)}\! \lesssim \!\Big|\!\Big|\!\sum_{k=1}^{n-1} {b^{(k)}(w)}  ( v\!-\!w )^k + ( v\!-\!w )^{n}\!\int_0^1 \!( 1\! - \!\tau )^{n-1}\, b^{(n)}( w\!+\!(v\!-\!w)  \tau )  \d{\tau}  \Big|\!\Big|_{L^2(\Omega)}\! .
\end{align*}
Recall the generalized H\"{o}lder inequality
\begin{align*}
\norm{\varphi \psi}{L^2(\Omega)} \le \norm{\varphi}{L^{2\rho'}(\Omega)} \,\norm{\psi}{L^{{2\rho}}(\Omega)}, \quad \text{ where } \quad 1/2 = 1/2 \, (1/\rho + 1/{\rho'}).
\end{align*}
Recall that $\norm{v^k}{L^\rho(\Omega)} = \norm{v}{L^{k\rho}(\Omega)}^k$. For any $k=1, \ldots, n-1$ and $1 < \rho < \infty$, it holds that
\begin{align*}
\norm{b^{(k)}(w) (v\!-\!w)^k}{L^2(\Omega)} \!\le \! \norm{b^{(k)}(w)}{L^{2\rho'}(\Omega)}\norm{v\!-\!w}{L^{2k \rho}(\Omega)}^{k} \!\stackrel{\mathclap{\ref{assump:compact}}}{\lesssim} \!(1\!+\!\norm{w}{L^{2(n\!-\!k)\rho'}(\Omega)}^{n-k}) \norm{v\!-\!w}{L^{2k\rho}(\Omega)}^k.
\end{align*}
For $d=1,2$, both norms can be estimated by the corresponding energy norm. For $d=3$, let $\rho=n/k$ and hence $\rho' = n/(n-k) > 1$. Note that $2(n-k)\rho' = 2n \le 6 = 2^\ast$ as well as $2k\rho = 2n \le 6 = 2^\ast$ by virtue of~\ref{assump:compact}. For the remainder term in~\eqref{appendix:a1:helper},~\ref{assump:compact} yields that $|b^{(n)}(\widetilde{w})| \lesssim 1$ for all $\widetilde{w} \in H^1_0(\Omega)$ and thus the integral is bounded by a constant. Altogether, this guarantees that 
\begin{align*}
\norm{b(v) \!-\! b(w)}{L^2(\Omega)} \!&\lesssim \!\sum_{k=1}^{n-1}\! \norm{ b^{(k)}(w)(v\!-\! w)^k}{L^{2}(\Omega)}\! +\! \norm{  (v\!-\!w)^n }{L^{2}(\Omega)}\! \lesssim \!\sum_{k=1}^{n} (1\!+\!\enorm{w}^{n-k} )\, \enorm{ v\!-\!w}^k \!,
\end{align*}
where the hidden constant depends only on $\Omega$, $\T$, $d$, $M$, $n$, and $R$ from~\ref{assump:compact}, and $\mu_0$. Note that $\enorm{v-w}^k \lesssim \enorm{v-w}$, where the hidden constant depends only on $M$. This concludes the proof.
\end{proof}
With Lemma~\ref{lemma:a1:vital} at hand, stability~\eqref{assumption:stab} follows as for a linear model problem~\cite{ckns2008}.
\begin{proof}[Proof of stability~\eqref{assumption:stab} for primal problem]
With the primal residual $\mathfrak{R}(v_\coarse)$ from~\eqref{appendix:resprimal}, the refinement indicators read
\begin{align*}
\eta_\coarse(T, v_\coarse)^2 = h_T^2 \,\norm{\mathfrak{R}(v_\coarse)}{L^2(T)}^2 + h_T \, \norm{\llbracket (\A \, \nabla v_\coarse + \f ) \, \cdot \, \n \rrbracket}{L^2(\partial T \cap \Omega)}^2.
\end{align*}
Define $\delta_\fine := v_\fine - v_\coarse \in \XX_\fine$ and $\mathfrak{D}(\delta_\fine) :=  \div(\A \, \nabla \delta_\fine) + b(v_\coarse) - b(v_\fine)$. Observe that
\begin{align*}
\begin{split}
\mathfrak{R}(v_\fine) \!&= \! \left[f + \div(\A \nabla v_\coarse - \f )- b(v_\coarse )  \right] + \left[\div(\A  \nabla \delta_{\fine}) + b(v_\coarse) - b(v_\fine) \right]\! =\! \mathfrak{R}(v_\coarse) + \mathfrak{D}(\delta_\fine).
\end{split}
\end{align*}
Elementary calculus proves that 
\begin{align}
\begin{split}\label{eq:aux:a1:est}
\hspace{-2mm}\eta_\fine(T, v_\fine) \!&= \!\left( h_T^2 \norm{\mathfrak{R}(v_\coarse) + \mathfrak{D}(\delta_\fine)}{L^2(T)}^2 + h_T \norm{\llbracket (\A \, \nabla ( v_\coarse + \delta_\fine) + \f ) \, \cdot \, \n \rrbracket}{L^2(\partial T \cap \Omega)}^2 \right)^{1/2}  \\
& \le \eta_\fine(T, v_\coarse) + h_T \norm{\mathfrak{D}(\delta_\fine)}{L^2(T)} + h_T^{1/2} \, \norm{\llbracket \A \, \nabla  \delta_\fine \, \cdot \, \n \rrbracket}{L^2(\partial T \cap \Omega)}.
\end{split}
\end{align}
Recalling the definition of $ \mathfrak{D}( \delta_\fine)$, we see that
\begin{align}\label{eq:aux:a1:sum}
\norm{ \mathfrak{D}( \delta_\fine)}{L^2(T)} \le \norm{ \div(\A\, \nabla  \delta_\fine) }{L^2(T)} + \norm{b(v_\coarse) - b(v_\fine)}{L^2(T)}.
\end{align}
For the first term in~\eqref{eq:aux:a1:sum}, we use the product rule and an inverse inequality to see that
\begin{align} \label{eq:aux:a1:summand1}
\begin{split}
\norm{ \div(\A\, \nabla  \delta_\fine) }{L^2(T)} &\le \norm{ (\div \A )  \cdot \nabla  \delta_\fine}{L^2(T)} + \norm{\A : \operatorname{D}^2 \!\delta_\fine }{L^2(T)} \\
& \lesssim \left( \norm{\div \A}{L^\infty(T)} +  h_T^{-1} \norm{\A}{L^\infty(T)} \right) \, \norm{\nabla  \delta_\fine}{L^2(T)},
\end{split}
\end{align}
where $:$ denotes the Frobenius scalar product on $\R^{d \times d}$ and $\operatorname{D}^2 \! \delta_\fine$ is the Hessian of $ \delta_\fine$. 
The jump term in~\eqref{eq:aux:a1:est} can be estimated by a discrete trace inequality:
\begin{align}\label{eq:aux:a1:jump}
\norm{\llbracket \A \, \nabla  \delta_\fine \, \cdot \, \n \rrbracket}{L^2(\partial T \cap \Omega)} &\lesssim  h_T^{-1/2} \, \norm{\A}{L^\infty(\Omega_\fine[T])} \, \norm{\nabla  \delta_\fine}{L^2(\Omega_\fine[T])}.
\end{align}
Collecting~\eqref{eq:aux:a1:est}--\eqref{eq:aux:a1:jump}, we obtain that
\begin{align*}
\hspace{-5mm}| \eta_\fine(T, v_\fine) - \eta_\fine(T, v_\coarse)| &\lesssim h_T \left[ \norm{\div \A}{L^\infty(T)} + h_T^{-1} \norm{\A}{L^\infty(\Omega_\fine[T])}\right] \,  \norm{\nabla  \delta_\fine}{L^2(\Omega_\fine[T])} \\
& \quad + h_T \, \norm{b(v_\coarse) - b(v_\fine)}{L^2(T)} \\
&\lesssim \big[|\Omega|^{1/d} \, \max_{T' \in \TT_0} \norm{\div \A}{L^\infty(T')} + \norm{\A}{L^\infty(\Omega)} \big] \norm{\nabla  \delta_\fine}{L^2(\Omega_\fine[T])} \\
& \quad + |\Omega|^{1/d} \norm{b(v_\coarse) - b(v_\fine)}{L^2(T)},
\end{align*}
where the hidden constant depends only on the shape regularity of $\TT_\fine$, and the polynomial degree $m$ of the ansatz spaces. Together with Lemma~\ref{lemma:a1:vital}, this yields that
\begin{align*}
\left| \eta_\fine(\TT_\fine \cap \TT_\coarse, v_\fine) - \eta_\fine(\TT_\fine \cap \TT_\coarse, v_\coarse) \right| & \le  \big( \!\!\!\!\!\sum_{T \in\, \TT_\fine \cap \TT_\coarse}  \left| \eta_\fine(T, v_\fine) - \eta_\fine(T, v_\coarse) \right|^2 \big)^{1/2}\\
&\hspace{-35mm}\lesssim \big(\!\!\!\!\! \sum_{T \in\, \TT_\fine \cap \TT_\coarse}  \norm{\nabla  \delta_\fine}{L^2(\Omega_\fine[T])}^2 + \norm{b(v_\coarse) - b(v_\fine)}{L^2(\Omega)}^2 \big)^{1/2} \\
&\hspace{-35mm}\lesssim \big( \norm{\nabla  \delta_\fine}{L^2(\Omega)}^2 + \norm{b(v_\coarse) - b(v_\fine)}{L^2(\Omega)}^2 \big)^{1/2}  \lesssim \enorm{v_\fine - v_\coarse}.
\end{align*}
The hidden constant depends only on $|\Omega|$, the shape regularity of $\TT_\fine$, $d$, $m$, $M$, $n$, $R$, $\mu_0$, and $\A$. Note that for any non-refined element $T \in \TT_\fine \cap \TT_\coarse$, it holds that $\eta_\fine(T, v_\coarse) = \eta_\coarse(T, v_\coarse)$. This concludes the proof. 
\end{proof}
\begin{proof}[Proof of stability~\eqref{assumption:stab} for dual problem]
With the dual residual $\mathfrak{R}^\ast(w; v_\coarse)$ from~\eqref{appendix:resdual}, the refinement indicators read
\begin{align*}
\zeta_\coarse(w; T, v_\coarse)^2 &= h_T^2 \,\norm{\mathfrak{R}^\ast(w; v_\coarse)}{L^2(T)}^2  + h_T \, \norm{\llbracket (\A \, \nabla v_\coarse - \g ) \, \cdot \, \n \rrbracket}{L^2(\partial T \cap \Omega)}^2.
\end{align*}
We define $\delta_\fine := v_\fine - v_\coarse \in \XX_\fine$ and $\mathfrak{D}^\ast(\delta_\fine) :=  \div(\A \, \nabla \delta_\fine) - b'(w) \delta_\fine$. Observe that similar arguments as for the proof of stability~\eqref{assumption:stab} of the primal problem lead to
\begin{align*}
\mathfrak{R}^\ast(w; v_\fine) = \mathfrak{R}^\ast(w; v_\coarse) + \mathfrak{D}^\ast(\delta_\fine)
\end{align*} 
and, hence,
\begin{align*}
\zeta_\fine(w; T, v_\fine)\! \le \zeta_\fine(w; T, v_\coarse) + h_T \norm{\mathfrak{D}^\ast(\delta_\fine)}{L^2(T)} + h_T^{1/2} \, \norm{\llbracket \A \, \nabla \delta_\fine \, \cdot \, \n \rrbracket}{L^2(\partial T \cap \Omega)}.
\end{align*}
Here, we only estimate the term $\norm{b'(w)\delta_\fine}{L^2(\Omega)}$, since the other terms follow from the arguments provided for the primal problem. To this end, choose $ 2 < \rho < \infty$ arbitrarily if $d\in \{1,2\}$. If $d = 3$, let $\rho = 3$ and, hence, $\rho' = 3/2$. Assumption~\ref{assump:compact} guarantees that the Sobolev embedding~\eqref{eq:embedding} holds with $r = 2\rho$ and $r = 2(n-1)\rho'$ simultaneously. Therefore, we obtain that
\begin{align*}
\norm{b'(w) \delta_\fine}{L^2(\Omega)} &\le \norm{b'(w)}{L^{2\rho'}(\Omega)} \norm{\delta_\fine}{L^{2\rho}(\Omega)} \lesssim ( 1 + \norm{w^{n-1}}{L^{2\rho'}(\Omega)}) \norm{\delta_\fine}{L^{2\rho}(\Omega)} \\
&\hspace{-20mm} = ( 1 + \norm{w}{L^{2(n-1)\rho'}(\Omega)}^{n-1}) \norm{\delta_\fine}{L^{2\rho}(\Omega)}  \lesssim ( 1 + \enorm{w}^{n-1}) \enorm{\delta_\fine} \lesssim \enorm{\delta_\fine}.
\end{align*}
Arguing as for the primal problem, we see that
\begin{align*}
\hspace{-5mm}\left| \zeta_\fine(w; \,\TT_\fine \cap \TT_\coarse, v_\fine) - \zeta_\fine(w; \,\TT_\fine \cap \TT_\coarse, v_\coarse) \right| &\\
&\hspace{-55mm}\lesssim \Big( \sum_{T \in\, \TT_\fine \cap \TT_\coarse}  \norm{\nabla \delta_\fine}{L^2(\Omega_\fine[T])}^2 + \norm{b'(w) (v_\coarse -  v_\fine)}{L^2(\Omega)}^2 \Big)^{1/2} \lesssim \, \enorm{v_\fine - v_\coarse}.
\end{align*}
The hidden constant depends only on $|\Omega|$, the shape regularity of $\TT_\fine$, $d$, the polynomial degree $m$ of the ansatz spaces, $M$, $n$, $R$, $\mu_0$, and $\A$. This concludes the proof. 
\end{proof}
The proof of reduction~\eqref{assumption:red} resembles the linear case in~\cite{ckns2008}.
\begin{proof}[Proof of reduction~\eqref{assumption:red}]
For $T \in \TT_\coarse\backslash \TT_\fine$, let $\TT_\fine \negmedspace\mid_T := \set{T' \in \TT_\fine}{T' \subseteq T}$ denote the set of its children. Note that NVB guarantees that
\begin{align}\label{eq:a2:bisection}
h_{T'} \le 2^{-1/d} h_T = 2^{-1/d} |T|^{1/d} \quad \text{ for all } T' \in \TT_\fine \negmedspace\mid_T\!.
\end{align}
Recall that
\begin{align*}
\eta_\fine(T', v_\coarse)^2 = h_{T'}^2 \,\norm{\mathfrak{R}(v_\coarse)}{L^2(T')}^2 + h_{T'} \, \norm{\llbracket (\A \, \nabla v_\coarse + \f ) \, \cdot \, \n \rrbracket}{L^2(\partial T' \cap \Omega)}^2.
\end{align*}
Applying the bisection estimate~\eqref{eq:a2:bisection}, we obtain that
\begin{align*}
\eta_\fine(\TT_\fine \backslash \TT_\coarse, v_\coarse)^2 &= \!\!\!\! \sum_{T' \in \TT_\fine \backslash \TT_\coarse}  \eta_\fine(T', v_\coarse)^2 = \!\!\!\! \sum_{T \in \TT_\coarse \backslash \TT_\fine} \, \, \, \,\sum_{\mathclap{T' \in \TT_\fine \mid_T}} \hspace{2mm}\eta_\fine(T', v_\coarse)^2 \\
&\hspace{-20mm}=\!\!\!\! \sum_{T \in \TT_\coarse \backslash \TT_\fine} \, \, \, \,\sum_{\mathclap{T' \in \TT_\fine \mid_T}}\hspace{2mm} \left( h_{T'}^2 \,\norm{\mathfrak{R}(v_\coarse)}{L^2(T')}^2 + h_{T'} \, \norm{\llbracket (\A \, \nabla v_\coarse + \f ) \, \cdot \, \n \rrbracket}{L^2(\partial T' \cap \Omega)}^2 \right).
\end{align*}
For the first term, it holds that 
\begin{align*}
\sum_{\mathclap{T' \in \TT_\fine \mid_T}}\hspace{2mm} h^2_{T'} \,\norm{\mathfrak{R}(v_\coarse)}{L^2(T')}^2 \le 2^{-2/d} \,h^2_{T} \,\norm{\mathfrak{R}(v_\coarse)}{L^2(T)}^2.
\end{align*}
For the second term, note that $v_\coarse \in \XX_\coarse$ is a coarse-mesh function and, hence, smooth in the interior of $T \in \TT_\coarse$. Hence, all jumps in the interior of $T \in \TT_\coarse$ vanish. This leads to 
\begin{align*}
\sum_{\mathclap{T' \in \TT_\fine \mid_T}}\hspace{2mm} h_{T'}\,\norm{\llbracket (\A \, \nabla v_\coarse + \f ) \, \cdot \, \n \rrbracket}{L^2(\partial T' \cap \Omega)}^2 &= \sum_{\mathclap{T' \in \TT_\fine \mid_T}}\hspace{2mm}h_{T'}\, \norm{\llbracket (\A \, \nabla v_\coarse + \f ) \, \cdot \, \n \rrbracket}{L^2(\partial T' \cap \partial T \cap\Omega)}^2 \\
&\hspace{-70mm}\le 2^{-1/d} \, h_T  \sum_{\mathclap{T' \in \TT_\fine \mid_T}} \norm{\llbracket (\A \, \nabla v_\coarse + \f )  \cdot  \n \rrbracket}{L^2(\partial T' \cap \partial T \cap\Omega)}^2 = 2^{-1/d} \, h_T \norm{\llbracket (\A \, \nabla v_\coarse + \f ) \cdot \n \rrbracket}{L^2(\partial T \cap\Omega)}^2.
\end{align*}
Altogether, we conclude reduction~\eqref{assumption:red} for the primal estimator
\begin{align*}
\eta_\fine(\TT_\fine \backslash \TT_\coarse, v_\coarse)^2 &\! \le \!2^{-2/d} \, \, \sum_{\mathclap{T \in \TT_\coarse \backslash \TT_\fine}} \, \,  h_T^2  \norm{\mathfrak{R}(v_\coarse)}{L^2(T)}^2  + 2^{-1/d}\,\sum_{\mathclap{T \in \TT_\coarse \backslash \TT_\fine}} \, h_{T} \, \norm{\llbracket (\A \, \nabla v_\coarse \!+ \! \f ) \cdot \n \rrbracket}{L^2(\partial T \cap \Omega)}^2 \notag\\
& \le 2^{-1/d} \, \eta_\coarse(\TT_\coarse \backslash \TT_\fine, v_\coarse)^2.
\end{align*}
 The same arguments apply for the dual estimator.
\end{proof}

The verification of~\eqref{assumption:rel}--\eqref{assumption:drel} follows standard arguments as found, e.g., in \cite{v2013}.
\begin{proof}[Sketch of proof of reliability~\eqref{assumption:rel}]
Assumptions~\ref{assump:ell} and~\ref{assump:mon} yield that, for all $u, w, z \in H^1_0(\Omega)$,
\begin{subequations}\label{appendix:mon}
\begin{align}
\enorm{u-u^\exact_\coarse}^2 &\stackrel{\mathclap{\eqref{eq:bm:stronglymonotone}}}{\le} \prod{\AA(u) -  \AA(u^\exact_\coarse)}{u-u^\exact_\coarse},\label{appendix:rel:aux} \\
\enorm{z-z^\exact_\coarse[w]}^2 &\le \prod{\AA'[w](z-z^\exact_\coarse[w])}{z-z^\exact_\coarse[w]} \label{appendix:rel:aux:dual}.
\end{align}
\end{subequations}
For all $v_\coarse \in \XX_\coarse$, the Galerkin orthogonalities for the primal and dual setting read
\begin{subequations}\label{appendix:gal}
\begin{alignat}{2}
\prod{\AA (u^\exact) - \AA (u^\exact_\coarse)}{ v_\coarse}  &=0 &&=\prod{\AA (u^\exact_\fine) - \AA (u^\exact_\coarse)}{ v_\coarse}, \label{eq:gal:primal}\\
\prod{\AA'[w](z^\exact[w])- \AA'[w](z^\exact_\coarse[w])}{ v_\coarse} &=0&&= \prod{\AA'[w](z^\exact_\fine[w])- \AA'[w](z^\exact_\coarse[w])}{ v_\coarse}. \label{eq:gal:dual}
\end{alignat}
\end{subequations}
 For $d\in \{2,3\}$, let $\I_{\coarse}\colon H^1_0(\Omega) \to \XX_{\coarse} $ be a Cl\'{e}ment-type quasi-interpolation operator, while $\I_\coarse$ is the nodal interpolation operator for $d=1$. For the primal setting, let $u \in \{u^\exact, u^\exact_\fine\}$ and choose $\XX \in \{H^1_0(\Omega), \XX_\fine \}$ accordingly. Then,~\eqref{appendix:mon}--\eqref{appendix:gal} and~\eqref{eq:weakform:primal} or~\eqref{eq:weakform:primal:discrete} (according to $u$) lead to
\begin{align}\label{appendix:aux:primal}
\enorm{u-u^\exact_\coarse} & \le \sup_{0 \neq v \in \XX}  \enorm{v}^{-1}  \prod{\AA(u) - \AA(u^\exact_\coarse)}{v}\, = \sup_{0 \neq v \in \XX}  \enorm{v}^{-1}  \prod{\AA(u) -\AA(u^\exact_\coarse)}{v-\I_\coarse \!v}  \notag\\
&= \, \, \sup_{0 \neq v \in \XX}  \enorm{v}^{-1}  \big[ \prod{f}{v-\I_\coarse \!v} + \prod{\f}{\nabla (v-\I_\coarse \!v)} - \prod{\AA(u^\exact_\coarse)}{v-\I_\coarse \!v}\big].
\end{align}
For the dual setting, let $z \in \{ z^\exact[w], z^\exact_\fine[w]\}$ and choose $\XX \in \{H^1_0(\Omega), \XX_\fine \}$ accordingly. Using~\eqref{appendix:mon}--\eqref{appendix:gal}, and~\eqref{eq:weakform:dual:exact} or \eqref{eq:weakform:dual:discrete:exact} (according to $z$), the same arguments as above yield that
\begin{align}\label{appendix:aux:dual}
\hspace{-3mm}\enorm{z-z^\exact_\coarse[w]} & \!\le \! \sup_{0 \neq v \in \XX}  \enorm{v}^{-1}  \big[ \prod{g\!}{\!v\!-\!\I_\coarse\! v} \!+ \!\prod{\g\!}{\!\nabla (v\!-\!\I_\coarse\! v)} \!-\! \prod{\AA'[w](z^\exact_\coarse[w])\!}{\!v\!-\!\I_\coarse \!v}\big].
\end{align}
Based on~\eqref{appendix:aux:primal}--\eqref{appendix:aux:dual}, standard arguments employing elementwise integration by parts and fine properties of Cl\'{e}ment-type operators conclude reliability~\eqref{assumption:rel}, i.e.,
\begin{align*}
 \enorm{u^\exact-u^\exact_\coarse} \lesssim \eta_\coarse(u^\exact_\coarse) \quad \text{ and } \quad \enorm{z^\exact[w] - z_\coarse^\exact[w]} 
	\le \Crel \, \zeta_\coarse(z_\coarse^\exact[w]).
\end{align*}
The hidden constants depend only on $\I_\coarse$ and, hence, only on $d$ and $\mu_0$.
\end{proof}
\begin{proof}[Sketch of proof of discrete reliability~\eqref{assumption:drel}] To prove discrete reliability~\eqref{assumption:drel}, we choose $\I_\coarse$ as the Scott--Zhang projector~\cite{sz1990} for $d\in \{2,3\}$, which is a Cl\'{e}ment-type quasi-interpolation operator, and note that $\I_\coarse$ can be chosen in such a way that $\left( v_\fine - \I_\coarse v_\fine \right) |_T = 0$ for all $T \in \TT_\coarse \cap \TT_\fine$ and $v_\fine \in \XX_\fine$; see~\cite{ckns2008}. Standard arguments then show that
\begin{align*}
\enorm{u^\exact_\fine- u^\exact_\coarse} \lesssim \eta_\coarse(\TT_\coarse \backslash \TT_\fine, u_\coarse^\exact) \quad \text{ and } \quad \enorm{z_\fine^{\exact}[w] - z_\coarse^{\exact}[w]}
	&\lesssim \zeta_\coarse(\TT_\coarse \backslash \TT_\fine, z_\coarse^{\exact}[w]).
\end{align*}
The hidden constants depend only on the dimension $d$, the polynomial degree $m$, and norm equivalence. This concludes the proof of discrete reliability~\eqref{assumption:drel}.
\end{proof}

\subsection{Stability of dual problem}\label{subsection:stability}

The next result transfers \cite[Lemma 6]{bip2020} to the present setting of semilinear PDEs. It shows that the norm difference of dual solutions can be estimated by that of the corresponding primal solutions.

\begin{lemma}\label{lemma:workhorse}
Suppose~\ref{assump:rhs},~\ref{assump:ell},~\ref{assump:car},~\ref{assump:mon}, and~\ref{assump:poly}. Let $M > 0$ and $w \in H^1_0(\Omega)$ with $\enorm{w} \le M$. Then, it holds that
\begin{align}\label{eq:workhorse}
  \enorm{z^\exact[u^\exact]-z^\exact[w]} +\enorm{z_\coarse^\exact[u^\exact]-z_\coarse^\exact[w]} \le \Cdiff \enorm{u^\exact - w},
\end{align}
where $\Cdiff = \Cdiff(|\Omega|, d, M, n, R, p, f, \f, g, \g, \mu_0)$.
\end{lemma} 

\begin{proof}
First, note that
\begin{equation} 
\begin{alignedat}{2}\label{eq:aux:dual:ortho}
\sprod{z^\exact[u^\exact]-z^\exact[w]}{v} + \prod{b'(u^\exact)z^\exact[u^\exact] - b'(w) z^\exact[w]}{v}&=0 \quad &&\text{ for all } v \in H^1_0(\Omega), \\
\sprod{z^\exact_\coarse[u^\exact]-z^\exact_\coarse[w]}{v_\coarse} + \prod{b'(u^\exact)z^\exact_\coarse[u^\exact] - b'(w) z^\exact_\coarse[w]}{v_\coarse}&=0 \quad &&\text{ for all } v_\coarse \in \XX_\coarse.
\end{alignedat}
\end{equation}
We aim to prove that
\begin{align*}
  \enorm{z^\exact[u^\exact]-z^\exact[w]} \le \Cdiff \enorm{u^\exact - w}.
\end{align*}
To this end, note that the strategy in the proof of Proposition~\ref{prop:dual:energybound} provides a similar estimate to~\eqref{eq:tbp:const} by choosing $t$ from Remark~\ref{remark:reg12}{\rm{(ii)}} instead of $s$ from Remark~\ref{remark:reg12}{\rm{(i)}}, i.e., 
\begin{align}\label{eq:workhorse:intermediate}
\norm{b'(u^\exact) - b'(w)}{L^{t''}(\Omega)} \le \Cdual' \enorm{u^\exact - w},
\end{align}
with $\Cdual' = \Cdual'(|\Omega|, d, \norm{u^\exact}{L^\infty(\Omega)}, M, n, R, p, f, \f, g, \g, \mu_0)> 0$. The H\"{o}lder inequality leads us to
\begin{align}\label{eq:workhorse:estimate}
\enorm{z^\exact[u^\exact]-z^\exact[w]}^2 &= \sprod{z^\exact[u^\exact]-z^\exact[w]}{ z^\exact[u^\exact]-z^\exact[w]} \nonumber \\
&\hspace{-25mm}\stackrel{\mathmakebox[\widthof{=}]{\eqref{eq:aux:dual:ortho}}}{=} - \prod{b'(u^\exact)z^\exact[u^\exact] - b'(w) z^\exact[w]}{z^\exact[u^\exact]-z^\exact[w]} \nonumber \\
&\hspace{-25mm}=  -\prod{(b'(u^\exact) - b'(w))z^\exact[u^\exact]}{z^\exact[u^\exact]-z^\exact[w]} \nonumber - \prod{b'(w)(z^\exact[u^\exact]-z^\exact[w])}{z^\exact[u^\exact]-z^\exact[w]} \nonumber \\
&\hspace{-28.5mm} \stackrel{\ref{assump:mon}}{\le} \, -\prod{(b'(u^\exact) - b'(w))z^\exact[u^\exact]}{z^\exact[u^\exact]-z^\exact[w]} \\
&\hspace{-25mm}\le \norm{b'(u^\exact) - b'(w)}{L^{t''}(\Omega)} \, \norm{z^\exact[u^\exact]}{L^{t}(\Omega)} \,\norm{z^\exact[u^\exact] - z^\exact[w]}{L^{{t''}}(\Omega)}  \nonumber \\
&\hspace{-25mm}\stackrel{\mathclap{\eqref{eq:workhorse:intermediate}}}{\lesssim} \, \enorm{u^\exact- w} \,\enorm{z^\exact[u^\exact]}\,\enorm{z^\exact[u^\exact] - z^\exact[w]} \nonumber, 		
\end{align}
where the hidden constant depends only on $\Cdual'$ from~\eqref{eq:workhorse:intermediate} and norm equivalence. Finally, recall that $\enorm{z^\exact[u^\exact]} \le \Cbnd$ from Lemma~\ref{lemma:seminorm}. The same reasoning applies for $\enorm{z^\exact_\coarse[u^\exact]-z^\exact_\coarse[w]}$. This concludes the proof.
\end{proof}

\subsection{Proof of Proposition~\ref{prop:convergenceA}}

The proof of Proposition~\ref{prop:convergenceA} builds on the following lemma, which adapts~\cite[Proposition~14]{bip2020} to the present setting.

\begin{lemma}
\label{prop:plain_convergence}
Suppose~~\ref{assump:rhs},~\ref{assump:ell},~\ref{assump:car},~\ref{assump:mon}, and~\ref{assump:poly}. Then, for any choice of the marking parameters $0 < \theta \le 1$ and $1 \le \Cmark \le \infty$, Algorithm~\ref{algorithm:first} guarantees that
\begin{itemize}
\item $\enorm{u^\exact - u_\ell^\exact} + \eta_\ell(u_\ell^\exact) \to 0$ \quad
 if $\#\set{k \in \N_0}{\MM_k \text{ satisfies~\eqref{eq:first:doerfler:u}}} = \infty$,
\item $\enorm{u^\exact - u_\ell^\exact} + \eta_\ell(u_\ell^\exact) + \enorm{z^\exact[u^\exact] - z_\ell^\exact[u_\ell^\exact]} + \enorm{z^\exact[u^\exact] - z_\ell^\exact[u^\exact]} + \zeta_\ell(z_\ell^\exact[u_\ell^\exact]) \to 0$ \quad \\
 if $\#\set{k \in \N_0}{\MM_k \text{ satisfies~\eqref{eq:first:doerfler:uz}}} = \infty$,
\end{itemize}
as $\ell \to \infty$. Moreover, at least one of these two cases is met.
\end{lemma}

\begin{proof}[Sketch of proof]
The proof is essentially verbatim to that of~\cite[Proposition~14]{bip2020} and therefore only sketched. From the C\'ea lemma~\eqref{eq:cea} for the primal problem (resp.~\eqref{eq:cea:dual} for the dual problem), the nestedness $\XX_\ell \subseteq \XX_{\ell+1}$ of the discrete spaces for all $\ell \in \N_0$, and the stability of the dual problem (Lemma~\ref{lemma:workhorse}), it follows that there exist  \textsl{a~priori} limits $u_\infty^\exact, z_\infty^\exact[u_\infty^\exact] \in H^1_0(\Omega)$ such that
\begin{align*}
 \enorm{u_\infty^\exact - u_\ell^\exact} + \enorm{z_\infty^\exact[u_\infty^\exact] - z_\ell^\exact[u_\ell^\exact]} 
 \xrightarrow{\ell \to \infty} 0.
\end{align*}
Together with stability~\eqref{assumption:stab} and reduction~\eqref{assumption:red}, the \emph{estimator reduction principle} proves that 
\begin{align*}
 \eta_\ell(u_\ell^\exact) \xrightarrow{\ell \to \infty} 0 
 &\quad \text{if }\#\set{k \in \N_0}{\MM_k \text{ satisfies}~\eqref{eq:first:doerfler:u}} = \infty, 
 \\
 \eta_\ell(u_\ell^\exact) + \zeta_\ell(z_\ell^\exact[u_\ell^\exact]) \xrightarrow{\ell \to \infty} 0 
 &\quad \text{if }\#\set{k \in \N_0}{\MM_k \text{ satisfies}~\eqref{eq:first:doerfler:uz}} = \infty.  
\end{align*}
Clearly, at least one of these two cases is met. With reliability~\eqref{assumption:rel}, it follows that $u^\exact = u_\infty^\exact$, while $z^\exact[u^\exact] = z_\infty^\exact[u_\infty^\exact]$ requires that $\#\set{k \in \N_0}{\MM_k \text{ satisfies}~\eqref{eq:first:doerfler:uz}} = \infty$;
see~\cite[Proposition~14]{bip2020} for details.
\end{proof}

\begin{proof}[\textbf{Proof of Proposition~\ref{prop:convergenceA}}]
The proof is verbatim that of~\cite[Proposition~1]{bip2020} and therefore only sketched. From~\eqref{assumption:stab}--\eqref{assumption:rel}, the C\'ea lemma~\eqref{eq:cea} for the primal problem (resp.~\eqref{eq:cea:dual} for the practical dual problem), and the nestedness of the discrete spaces, there follows boundedness 
\begin{align*}
 \eta_\ell(u_\ell^\exact) + \zeta_\ell(z_\ell^\exact[u_\ell^\exact])
 \lesssim \eta_0(u_0^\exact) + \zeta_0(z_0^\exact[u_0^\exact]) < \infty
 \quad \text{for all } \ell \in \N_0;
\end{align*}
see~\cite[Section~4.1]{bip2020} for details. Together with the convergence results of Lemma~\ref{prop:plain_convergence}, this yields convergence
\begin{align*}
 \eta_\ell(u_\ell^\exact) \big[\eta_\ell(u_\ell^\exact)^2 + \zeta_\ell(z_\ell^\exact[u_\ell^\exact])^2 \big]^{1/2} \xrightarrow{\ell \to \infty} 0.
\end{align*}
This concludes the proof.
\end{proof}

\subsection{Auxiliary results}

We continue with some preliminary results, which are needed for proving the quasi-orthogonalities and which are, hence, crucial to prove linear convergence. To this end, consider the Fr\'{e}chet derivative of $\AA$ at $w \in H^1_0(\Omega)$, i.e.,
\begin{align}\label{eq:frechet:dual}
\AA'[w]( \, \cdot \,)\colon H^1_0(\Omega) \to H^{-1}(\Omega), \quad \AA'[w](z) := \sprod{z}{\cdot \,} + \prod{b'(w)z}{\cdot}.
\end{align}

\begin{lemma}\label{lemma:qo:taylor}
Suppose~\ref{assump:rhs},~\ref{assump:ell},~\ref{assump:car},~\ref{assump:mon}, and~\ref{assump:poly}.
Then, there exists a constant $C=C(|\Omega|, d, \norm{u^\exact}{L^\infty(\Omega)}, n, R, p, f, \f, \mu_0)$ such that
\begin{align}
\prod{\AA (u^\exact) - \AA (u^\exact_\ell) -  \AA'[u^\exact]( u^\exact-u^\exact_\ell)}{v}\le C \, \enorm{u^\exact - u^\exact_\ell}^{2} \enorm{v} \quad\text{ for all } v \in H^1_0(\Omega).
\end{align}
\end{lemma}
\begin{proof}
Due to the linearity of $\sprod{\,\cdot}{\cdot\,}$ in the left-hand argument, we conclude that the only contribution is due to $b$, i.e., for all $v \in H^1_0(\Omega)$, it holds that
\begin{align*}
\prod{\AA (u^\exact) - \AA (u^\exact_\ell) -  \AA'[u^\exact]( u^\exact-u^\exact_\ell)}{v} = \prod{b(u^\exact) - b(u^\exact_\ell) - b'(u^\exact) (u^\exact-u^\exact_\ell)}{v}. 
\end{align*}
For $v \in H^1_0(\Omega)$, the H\"{o}lder inequality with arbitrary $1 < s < \infty$ if $d\in \{1,2\}$ and $s = 2^\ast$ if $d = 3$ proves that
\begin{align*}
\prod{b(u^\exact) - b(u^\exact_\ell) - b'(u^\exact) (u^\exact-u^\exact_\ell)}{v} \lesssim  \norm{b(u^\exact) - b(u^\exact_\ell) - b'(u^\exact) (u^\exact-u^\exact_\ell)}{L^{s'}(\Omega)} \enorm{v}.
\end{align*}
From the Taylor expansion~\eqref{eq:primal:taylorexp}, note that
\begin{align*}
b(u^\exact) - b(u^\exact_\ell) - b'(u^\exact) (u^\exact-u^\exact_\ell) &= - \sum_{k=2}^{n-1} b^{(k)}(u^\exact) \, \frac{( u^\exact_\ell - u^\exact )^k}{k!} \\
& \hspace{3mm}- \frac{ \, ( u^\exact_\ell- u^\exact )^{n}}{(n-1)!} \, \int_0^1 ( 1 - \tau )^{n-1} \, b^{(n)}\big( u^{\exact} + (u^\exact_\ell - u^\exact) \, \tau \big)  \d{\tau}.
\end{align*}
Together with Lemma~\ref{lemma:seminorm} and $u^\exact \in L^\infty(\Omega)$, the assumption~\ref{assump:poly} yields that
\begin{align*}
&\norm{b(u^\exact) - b(u^\exact_\ell) - b'(u^\exact) (u^\exact-u^\exact_\ell)}{L^{s'}(\Omega)} 
\lesssim \sum_{k=2}^{n} \norm{(u^\exact-u^\exact_\ell)^k}{L^{s'}(\Omega)} 
\\& \qquad
= \sum_{k=2}^{n} \norm{u^\exact-u^\exact_\ell}{L^{ks'}(\Omega)}^k 
\lesssim \norm{\nabla(u^\exact-u^\exact_\ell)}{L^2(\Omega)}^2 
\simeq \enorm{u^\exact-u^\exact_\ell}^2,
\end{align*}
where the hidden constants depend only on $\Cbnd$ from Lemma~\ref{lemma:seminorm}, $n, R$ from~\ref{assump:poly} and norm equivalence. This concludes the proof.
\end{proof}
The next lemma is an auxiliary result for establishing quasi-orthogonality. Our proof combines arguments from the linear setting~\cite[Lemma~17]{bhp2017} with ideas from~\cite[Lemma~6.10]{ffp2014}. We stress that the proof exploits the \textsl{a~priori} convergence $\enorm{u^\exact - u_\ell^\exact} \to 0$ from Lemma~\ref{prop:plain_convergence}.

\begin{lemma}\label{lemma:weakconvergence}
Suppose~~\ref{assump:rhs},~\ref{assump:ell},~\ref{assump:car},~\ref{assump:mon}, and~\ref{assump:poly}.
Then, the normalized sequences
\begin{align}
e_\ell := \begin{cases}
        \displaystyle \frac{u^\exact-u^\exact_\ell}{\enorm{u^\exact - u^\exact_\ell}}, & \text{for } u^\exact \neq u^\exact_\ell\\
        0, & \text{otherwise} 
        \end{cases} \quad \text{ and } \quad E_\ell := \begin{cases}\displaystyle
         \frac{u^\exact_{\ell + 1}-u^\exact_\ell}{\enorm{u^\exact_{\ell + 1} - u^\exact_\ell}}, & \text{for } u^\exact_{\ell + 1} \neq u^\exact_\ell\\
        0, & \text{otherwise} 
        \end{cases}
\end{align}
converge weakly to $0$ in $H^1_0(\Omega)$.
\end{lemma}

\begin{proof}
We only show the statement for $e_\ell$. The proof for $E_\ell$ follows by similar arguments. To prove that $e_\ell \weakly 0$ in $H^1_0(\Omega)$, we show that each subsequence $(e_{\ell_k})_{k \in \N_0}$ admits a further subsequence $(e_{\ell_{k_j}})_{j \in \N_0}$ such that $e_{\ell_{k_j}} \weakly 0$ as $j \to \infty$. To this end, consider a subsequence $(e_{\ell_k})_{k \in \N_0}$ of $(e_\ell)_{\ell \in \N_0}$. Without loss of generality, we may assume that $e_{\ell_k} \neq 0$ for all $k \in \N_0$. Note that $\enorm{e_{\ell_k}} \le 1$. Hence, the Banach--Alaoglu theorem yields a further subsequence $(e_{\ell_{k_j}})_{j \in \N_0}$ satisfying weak convergence $e_{\ell_{k_j}} \weakly w_\infty \in H^1_0(\Omega)$ as $j \to \infty$. It remains to show that $w_\infty = 0$. Lemma~\ref{prop:plain_convergence} implies that $u^\exact \in \XX_\infty$ and, hence, $e_\ell \in \XX_\infty$ for all $\ell \in \N_0$. Mazur's lemma (see, e.g., \cite[Theorem~25.2]{fk1980}) yields that $w_\infty \in \XX_\infty$.

First, the Galerkin orthogonality shows that
\begin{align*}
\prod{\AA(u^\exact) - \AA(u^\exact_{\ell_{k_j}})}{v_i} = 0 \quad \text{ for all } i\le {\ell_{k_j}} \text{ and } v_i \in \XX_i.
\end{align*}
Letting $j \to \infty$, we infer that
\begin{align*}
\lim_{j \to \infty} \,\frac{\prod{\AA(u^\exact) - \AA(u^\exact_{\ell_{k_j}})}{v_i}}{\enorm{u^\exact - u^\exact_{\ell_{k_j}}}} = 0 \quad \text{ for all }i \in \N_0 \text{ and } v_i \in \XX_i.
\end{align*}
Let $v_\infty \in \XX_\infty$. By definition of $\XX_\infty$, there exists a sequence $(v_i)_{i \in \N_0}$ with $v_i \in \XX_i$ and $\enorm{v_\infty - v_i} \to 0$ as $i \to \infty$. Given $\varepsilon > 0$, there exists $i_0 \in \N_0$ such that $\enorm{v_i - v_\infty} \le \varepsilon$ for all $i_0 \le i \in \N_0$. Estimate~\eqref{eq:critical:lipschitz} yields that 
\begin{align*}
\lim_{j \to \infty}  \,\frac{\prod{\AA(u^\exact) - \AA(u^\exact_{\ell_{k_j}})}{v_i - v_\infty}}{\enorm{u^\exact - u^\exact_{\ell_{k_j}}}} \lesssim \enorm{ v_i-v_\infty}\le \varepsilon ,
\end{align*}
where the hidden constant depends only on $\CLip$. Hence, we get that
\begin{align*}
\lim_{j \to \infty} \,\frac{\prod{\AA(u^\exact) - \AA(u^\exact_{\ell_{k_j}})}{v_\infty}}{\enorm{u^\exact - u^\exact_{\ell_{k_j}}}} = 0 \quad \text{ for all } v_\infty \in \XX_\infty.
\end{align*}
Moreover, Lemma~\ref{lemma:qo:taylor} and the triangle inequality lead to
\begin{align*}
 \frac{|\prod{ \AA'[u^\exact]( u^\exact - u^\exact_{\ell_{k_j}})}{v_\infty}|}{\enorm{u^\exact - u^\exact_{\ell_{k_j}}}} \le \frac{|\prod{\AA(u^\exact) - \AA(u^\exact_{\ell_{k_j}})}{v_\infty}|}{\enorm{u^\exact - u^\exact_{\ell_{k_j}}}} + C \,{\enorm{u^\exact - u^\exact_{\ell_{k_j}}}} \enorm{v_\infty}. 
\end{align*}
Together with a priori convergence $\enorm{u^\exact - u^\exact_{\ell_{k_j}}} \to 0$, we thus obtain that
\begin{align}\label{eq:qo:helper2}
\lim_{j \to \infty} \,\frac{|\prod{ \AA'[u^\exact]( u^\exact - u^\exact_{\ell_{k_j}})}{v_\infty}|}{\enorm{u^\exact - u^\exact_{\ell_{k_j}}}} =0.
\end{align}
Note that due to~\ref{assump:ell} and~\ref{assump:mon}, $\AA'[u^\exact]( \, \cdot \, )$ is bounded from below, i.e.,
\begin{align*} 
\enorm{v}^2 \le \prod{ \AA'[u^\exact](v)}{v} \lesssim \norm{\AA'[u^\exact]( v)}{H^{-1}(\Omega)} \enorm{v}\quad \text{ for all } v \in H^1_0(\Omega).
\end{align*}%
Due to the smoothness of $\xi \mapsto b(t,\xi)$ and the $L^\infty$-bound for $u^\exact$ from Proposition~\ref{proposition:weakform:primal:Linf}, we infer that $0 \le b'(u^\exact)\le C$. Hence, $\AA'[u^\exact]( \, \cdot \,)$ is a bounded linear operator and the restriction $\AA'[u^\exact](\, \cdot \,)|_{\XX_\infty}\colon \XX_\infty \to \XX_\infty^\ast$ is an isomorphism. Consequently, also the adjoint $(\AA'[u^\exact]|_{\XX_\infty})^\ast\colon \XX_\infty^\ast \to \XX_\infty$ is an isomorphism, where we note that $\XX_\infty$ is a closed subspace of the Hilbert space $H^1_0(\Omega)$ and, hence, reflexive. Hence, for every $\widetilde{v}_\infty\in \XX_\infty$, there exists $v_\infty \in \XX_\infty$ such that
\begin{align*}
0 =\lim_{j \to \infty} \,\frac{|\prod{ \AA'[u^\exact]( u^\exact - u^\exact_{\ell_{k_j}})}{v_\infty}|}{\enorm{u^\exact - u^\exact_{\ell_{k_j}}}}&= \lim_{j \to \infty} \,\frac{|\prod{ \AA'[u^\exact]^\ast( v_\infty)}{u^\exact - u^\exact_{\ell_{k_j}}}|}{\enorm{u^\exact - u^\exact_{\ell_{k_j}}}} \\
&= \lim_{j \to \infty} \,\frac{|\sprod{u^\exact - u^\exact_{\ell_{k_j}}}{\widetilde{v}_\infty}|}{\enorm{u^\exact - u^\exact_{\ell_{k_j}}}}=\lim_{j \to \infty} \,\sprod{e_{\ell_{k_j}}}{\widetilde v_\infty}.
\end{align*}
This shows that $w_\infty = 0$ and concludes the proof.
\end{proof}

\subsection{Quasi-orthogonalities}\label{subsection:quasiorthogonalities}

Our proof of the crucial quasi-orthogonalities adapts that of~\cite[Lemma 17, 18]{bhp2017} from the linear setting in the Lax--Milgram framework to the present nonlinear setting. However, we stress that the following results all need the stronger growth condition~\ref{assump:compact}, while our earlier results only require~\ref{assump:poly}.

\begin{lemma}[quasi-orthogonality for primal problem]
\label{lemma:qo:primal}
Suppose~\ref{assump:rhs},~\ref{assump:ell},\linebreak[3]~\ref{assump:car},~\ref{assump:mon}, and~\ref{assump:compact}. 
Then,
for all $0 < \eps < 1$, there exists $\ell_0 \in \N$ such that 
for all $\ell \ge \ell_0$ and all $k \in \N_0$, it holds that
\begin{align}\label{eq:qo:primal}
 \enorm{u^\exact- u^\exact_{\ell+k}}^2 + \enorm{u^\exact_{\ell+k} - u^\exact_\ell}^2 
 \le \frac{1}{1-\eps} \, \enorm{u^\exact - u^\exact_\ell}^2.
\end{align}%
\end{lemma}
\begin{proof}
Together with the Rellich--Kondrachov compactness theorem (see, e.g.,~\cite[Theorem 5.8.2]{kof1977}), Lemma~\ref{lemma:weakconvergence} yields strong convergence 
\begin{align*}
\norm{e_\ell}{L^{\sigma}(\Omega)}, \norm{E_\ell}{L^{\sigma}(\Omega)} \xrightarrow{\ell \to \infty} 0
\quad \text{ where } \quad \begin{cases}
\sigma \in [1, \infty), \quad &\text{ if } d\in \{1,2\}, \\
 {\sigma} \in [1, 2^\ast), \quad &\text{ if } d = 3. 
\end{cases}
\end{align*}
If $d=1,2$, let $1 < \sigma < \infty$ be arbitrary with H\"{o}lder conjugate $\sigma' > 1$. If $d = 3$, let $ \sigma = 5 = 2^\ast -1$ and hence $\sigma' = 5/4 =  (2^\ast-1)/(2^\ast-2)$. Note that \ref{assump:compact} yields that $n\sigma' \le \sigma$ and hence $\norm{e_{\ell}}{L^{n\sigma'}(\Omega)} \lesssim \norm{e_{\ell}}{L^\sigma(\Omega)} \to 0$ as $ \ell \to \infty$. We argue as for~\eqref{eq:taylor:compact}: By the Taylor expansion, $n\sigma' < 2^\ast$, and with Lemma~\ref{lemma:seminorm}, we obtain that
\begin{align}
\begin{split}\label{eq:helper:qo:1}
 \norm{b(u^\exact) - b(u_{\ell}^\exact)}{L^{\sigma'}(\Omega)} & \lesssim
  \sum_{j=1}^{n }  \norm{ u^\exact -u_{\ell}^\exact}{L^{j\sigma'}(\Omega)}^j \lesssim \norm{u^\exact-u^\exact_{\ell}}{L^{n\sigma'}(\Omega)} \\
  & \hspace{-30mm}= \enorm{u^\exact - u^\exact_{\ell}} \norm{e_{\ell}}{L^{n\sigma'}(\Omega)} \lesssim \enorm{u^\exact - u^\exact_{\ell}} \norm{e_{\ell}}{L^{\sigma}(\Omega)}.
  \end{split}
\end{align}  
Furthermore, for $k, \ell \in \N$, recall the Galerkin orthogonality
\begin{align}\label{eq:qo:primal:gal}
\sprod{u^\exact- u_{\ell+k}^\exact}{ v_{\ell+k}} + \prod{b(u^\exact)-b(u_{\ell+k}^\exact)}{v_{\ell+k}} = 0 \quad \text{ for all } v_{\ell+k} \in \XX_{\ell+k}.
\end{align}
Due to the bilinearity and symmetry of $\sprod{\,\cdot}{\cdot\,}$, we have that
\begin{align}\label{eq:qo:helper:symm}
\enorm{u^\exact - u_{\ell}^\exact}^2 = \enorm{u^\exact - u_{\ell+k}^\exact}^2 + \enorm{u_{\ell+k}^\exact - u_{\ell}^\exact}^2 + 2 \sprod{u^\exact -u_{\ell+k}^\exact}{ u_{\ell+k}^\exact - u_{\ell}^\exact}. 
\end{align}
Let $0 < \varepsilon < 1$. Note that $u_{\ell+k}^\exact - u_\ell^\exact \in \XX_{\ell+k}$ due to nestedness of the discrete spaces. Exploiting the Galerkin orthogonality~\eqref{eq:qo:primal:gal} and the Young inequality, we thus obtain that there exists $\ell_0$ such that, for all $\ell \ge \ell_0$ and all $k \ge 0$,
\begin{align*}
2\sprod{u^\exact -u_{\ell+k}^\exact}{ u_{\ell+k}^\exact - u_{\ell}^\exact} & \stackrel{\mathclap{\eqref{eq:qo:primal:gal}}}{\ge} -2 |\prod{b(u^\exact) - b(u_{\ell+k}^\exact)}{u_{\ell+k}^\exact - u_{\ell}^\exact}| \\
& \ge -2\norm{b(u^\exact) - b(u_{\ell+k}^\exact)}{L^{\sigma'}(\Omega)} \norm{u^\exact_{\ell+k} - u^\exact_{\ell}}{L^{\sigma}(\Omega)} \\
& \stackrel{\mathclap{\eqref{eq:helper:qo:1}}}{\ge} -2 \varepsilon  \enorm{ u^\exact- u^\exact_{\ell+k}} \enorm{u^\exact_{\ell+k} - u^\exact_{\ell}} \\
& \gtrsim -\varepsilon [\enorm{u^\exact - u^\exact_{\ell+k}}^2 + \enorm{u^\exact_{\ell+k} - u^\exact_{\ell}}^2]
\end{align*}
The combination with~\eqref{eq:qo:helper:symm} proves that
\begin{align*}
\frac{1}{1 - \varepsilon} \enorm{u^\exact - u_{\ell}^\exact}^2 \ge \enorm{u^\exact - u_{\ell+k}^\exact}^2 + \enorm{u_{\ell+k}^\exact - u_{\ell}^\exact}^2. 
\end{align*}
This concludes the proof.
\end{proof}

While Lemma~\ref{prop:plain_convergence} guarantees \textsl{a~priori} convergence $\enorm{u^\exact - u_\ell^\exact} \to 0$ of the primal problem, \textsl{a~priori} convergence of the dual problem has to be assumed (and depends on the marking steps).

\begin{lemma}[quasi-orthogonality for \textit{exact} practical dual problem]
\label{lemma:qo:dual}%
Suppose \linebreak[3]\ref{assump:rhs},~\ref{assump:ell},~\ref{assump:car},~\ref{assump:mon}, and~\ref{assump:compact}. 
Suppose that $\enorm{z^{\exact}[u^\exact] - z_\ell^{\exact}[u^\exact]} \to 0$ as $\ell \to \infty$. Then,
for all $0 < \eps < 1$, there exists $\ell_0 \in \N$ such that for all $\ell \ge \ell_0$ and all $k \in \N_0$, it holds that
\begin{align}\label{eq:qo:dual}
 \enorm{z^\exact[u^\exact] - z^\exact_{\ell+k}[u^\exact]}^2 + \enorm{z^\exact_{\ell+k}[u^\exact] - z_\ell^\exact[u^\exact]}^2 
 \le \frac{1}{1-\eps} \, \enorm{z^\exact[u^\exact] - z^\exact_\ell[u^\exact]}^2.
\end{align}%
\end{lemma}
\begin{proof}
Note that the dual problem reads
\begin{align*}
a(z^\exact[u^\exact], v) + \prod{\KK(z^\exact[u^\exact])}{v} = G(v) \quad \text{ for all } v \in H^1_0(\Omega),
\end{align*}
where $\KK(w) := b'(u^\exact)w \in L^2(\Omega)$ defines a compact operator $\KK \colon H^1_0(\Omega) \to H^{-1}(\Omega)$. The claim thus follows from \cite[Lemma~17,~18]{bhp2017}.
 \end{proof}
 
\begin{lemma}[combined quasi-orthogonality for \textit{inexact} practical dual problem]\label{lemma:orth:dual}%
Suppose~\ref{assump:rhs},~\ref{assump:ell},~\ref{assump:car},~\ref{assump:mon}, and~\ref{assump:compact}. 
Suppose that $\enorm{z^\exact[u^\exact] - z^\exact_\ell[u^\exact_\ell]} \to 0$ as $\ell \to \infty$.
Then, for all $0 < \delta < 1$, there exists $\ell_0 \in \N$ such that for all $\ell \ge \ell_0$ and all $k \in \N_0$, it holds that
\begin{align}\label{eq:qo:combined}
 &\big[\, \enorm{u^\exact-u^\exact_{\ell+k}}^2 + \enorm{z^\exact[u^\exact] - z^\exact_{\ell+k}[u^\exact_{\ell+k}]}^2 \, \big]
 + \big[ \, \enorm{u^\exact_{\ell+k} - u^\exact_\ell}^2 + \enorm{z^\exact_{\ell+k}[u^\exact_{\ell+k}] - z^\exact_\ell[u^\exact_\ell]}^2 \, \big] \nonumber
\\& \qquad
 \le \frac{1}{1-\delta} \, \big[ \, \enorm{u^\exact - u^\exact_\ell}^2 + \enorm{z^\exact[u^\exact] - z^\exact_\ell[u^\exact_\ell]}^2 \, \big].
\end{align}\vspace*{-\baselineskip} 
\end{lemma}
\begin{proof}
According to Lemma~\ref{lemma:workhorse}, it holds that
\begin{align*}
 \enorm{z^\exact[u^\exact] - z^\exact_\ell[u^\exact]}
 &\le \enorm{z^\exact[u^\exact] - z^\exact_\ell[u^\exact_\ell]} + \enorm{z^\exact_\ell[u^\exact] - z^\exact_\ell[u^\exact_\ell]} \\
 & \stackrel{\mathclap{\eqref{eq:workhorse}}}{\lesssim} \enorm{z^\exact[u^\exact] - z^\exact_\ell[u^\exact_\ell]} + \enorm{u^\exact - u^\exact_\ell}
 \xrightarrow{\ell \to \infty} 0.
\end{align*}
Hence, we may exploit the conclusions of Lemma~\ref{lemma:qo:primal} and Lemma~\ref{lemma:qo:dual}.
For arbitrary $\alpha > 0$, the Young inequality guarantees that 
\begin{align*}
 \enorm{z^\exact[u^\exact] - z^\exact_{\ell+k}[u^\exact_{\ell+k}]}^2
 &\le (1 + \alpha) \, \enorm{z^\exact[u^\exact] - z^\exact_{\ell+k}[u^\exact]}^2
 + (1 + \alpha^{-1}) \, \enorm{z^\exact_{\ell+k}[u^\exact] - z^\exact_{\ell+k}[u^\exact_{\ell+k}]}^2,
 \\
 \enorm{z^\exact_{\ell+k}[u^\exact_{\ell+k}] - z^\exact_\ell[u^\exact_\ell]}^2
 &\le (1+\alpha) \, \enorm{z^\exact_{\ell+k}[u^\exact] - z^\exact_{\ell}[u^\exact]}^2
 + (1+\alpha^{-1})^2 \, \enorm{z^\exact_{\ell}[u^\exact] - z^\exact_{\ell}[u^\exact_\ell]}^2 
 \\ &\qquad
 + (1+\alpha)(1+\alpha^{-1}) \, \enorm{z^\exact_{\ell+k}[u^\exact] - z^\exact_{\ell+k}[u^\exact_{\ell+k}]}^2,
 \\
 \enorm{z^\exact[u^\exact] - z^\exact_{\ell}[u^\exact]}^2
 &\le (1 + \alpha) \, \enorm{z^\exact[u^\exact] - z^\exact_{\ell}[u^\exact_{\ell}]}^2
 + (1 + \alpha^{-1}) \, \enorm{z^\exact_{\ell}[u^\exact] - z^\exact_{\ell}[u^\exact_{\ell}]}^2.
\end{align*}
Together with Lemma~\ref{lemma:qo:dual}, this leads to
\begin{align}\label{eq1:orth:dual}
 &\enorm{z^\exact[u^\exact] - z^\exact_{\ell+k}[u^\exact_{\ell+k}]}^2
 + \enorm{z^\exact_{\ell+k}[u^\exact_{\ell+k}] - z^\exact_\ell[u^\exact_\ell]}^2 \notag
 \\ & \quad
 \le (1+\alpha) \, \big[ \, \enorm{z^\exact[u^\exact] - z^\exact_{\ell+k}[u^\exact]}^2 + \enorm{z^\exact_{\ell+k}[u^\exact] - z^\exact_{\ell}[u^\exact]}^2 \, \big] \notag
 \\& \qquad
 + (2+\alpha)(1+\alpha^{-1}) \, \enorm{z^\exact_{\ell+k}[u^\exact] - z^\exact_{\ell+k}[u^\exact_{\ell+k}]}^2
 + (1+\alpha^{-1})^2 \, \enorm{z^\exact_{\ell}[u^\exact] - z^\exact_{\ell}[u^\exact_\ell]}^2 \notag
 \\ 
 & \quad
\stackrel{\mathclap{\eqref{eq:qo:dual}}}{\le} \frac{1+\alpha}{1-\eps} \, \enorm{z^\exact[u^\exact] - z^\exact_\ell[u^\exact]}^2
 + (1+\alpha^{-1})^2 \, \enorm{z^\exact_{\ell}[u^\exact] - z^\exact_{\ell}[u^\exact_\ell]}^2 \notag
 \\& \qquad
 + (2+\alpha)(1+\alpha^{-1}) \, \enorm{z^\exact_{\ell+k}[u^\exact] - z^\exact_{\ell+k}[u^\exact_{\ell+k}]}^2
\notag
 \\ & \quad
 \le \frac{(1+\alpha)^2}{1-\eps} \, \enorm{z^\exact[u^\exact] - z^\exact_{\ell}[u^\exact_{\ell}]}^2
 + \Big[ (1+\alpha^{-1})^2 + \frac{(1+\alpha^{-1})(1+\alpha)}{1-\eps} \Big] \, \enorm{z^\exact_{\ell}[u^\exact] - z^\exact_{\ell}[u^\exact_\ell]}^2
 \hspace*{-5mm} \notag
 \\& \qquad
 + (2+\alpha)(1+\alpha^{-1}) \, \enorm{z^\exact_{\ell+k}[u^\exact] - z^\exact_{\ell+k}[u^\exact_{\ell+k}]}^2
\end{align}
for all $0 < \eps < 1$ and all $\ell \ge \ell_0$, where $\ell_0 \in \N_0$ depends only on $\eps$. If $d\in \{1,2\}$, let $1 < t < \infty$ be arbitrary. If $d =3$, let $t= 2^\ast$ and, hence, $t'' = 3/2$; cf.\ Remark~\ref{remark:reg12}. We argue as for~\eqref{eq:taylor:compact}: By the Taylor expansion, $\sigma:=(n-1)t'' < 2^\ast$, and with Lemma~\ref{lemma:seminorm}, we obtain that
\begin{align}
\begin{split}\label{eq:helper:qo:1:dual}
\hspace{-2mm}\norm{b'(u^\exact) \!- \!b'(u_{\ell}^\exact)}{L^{t''}(\Omega)}\! \lesssim \!
  \sum_{j=1}^{n-1}  \!\norm{ u^\exact \!-\!u_{\ell}^\exact}{L^{jt''}(\Omega)}^j \! \lesssim \! \norm{u^\exact\! - \!u^\exact_{\ell}}{L^{(n-1)t''}(\Omega)}\!  \lesssim \! \enorm{u^\exact \!- \!u^\exact_{\ell}} \norm{e_{\ell}}{L^{\sigma}(\Omega)},
  \end{split}
\end{align} 
where $\norm{e_{\ell}}{L^{\sigma}(\Omega)} \to 0$ as $ \ell \to \infty$. Recall that the inequality~\eqref{eq:workhorse:estimate} in the proof of Lemma~\ref{lemma:workhorse} does not rely on any $L^\infty(\Omega)$-bounds; hence, we may exploit the discrete analogue of~\eqref{eq:workhorse:estimate} in combination with the Hölder inequality to obtain that
\begin{align*}
\enorm{z_\ell^\exact[u^\exact]-z_\ell^\exact[u^\exact_\ell]}^2 \, & \stackrel{\mathclap{\eqref{eq:workhorse:estimate}}}{\le} -\prod{[b'(u^\exact) - b'(u^\exact_\ell)]z^\exact_\ell[u^\exact]}{z^\exact_\ell[u^\exact]-z^\exact_\ell[u^\exact_\ell]} \\
&\hspace{-25mm}\lesssim  \norm{b'(u^\exact) - b'(u^\exact_\ell)}{L^{t''}(\Omega)} \norm{z^\exact_\ell[u^\exact]}{L^{t}(\Omega)} \norm{ z^\exact_\ell[u^\exact] - z^\exact_\ell[u^\exact_\ell]}{L^{t}(\Omega)}  \\
&\hspace{-25mm}\stackrel{\mathclap{\eqref{eq:workhorse}}}{\lesssim}  \enorm{z^\exact_\ell[u^\exact]}  \norm{b'(u^\exact) - b'(u_{\ell}^\exact)}{L^{t''}(\Omega)} \enorm{u^\exact-u^\exact_\ell} \,\stackrel{\mathclap{\eqref{eq:helper:qo:1:dual}}}{\lesssim}  \enorm{z^\exact_\ell[u^\exact]} \norm{e_{\ell}}{L^{\sigma}(\Omega)}\enorm{ u^\exact-u^\exact_\ell}^2.
\end{align*}
Since $\enorm{z^\exact_\ell[z^\exact]} \le \Cbnd$ due to Lemma~\ref{lemma:seminorm}, this proves that
\begin{align}\label{eq2:orth:dual}
\enorm{z_\ell^\exact[u^\exact]-z_\ell^\exact[u^\exact_\ell]}^2 \, &\le \, \kappa_{\ell} \enorm{u^\exact - u^\exact_\ell}^2 \quad \text{ with } \quad  0 \le \kappa_\ell \xrightarrow{\ell \to \infty} 0.
\end{align}
Plugging~\eqref{eq2:orth:dual} into~\eqref{eq1:orth:dual}, we thus have shown that
\begin{align*}
 &\enorm{z^\exact[u^\exact] - z^\exact_{\ell+k}[u^\exact_{\ell+k}]}^2
 + \enorm{z^\exact_{\ell+k}[u^\exact_{\ell+k}] - z^\exact_\ell[u^\exact_\ell]}^2
 \\ & \quad
 \le \frac{(1+\alpha)^2}{1-\eps} \, \enorm{z^\exact[u^\exact] - z^\exact_{\ell}[u^\exact_{\ell}]}^2
 + \Big[ (1+\alpha^{-1})^2 + \frac{(1+\alpha^{-1})(1+\alpha)}{1-\eps} \Big] \, \kappa_\ell \, \enorm{u^\exact - u^\exact_\ell}^2
 \\ & \qquad
 + (2+\alpha)(1+\alpha^{-1}) \, \kappa_{\ell+k} \, \enorm{u^\exact - u^\exact_{\ell+k}}^2
\end{align*}
for all $0 < \eps < 1$, all $\alpha > 0$, and all $\ell \ge \ell_0$, where $\ell_0 \in \N_0$ depends only on $\eps$.
We combine this estimate with that of Lemma~\ref{lemma:qo:primal}. This leads to 
\begin{align*}
 & \big[\, \enorm{u^\exact-u^\exact_{\ell+k}}^2 + \enorm{z^\exact[u^\exact] - z^\exact_{\ell+k}[u^\exact_{\ell+k}]}^2 \, \big]
 + \big[ \, \enorm{u^\exact_{\ell+k} - u^\exact_\ell}^2 + \enorm{z^\exact_{\ell+n}[u^\exact_{\ell+k}] - z^\exact_\ell[u^\exact_\ell]}^2 \, \big]
 \\& \quad
 \le  C(\alpha,\eps,\ell) \, \big[ \, \enorm{u^\exact-u^\exact_\ell}^2 + \enorm{z^\exact[u^\exact] - z^\exact_{\ell}[u^\exact_{\ell}]}^2 \, \big] 
 + (2+\alpha)(1+\alpha^{-1}) \, \kappa_{\ell+k} \, \enorm{u^\exact - u^\exact_{\ell+k}}^2,
\end{align*}
where, since $1/(1-\varepsilon) \le (1+\alpha)^2 / (1-\varepsilon)$,  
\begin{equation*}
	C(\alpha,\eps,\ell)
	:=
	\max\bigg\{\frac{(1+\alpha)^2}{1-\eps} \, , \, \Big[ (1+\alpha^{-1})^2 + \frac{(1+\alpha^{-1})(1+\alpha)}{1-\eps} \Big] \, \kappa_\ell \bigg\}
\end{equation*}
for all $0 < \eps < 1$, all $\alpha > 0$, and all $\ell \ge \ell_0$, where $\ell_0 \in \N_0$ depends only on $\eps$.
For arbitrary $0 < \alpha, \beta, \eps < 1$, there exists $\ell_0' \in \N_0$ such that for all $\ell \ge \ell_0'$, it holds that 
\begin{align*}
(2+\alpha)(1+\alpha^{-1}) \kappa_{\ell+k} \le \beta
\end{align*}
as well as 
\begin{align*}
\Big[ (1+\alpha^{-1})^2 + \frac{(1+\alpha^{-1})(1+\alpha)}{1-\eps} \Big] \, \kappa_\ell
\le \frac{(1+\alpha)^2}{1-\eps}.
\end{align*}
Hence, we are led to
\begin{align}\label{eq3:orth:dual}
\begin{split}
 &\big[\, \enorm{u^\exact-u^\exact_{\ell+k}}^2 + \enorm{z^\exact[u^\exact] - z^\exact_{\ell+k}[u^\exact_{\ell+k}]}^2 \, \big]
 + \big[ \, \enorm{u^\exact_{\ell+k} - u^\exact_\ell}^2 + \enorm{z^\exact_{\ell+k}[u^\exact_{\ell+k}] - z^\exact_\ell[u^\exact_\ell]}^2 \, \big]
 \\& \quad
 \le \frac{(1+\alpha)^2}{(1-\eps)(1-\beta)} \, \big[ \, \enorm{u^\exact-u^\exact_\ell}^2 + \enorm{z^\exact[u^\exact] - z^\exact_{\ell}[u^\exact_{\ell}]}^2 \, \big].
\end{split}
\end{align}
Given $0 < \delta < 1$, we first fix $\alpha > 0$ such that $(1+\alpha)^2 < \frac{1}{1-\delta}$. Then, we choose $0 < \eps, \beta < 1$ such that $\frac{(1+\alpha)^2}{(1-\eps)(1-\beta)} \le \frac{1}{1-\delta}$. The choices of $\eps$ and $\beta$ also provide some index $\ell_0 \in \N_0$ such that estimate~\eqref{eq3:orth:dual} holds for all $\ell \ge \ell_0$. This concludes the proof.
\end{proof}

\begin{remark}\label{rem:upperbound}
For $d=3$, assumption~\ref{assump:compact} requires $n \in \{ 2, 3\}$. We note that, while well-posedness of the residual error estimator relies on this assumption, the quasi-orthogonalities~\eqref{eq:qo:primal} and~\eqref{eq:qo:combined} only require $n \in \{ 2, 3, 4 \}$ for $d=3$.
\end{remark}
\begin{remark}
If $d > 3$, the same reasoning using the H\"{o}lder inequality still holds true, though the polynomial degree $n$ in~\ref{assump:compact} becomes more constrained. 
\end{remark}

\subsection{Proof of Theorem~\ref{theorem:main} and Theorem~\ref{theorem:main2}}

It is a key observation in the analysis of~\cite{bip2020} that it suffices to prove 
\begin{itemize}
\addtolength{\itemindent}{-5mm}
\item stability of the (practical) dual problem (see Lemma~\ref{lemma:workhorse} resp.\ \cite[Lemma~6]{bip2020}), 
\item quasi-orthogonality of the primal problem (see Lemma~\ref{lemma:qo:primal} resp.\ \cite[Lemma~11]{bip2020}),
\item combined quasi-orthogonality for the practical dual problem (see Lemma~\ref{lemma:orth:dual} resp. \cite[Lemma~13]{bip2020}).
\end{itemize}
Then, the estimator axioms~\eqref{assumption:stab}--\eqref{assumption:drel} already prove linear convergence~\eqref{eq:linear} in the sense of Theorem~\ref{theorem:main} (see~\cite[Theorem~2(i)]{bip2020} and \cite[Section~6.1]{bip2020}) with optimal convergence rates~\eqref{eq:optimal} in the sense of Theorem~\ref{theorem:main2} (see~\cite[Theorem~2(ii)]{bip2020} and \cite[Section~6.2]{bip2020}).

%% file: 05_numerics.tex

\section{Numerical experiments}\label{section:numerical}
In this section, we test and illustrate Algorithm~\ref{algorithm:first} with numerical experiments for $d=1$ and $d=2$. We consider equation~\eqref{eq:strongform:primal}, where $\A = \mathbbm{1}$. The adaptivity parameter is set to $\theta = 0.5$. We compare the proposed GOAFEM (Algorithm~\ref{algorithm:first}) with standard AFEM (adapted from, e.g.,~\cite{axioms, ckns2008}), where mesh-refinement is driven by the primal estimator (i.e., Algorithm~\ref{algorithm:first} with $\MM_\ell := \overline \MM_\ell^{u}$ in step~\textrm{(v)}) and standard AFEM driven by the product space estimator (see Remark~\ref{rem:afemplus}).
\begin{example}[boundary value problem in 1D]\label{example:arctan}
For $d=1$ and $\Omega = (0,1)$, consider
\begin{align}\label{eq:arctan}
 -(u^\exact)'' + \arctan(u^\exact) = f \quad \text{ in } \Omega \quad \text{ subject to } \quad u^\exact(0) = u^\exact(1) = 0, 
\end{align}
with semilinearity $b(v) = \arctan(v)$ and hence $b'(v) = 1/(1+v^2)$. We set $\f = 0$ and choose $f$ in such a way that
\begin{align*}
u^\exact(x) = \sin(\pi x).
\end{align*}

  The implementation of conforming finite elements of order $m \in \{1, 2, 3, 4\}$ is done using Legendre polynomials and Gauss--Legendre quadrature and Gauss--Jacobi quadrature for the interval containing the left interval endpoint. For mesh refinement, 1D bisection is used. Moreover, we employ the (damped) Newton method from~\cite[Section 3]{aw2015} for step~{\rm{(i)}} in Algorithm~\ref{algorithm:first} to approximate the nonlinear primal problem. Let $g = x^{-9/20} \in L^2(\Omega)$ and $\g = 0$ serve as the goal functions. As a reference, we use the value of the integral which reads 
\begin{align}
G(u^\exact) = \int_0^1 \frac{\sin(\pi x)}{x^{9/20}} \d{x} \approx 0.95925303932778833\ldots.
\end{align}
The uniform initial mesh is given by $\TT_0 = \set{[\tfrac{k-1}{5}, \tfrac{k}{5}]}{k=1, \ldots, 5}$.

The numerical results are depicted in Figure~\ref{fig:arctan}. We observe that the estimator as well as the goal error achieve the expected rate $\# \TT_\coarse^{-2m}$ if computed with Algorithm~\ref{algorithm:first}. In contrast, standard AFEM leads to a slower convergence for $m \ge 2$, since singularities induced by the goal functional may not be resolved properly. 
\begin{figure}
    \centering
    \begin{subfigure}[b]{0.45\textwidth}
       \raisebox{2.5mm}{\includegraphics[width=\textwidth]{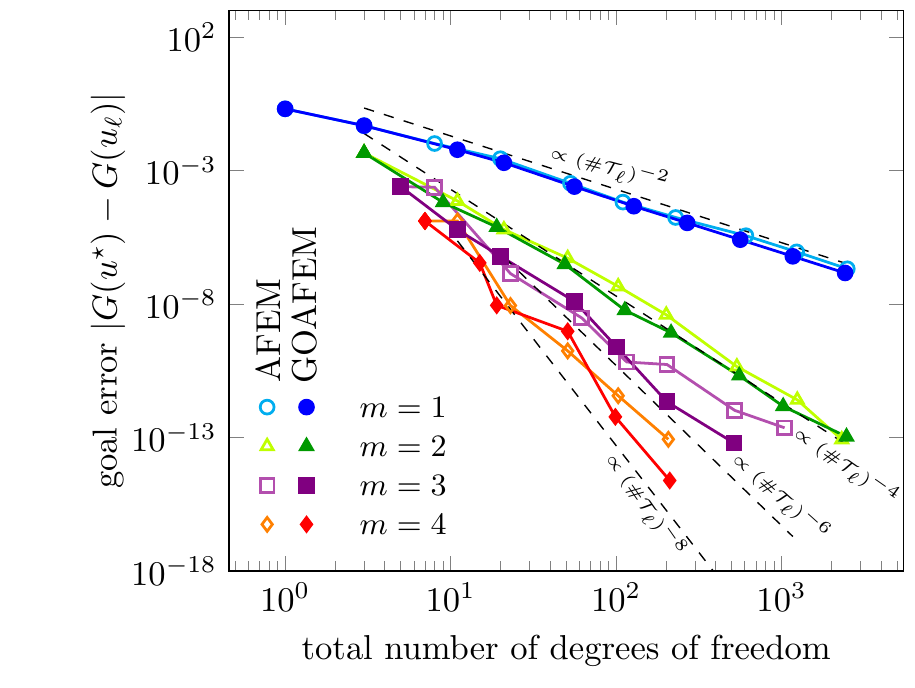}}
        \caption{Goal error $|G(u^\exact) - G(u_\ell^\exact)|$.}
        \label{fig:arctanA}
    \end{subfigure}
    \hspace{-5mm}
    \begin{subfigure}[b]{0.45\textwidth}
        \includegraphics[width=\textwidth]{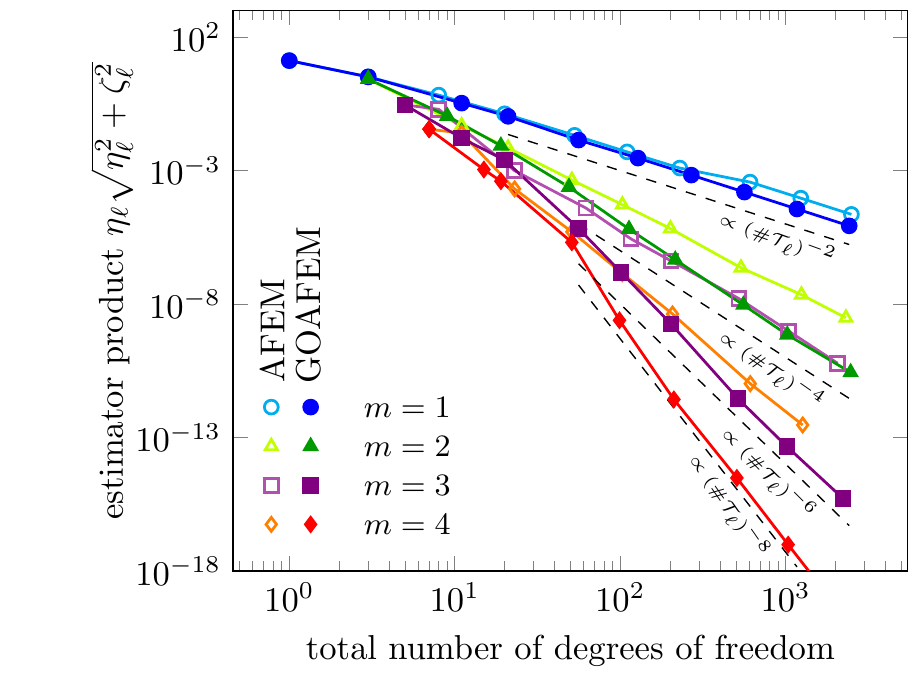}
        \caption{Estimator product of $\eta_\ell \sqrt{\eta_\ell^2 + \zeta_\ell^2}$.}
        \label{fig:arctanB}
    \end{subfigure}
    \caption{Goal error and the estimator product in Example~\ref{example:arctan} for ansatz spaces with $m \in \{1, 2, 3, 4\}$ for GOAFEM (Algorithm~\ref{algorithm:first}) and standard AFEM.}\label{fig:arctan}
\end{figure}

\end{example} 

\begin{example}\label{example:ms} For $\Omega = (0,1)^2$, we test Algorithm~\ref{algorithm:first} with a semilinear variant of~\cite[Example~7.3]{ms2009}: The weak formulation of the primal problem reads: Find $u^\exact \in H^1_0(\Omega)$ such that
\begin{align}
\sprod{u^\exact}{v} + \prod{b(u^\exact)}{v} = \int_{\Omega} \f \cdot \nabla v \d{x}, \quad \text{ for all } v \in H^1_0(\Omega),
\end{align}
where $b(v) = v^3$ and $\f =  \chi_{\Omega_{\f}}\, (-1, 0)$ with the characteristic function $\chi_{\Omega_{\f}}$ of $\Omega_{\f} = \set{x \in \Omega}{x_1 + x_2 \le \tfrac{1}{2}}$.
The weak formulation of the practical dual problem for $w \in H^1_0(\Omega)$ reads: Find $z^\exact[w] \in H^1_0(\Omega)$ such that
\begin{align*}
\sprod{z^\exact[w]}{v} + \prod{b'(w)z^\exact[w]}{v} = \int_{\Omega} \g \cdot \nabla v \d{x}, \quad \text{ for all } v \in H^1_0(\Omega),
\end{align*}
where $b'(v) = 3v^2$ and $\g = \chi_{\Omega_{\g}}\,(-1, 0)$ with $\Omega_{\g}= \set{x \in \Omega}{x_1 + x_2 \ge \tfrac{3}{2}}$. Our implementation adapts the P1-AFEM code from~\cite{fpw2011}.

In Figure~\ref{fig:2d}, for polynomial degree $m = 1,2$, we compare the goal value calculated with the proposed GOAFEM algorithm to the goal evaluation of the standard AFEM implementation and AFEM using $\eta^2_\ell + \zeta^2_\ell$ as a marking criterion (AFEM+). In all cases, we employ a (undamped) Newton iteration following~\cite[Section 3]{aw2015}. The reference goal value $G(u^\exact) = -0.001584951808832$ is obtained by extrapolation from the calculated goal values using GOAFEM with $m=2$. For $m=1$, an example of the meshes generated by GOAFEM (Algorithm~\ref{algorithm:first}) is shown in Figure~\ref{subfig:goafem}, by the standard AFEM algorithm in Figure~\ref{subfig:afem}, and AFEM+ in Figure~\ref{subfig:afemplus}. One clearly sees that, for GOAFEM and AFEM+, the singularities for both the primal and the dual problem are resolved, whereas for standard AFEM only those of the primal problem are taken into account. The meshes for $m=2$ look similar (not displayed). In particular, GOAFEM and AFEM+ lead to similar results, although in practice AFEM+ is slightly inferior from the point of theory (see Remark~\ref{rem:afemplus}).
\begin{figure}
    \hspace{-35mm}
     \begin{subfigure}[b]{0.3\textwidth}
    \centering
\includegraphics[scale =0.9]{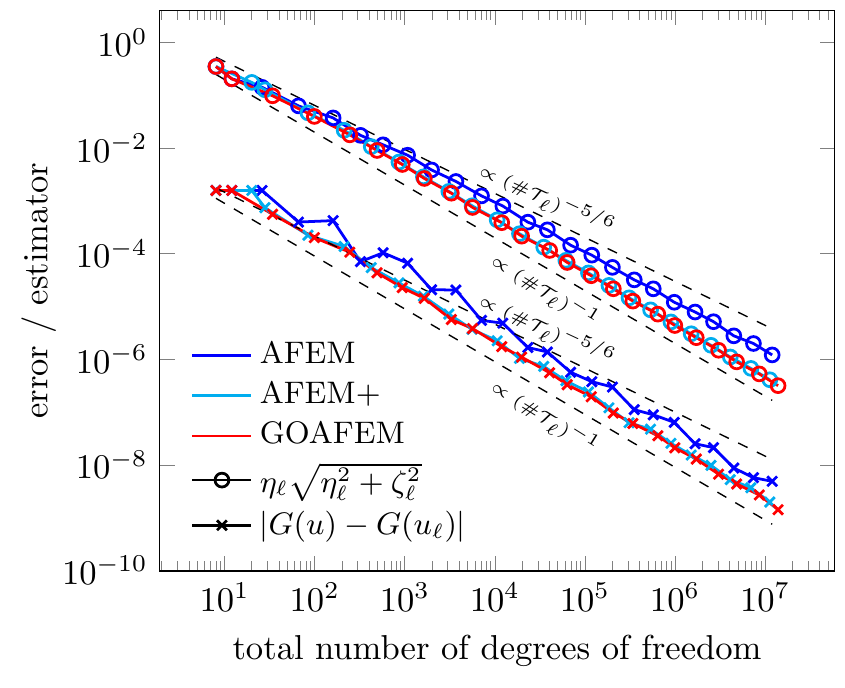}
        \label{subfig:p1}
    \end{subfigure}\qquad \hspace{20mm}
\begin{subfigure}[b]{0.3\textwidth}
    \centering
\includegraphics[scale =0.9]{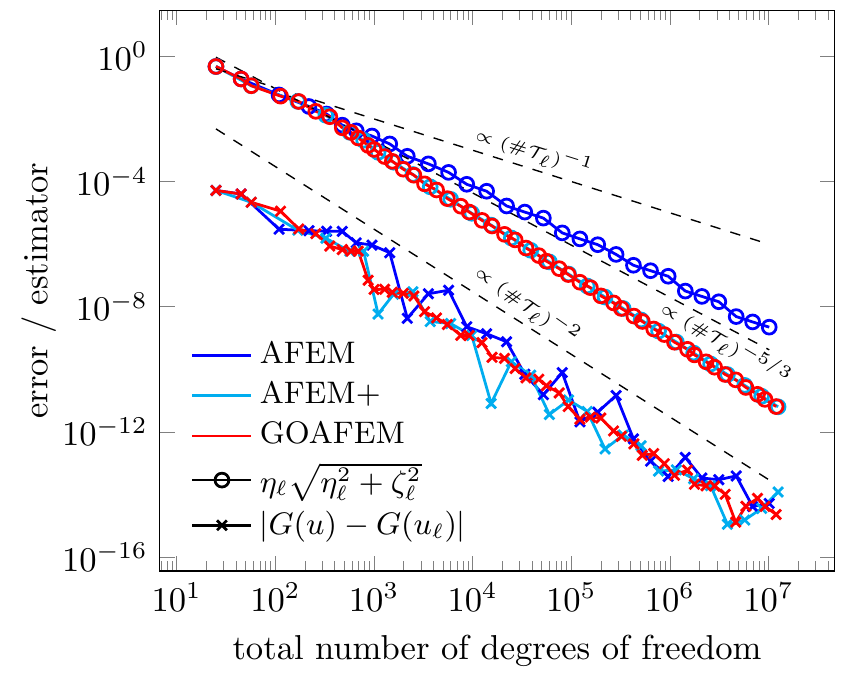}
        \label{subfig:p2}
    \end{subfigure}
    \vspace*{-5mm}
        \caption{Goal error $|G(u^\exact) - G(u_\ell^\exact)|$ (cross) and estimator product $\eta_\ell \sqrt{\eta_\ell^2 + \zeta_\ell^2}$ (circle) for $m=1$ (left) and $m=2$ (right) with adaptive refinement according to Algorithm~\ref{algorithm:first} (red) compared to standard AFEM (blue), and AFEM+ (cyan).}\label{fig:2d}
\end{figure}
\begin{figure}
    \centering
    \begin{subfigure}[b]{0.33\textwidth}
    \centering
\includegraphics[scale=1.5]{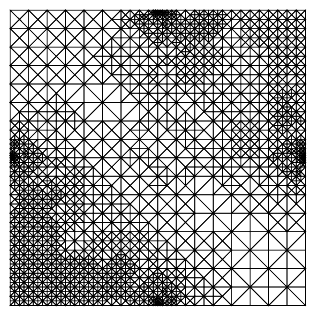}
      \caption{\centering GOAFEM mesh \newline with $\# \TT_\ell = 3284$.}
        \label{subfig:goafem}
    \end{subfigure}
     ~
    \begin{subfigure}[b]{0.33\textwidth}
    \centering
\includegraphics[scale=1.5]{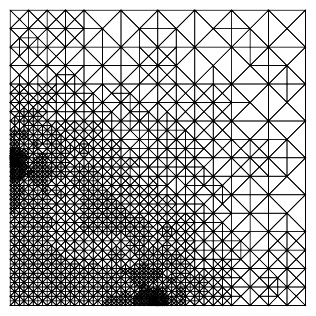}
       \caption{\centering Standard AFEM mesh \newline with $\# \TT_\ell = 3694$.}
        \label{subfig:afem}
    \end{subfigure}
    ~
  \begin{subfigure}[b]{0.33\textwidth}
    \centering
\includegraphics[scale=1.5]{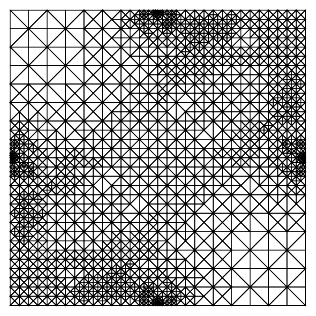}
       \caption{\centering AFEM+ mesh \newline with $\# \TT_\ell = 3080$.}
        \label{subfig:afemplus}
    \end{subfigure}
      \caption{Visualization of adaptive meshes for Example~\ref{example:ms} generated by Algorithm~\ref{algorithm:first} (left), standard AFEM (center), and AFEM+ (right) for $m=1$.}
    \label{fig:meshes}
\end{figure}
\end{example}